\documentclass[a4paper,11pt]{article}
\usepackage{amsmath,mathtools}
\usepackage{amsfonts}
\usepackage{amssymb}
\usepackage{mathrsfs}
\usepackage{hyperref}
\usepackage[retainorgcmds]{IEEEtrantools}
\usepackage{enumerate}
\usepackage{enumitem}
\usepackage{amsthm}
\usepackage{fullpage}
\usepackage{graphicx}
\usepackage{caption}
\usepackage{color}
\usepackage{float}
\usepackage{subfigure}
\usepackage[final]{microtype}
\usepackage{verbatim}
\usepackage{qtree}
\usepackage{tikz}
\usepackage{tikz-qtree}
\usetikzlibrary{arrows,automata,positioning}
\usepackage[style=numeric,backend=bibtex,maxbibnames=100]{biblatex} 
\theoremstyle{definition} \newtheorem{Definition}{Definition}[section]
\theoremstyle{plain} \newtheorem{Theorem}[Definition]{Theorem}
\theoremstyle{plain} \newtheorem{corollary}[Definition]{Corollary}
\theoremstyle{plain} \newtheorem{lemma}[Definition]{Lemma}
\theoremstyle{definition} \newtheorem{Remark}[Definition]{Remark}
\theoremstyle{plain}
\newtheorem{proposition}[Definition]{Propostion}
\theoremstyle{plain} \newtheorem{claim}[Definition]{Claim}
\theoremstyle{plain} 
\theoremstyle{definition}

\theoremstyle{plain} \newtheorem{Question}[Definition]{\bfseries{Question}}

\newcommand{\gen}[1]{\left\langle #1 \right\rangle}
\newcommand{\rfty}{\mbox{R}_{\infty}}

\newcommand{\im}{\mbox{ im}}

\newcommand{\id}{\mbox{id}}
\newcommand{\crit}{\mbox{crit}}
\newcommand{\aut}[1]{\mbox{Aut}({#1})}
\newcommand{\out}[1]{\mbox{Out}({#1})}
\newcommand{\inn}[1]{\mbox{Inn}(#1)}

\newcommand{\T}[1]{\mathcal{#1}}
\newcommand{\CCmr}[1]{\mathfrak{C}_{#1}}
\newcommand{\CCnr}{\mathfrak{C}_{n,r}}
\newcommand{\CCn}{\mathfrak{C}_{n}}

\newcommand{\Tn}{T_{n}}
\newcommand{\Tnr}{T_{n,r}}
\newcommand{\Gnr}{G_{n,r}}
\newcommand{\Rnr}{\mathcal{R}_{n,r}}
\newcommand{\Bnr}{\mathcal{B}_{n,r}}
\newcommand{\TBnr}{\T{T}\Bnr}
\newcommand{\TBmr}[2]{\T{T}\mathcal{B}_{#1,#2}}
\newcommand{\Tmr}[2]{T_{#1,#2}}
\newcommand{\Gmr}[2]{G_{#1,#2}}

\newcommand{\On}{\T{O}_{n}}
\newcommand{\Onr}{\T{O}_{n,r}}
\newcommand{\Ons}[1]{\T{O}_{n,#1}}
\newcommand{\Oms}[2]{\T{O}_{#1,#2}}
\newcommand{\XOn}{\T{X}_{n}}
\newcommand{\XOnr}{\T{X}_{n,r}}
\newcommand{\XOns}[1]{\T{X}_{n,#1}}
\newcommand{\XOms}[2]{\T{X}_{#1,#2}}
\newcommand{\TOn}{\T{T}\On}

\newcommand{\TOnr}{\T{T}\Onr}
\newcommand{\TOns}[1]{\T{T}\T{O}_{n,#1}}
\newcommand{\TOms}[2]{\T{T}\T{O}_{#1,#2}}
\newcommand{\WTOnr}{\widetilde{\T{TO}_{n,r}}}
\newcommand{\WTOns}[1]{\widetilde{\T{TO}_{n,#1}}}
\newcommand{\WTOn}{\widetilde{\TOn}}

\newcommand{\Xn}{X_{n}}
\newcommand{\Xns}{\Xn^{*}}
\newcommand{\Xnp}{\Xn^{+}}
\newcommand{\RXnp}{\mathcal{X}_{n}^{+}}

\newcommand{\Xnz}{X_n^{\Z}}

\newcommand{\Xnr}{X_{n,r}}

\newcommand{\Xnrs}{\Xnr^{*}}
\newcommand{\Xnrp}{\Xnr^{+}}

\newcommand{\Banr}{{\bf{B}}_{n,r}}
\newcommand{\Ban}{{\bf{B}}_{n}}
\newcommand{\Ln}[1]{\mathcal{L}_{#1}}

\newcommand{\hn}[1]{\mathcal{H}_{#1}}

\newcommand{\core}{\mathrm{Core}}

\newcommand{\ac}[1]{{\bf{\bar{#1}}}}
\newcommand{\dec}{\mathrm{Dec}}
\newcommand{\viable}[1]{\mathfrak{v}_{#1}}

\newcommand{\leqlex}{\le_{\mbox{lex}}}
\newcommand{\lelex}{<_{\mbox{lex}}}
\newcommand{\leqslex}{\le_{\mbox{slex}}}
\newcommand{\leslex}{<_{\mbox{slex}}}

\newcommand{\Z}{\mathbb{Z}}
\newcommand{\N}{\mathbb{N}}
\newcommand{\sym}{\mathmbox{Sym}}
\newcommand{\seteq}{:=}

\newcommand{\dotr}{{\bf{\dot{r}}}}
\newcommand{\simeqI}{\simeq_{{\bf{I}}}}
\newcommand{\rot}{\mathmbox{rot}}
\newcommand{\simrot}{\sim_{{\rot}}}
\newcommand{\rotclass}[1]{[#1]_[\simrot]}
\newcommand{\sig}{\mathrm{sig}}
\newcommand{\rsig}{\mathrm{\overline{sig}}}
\newcommand{\shift}[1]{\sigma_{#1}}
\newcommand{\rev}[1]{\overleftarrow{#1}}

\renewcommand{\restriction}{\mathord{\upharpoonright}}

\makeatletter
\renewcommand*{\eqref}[1]{%
  \hyperref[{#1}]{\textup{\tagform@{\ref*{#1}}}}%
}
\makeatother

\begin{document}
\author{
  		Olukoya, Feyishayo\\
  		Department of Mathematics,\\
  		University of Aberdeen, \\ 
  		Fraser Noble Building,\\ 
  		Aberdeen,\\
  		\texttt{feyisayo.olukoya@abdn.ac.uk}
  	}
\title{ Automorphisms of the generalised Thompson groups \texorpdfstring{$T_{n,r}$}{Lg} and the \texorpdfstring{$R_{\infty}$}{Lg} property}
\maketitle
\begin{abstract}
The recent paper \emph{The further chameleon groups of Richard Thompson
and Graham Higman: automorphisms via dynamics for
the Higman groups $\Gnr$} of Bleak, Cameron, Maissel, Navas and Olukoya (BCMNO) characterises the automorphisms of the Higman-Thompson groups $G_{n,r}$ as the specific subgroup of the rational group $\T{R}_{n,r}$ of Grigorchuk, Nekrashevych and Suchanski{\u i}'s consisting of those elements which have the additional property of being bi-synchronizing. In this article, we extend the arguments of BCMNO and characterise the automorphism group of $T_{n,r}$ as a subgroup of $\aut{G_{n,r}}$. We then show that the groups $\out{\Tnr}$ can be identified with subgroups of the group $\out{T_{n,n-1}}$. Extending results of Brin and Guzm{\'a}n, we show that  the groups $\out{\Tnr}$, for $n>2$, are all infinite and contain an isomorphic copy of Thompson's group $F$. For $X \in  \{T,G\}$,  we  study the groups  $\out{X_{n,r}}$ and show that these fit in a lattice structure where $\out{X_{n,1}} \unlhd \out{X_{n,r}}$ for all $1 \le r \le n-1$ and $\out{X_{n,r}} \unlhd \out{X_{n,n-1}}$. This gives a partial answer to a question in BCMNO concerning the normal subgroup structure of $\out{G_{n,n-1}}$. Furthermore, we deduce that for $1\le j,d \le n-1$ such that $d = \gcd(j, n-1)$, $\out{X_{n,j}}  = \out{X_{n,d}}$ extending a result of BCMNO for the groups $\Gnr$ to the groups $\Tnr$. We give a negative answer to the  question in BCMNO which asks whether or not $\out{G_{n,r}} \cong \out{G_{n,s}}$ if and only if $\gcd(n-1,r) = \gcd(n-1,s)$. We conclude by showing that the groups $\Tnr$ have the $R_{\infty}$ property extending the result of Burillo, Matucci and Ventura and, independently, Gon{\c c}alves and Sankaran,  for Thompson's group $T$.

\end{abstract}
\section{Introduction}
The recent paper \cite{BCMNO} characterises the automorphisms of the Higman-Thompson groups $G_{n,r}$ as a subgroup of the rational group $\T{R}_{n,r}$ consisting of those elements which have the additional property of being bi-synchronizing. In this article, we extend the arguments of \cite{BCMNO} and characterise the automorphism group of $T_{n,r}$, generalisations of Thompson's group $T$, as a subgroup of $\aut{G}_{n,r}$.

 The automorphism group  of $T_{2,1}$, or $T$, as well as the automorphism group of Thompson's group $F$, were studied in the paper \cite{MBrin2}, which also demonstrates that $\out{T_2}$ is isomorphic to the cyclic group of order $2$. The paper \cite{JBurilloSClearyF} studies the metric properties of $\aut{F}$ and gives a presentation for this group, whilst  the follow up paper \cite{MBrinFGuzman} studies the automorphisms of the groups $F_{n}$ and $T_{n,n-1}$, generalisations of Thompson's group $F$ and $T$ (the approach there gives no information about $\aut{T_{n,r}}$ when $r \ne n-1$). The  paper \cite{MBrinFGuzman}, amongst other things, demonstrates that $\out{T_{n,n-1}}$  for $n\ge3$ is infinite and contains an isomorphic copy of Thompson's group $F$. 

In this paper we extend the results of \cite{MBrinFGuzman} to the groups $\Tnr$ for $1 \le r < n-1$. Firstly, we prove the following,

\begin{Theorem}
The group $\aut{T_{n,r}}$ consists of those elements of $ \Rnr$ which can be represented by bi-synchronizing transducers such that the induced homeomorphisms on Cantor space respects the cyclic ordering. 
\end{Theorem}

One immediately deduces the following corollary:
\begin{corollary}
$\aut{\Tnr} < \aut{G_{n,r}}$.
\end{corollary}

We now turn our attention to the quotient group $\out{\Tnr} = \aut{\Tnr}/ \inn{\Tnr}$. The results below are equally true in the $\Gnr$ context and so to avoid repetition we shall sometimes use the symbol $X$ to represent either $T$ or $G$, analogously the symbol $\T{X}$ will represent either $\T{TO}$ or $\T{O}$.
 
 It is a result in the paper \cite{BCMNO} that for $1 \le r < n$ the group $\out{\Gnr}$ is isomorphic to a subgroup $\Onr$ of a group $\On$ consisting of non-initial, bi-synchronizing transducers. Furthermore, that paper also shows that $\On = \cup_{1\le r \le n-1} \Onr$. More specifically, it is demonstrated in \cite{BCMNO} that for $1 \le i \le j \le n-1$ such that $i$ divides $j$ in the additive group $\Z_{n-1}$,   $\Ons{i} \le \Ons{j}$. As a consequence of this, one may deduce that for all $1 \le i \le n-1$ $\Ons{1}\le \Ons{i}$, $\Ons{i} \le \Ons{n-1}$  (and so $\On = \Ons{n-1}$) and $\Ons{i} = \Ons{d}$ for $d = \gcd(n-1, i)$.  We extend these result to the group $\out{\Tnr}$. That is, we have:

\begin{Theorem}
For all $1 \le r \le n-1$, the group $\out{\Tnr}$ is isomorphic to a subgroup $\TOnr$ of $\Onr$. Moreover for all $1 \le i \le j \le n-1$ such that $i$ divides $j$ in the additive group $\Z_{n-1}$, we have   $\TOns{i} \le \TOns{j}$. 
\end{Theorem} 

\begin{corollary}\label{Corollary:OniequalsOnjifgcdequal}
For $1 \le 1 \le n-1$, we have $\TOns{1} \le \TOns{i}$, $\TOns{i} \le \TOns{n-1}$  and $\TOns{i} = \TOns{d}$ for $d = \gcd(n-1, i)$.
\end{corollary}

We shall subsequently  use the symbol $\TOn$ for the group $\TOns{n-1}$.

The last phrase of Corollary~\ref{Corollary:OniequalsOnjifgcdequal} is perhaps to be expected, certainly when $X = G$, as  results of Higman (\cite{GHigman}), Pardo (\cite{EPardo}), and Dicks and Mart{\' \i}nez-P{\' e}rez (\cite{DicksPerez})  demonstrates that $\Gnr \cong \Gmr{m}{s}$ if and only if $n=m$ and $\gcd(n-1, r) = \gcd(n-1,s)$. In fact it is a question in \cite{BCMNO} whether or not $\Onr \cong \Ons{s}$ if and only if $\gcd(n-1,r) = \gcd(n-1,s)$. We show that this question has a negative answer in both the $\Gnr$ and $\Tnr$ context. That is we prove the following:

\begin{Theorem}
There is a number $n \in \N$, $n >2$, and $1 \le r, s \le n-1$ such that, for $X = \T{TO}, \T{O}$, $\XOnr = \XOns{s}$ but $\gcd(n-1, r) \ne \gcd(n-1, s)$.
\end{Theorem}

Our next result demonstrates that the groups $\XOns{r}$ for all $1 \le r \le n-1$ are normal subgroups of $\XOn$.

\begin{Theorem}
For all $1 \le r \le n-1$ we have $\XOns{r} \unlhd \XOn$.
\end{Theorem}

We next extend a result of Brin and Guzm{\'a}n \cite{MBrinFGuzman} for $\TOns{n-1}$, and show that for $n \ge 3$, $\XOns{1}$ contains an isomorphic copy of R. Thompson's group $F$:

\begin{Theorem}
Let $n \ge 3$, then $\out{X_{n,1}}$ contains a subgroup isomorphic to R. Thompson's group $F$. 
\end{Theorem}

 From this and Corollary~\ref{Corollary:OniequalsOnjifgcdequal}, we have the following:
 \begin{corollary}
 For all $1 \le r\le n-1$, $\XOns{r}$ contains  a subgroup isomorphic to R. Thompson's group $F$.
 \end{corollary}

  We further demonstrate (Section~\ref{Section:nestingproperties2}), in the case $r=4$, that the set  $\TOms{4}{3}\backslash \TOms{4}{1}$ is non-empty. Notice that since $3$ is prime, for $ 1 \le r <3$, $\TOms{4}{r} = \TOms{4}{1}$, thus this result indicates that, in general, the group $\TOnr$ might depend on $r$.

 We also investigate in Section~\ref{Section:nestingproperties2} the nesting properties of the groups $\XOnr$ of $\XOn$ for $1 \le r \le n-1$. We show that these groups from a lattice with the `meet' of $\XOnr$ and $\XOns{s}$ being the intersection of the two groups and the `join' of $\XOnr$ and $\XOns{s}$ being the smallest $t$, $1 \le t \le n-1$ such  that $\XOnr$ and $\XOns{s}$ are subgroups of $\XOns{t}$. We do not know if it is in fact the case that $\XOns{t} = \gen{\XOns{s}, \XOns{r}}$ (see Question~\ref{Question:canjoinbereplacedwithubgroupgenerated}). To each element of $\XOn$ we associate a numerical invariant which yields a group homomorphism from $\XOn$ to the group of units of $\Z_{n-1}$ (notice that for $n=2$,  $\Z_{n-1}$ is equal to its group of units), with kernel $\XOns{1}$. That is, we prove the following:
 
 \begin{Theorem}\label{Theorem:homfromoutintounitsofZn}
 There is  homomorphism from $\XOn$ to the group of units of $\Z_{n-1}$ with kernel $\XOns{1}$. 
 \end{Theorem}
 
We should point out that the existence of this homomorphism is already a consequence of the fact that the  dimension group (see \cite{KriegerW1980} for the definition of the dimension group) of $\XOnr$ is equal to the additive group $\Z_{n-1}$ and automorphisms of $\XOnr$ yield automorphisms of the dimension group of $\XOnr$. The author is grateful to  Prof. Nekrashevych for drawing these facts to his attention. Our proof of Theorem~\ref{Theorem:homfromoutintounitsofZn} however does not rely on these observations, instead we explicitly construct the homomorphism by making use of the synchronizing condition to associate a numerical invariant to every element of $\XOnr$. Our approach gives a means of resolving this question as we  reduce it to one of constructing transducers which have certain properties. In particular, Theorem~\ref{Theorem:homfromoutintounitsofZn} demonstrates that the homomorphism from $\XOns{1}$ into the group of units of $\Z_{n-1}$ is the trivial homomorphism. Though we are unable show in general that the homomorphism of Theorem~\ref{Theorem:homfromoutintounitsofZn} is surjective, we show that under the assumption that it is surjective, then  for $1 \le s,r,t \le n-1$ and $t$ the smallest element of $\Z_{n}$ such that $\XOms{n}{t}$ contains $\XOms{n}{s}$ and $\XOms{n}{r}$, $\XOms{n}{t}= \gen{\XOms{n}{s}, \XOms{n}{r}}$. The Theorem~\ref{Theorem:surjectivityofrsigpartial} below, proven in Section~\ref{Section:onsurjectivityofrsig}, indicates that the homomorphism of Theorem~\ref{Theorem:homfromoutintounitsofZn} from $\On$ to the group of units of $\Z_{n-1}$ is surjective in many cases, indeed, elementary results in number theory indicate that there are infinitely many numbers $n$ which satisfying the hypothesis of the theorem. We do not know if the restriction to  $\TOn$ is also surjective.

\begin{Theorem}\label{Theorem:surjectivityofrsigpartial}
If the divisiors of $n$ generate the group of units of $\Z_{n-1}$ then the homomorphism of Theorem~\ref{Theorem:rsigisahomomorphism} defined on $\On$ is unto the group of units of $\Z_{n-1}$.
\end{Theorem}

We conclude in Section~\ref{Section:RftyTn} with the following result:
\begin{Theorem}
The groups $\Tnr$ have the $\rfty$ property.
\end{Theorem}
We recall that a group $G$ is said to have the $R_{\infty}$ property if for every automorphism $\varphi$ of $G$, the equivalence relation defined on $G$ by, for $x,y \in G$, $x$ is equivalent to $y$ if there is an element $h$ in $G$ such that $h^{-1}x(h)\varphi = y$,  has infinitely many equivalence classes. The question of which groups have the $\rfty$ property has received a lot of attention. This class of groups has been shown to  include, for instance, all non-elementary Gromov hyperbolic groups (\cite{Fel'shtyn2004}, \cite{GLevittMLustig}), Baumslag-Solitar groups $BS(m,n)$ for $m,n \in \Z\backslash\{0\}, (m,n) \ne (1,1)$ (\cite{AFelshtynDGonclavesBSG}), lamplighter groups $\Z_{n} \wr \Z$ where $2|n$ or $3|n$ (\cite{DGonclavesPWong}), the first Grigorchuk  and the Gupta-Sidki groups (\cite{AFelshtynLYuriyandETroitsky}), and R. Thompson's group $F$ (\cite{BlkFelAlxGon}, \cite{BMV}). It was shown independently by Burillo, Matucci, and Ventura (\cite{BMV}), and  Gon{\c c}alves and Sankaran (\cite{DGonclavesPSankaran}) that $R$. Thompson's group $T$ is in this class of groups. We thus extend this result to the family of groups $T_{n,r}$.

Interlaced through out the text are several open questions. In work in preparation we apply the techniques in the current paper and in the paper \cite{BCMNO} to the generalizations of Thompson's group $F$. 
\section*{Acknowledgements}
The author is grateful to Collin Bleak and Matthew Brin for helpful discussions and their comments on early drafts of this article. This work was partially supported by Leverhulme Trust Research Project Grant RPG-2017-159.
\section{Some preliminaries and the groups \texorpdfstring{$G_{n,r} \mbox{ and } T_{n,r}$}{Lg}}\label{Section:Preliminaries}
We begin by setting up some notation.

Let $X$ be a topological space, we shall denote by $H(X)$ the group of homeomorphisms of $X$, and for $G \le H(X)$, $N_{H(X)}(G)$ shall denote the normaliser of $G$ in $H(X)$. Let $S_r:= \mathbb{R}/r \Z$ for $r \in \mathbb{R}$, the circle of  length $r$ and for $n \in \N \backslash \{0\}$, let $\Z[1/n] := \{ a/n : a \in \Z\}$, the  $n$-adic rationals.  Though we will mainly think of $T_{n,r}$ as a subgroup of $G_{n,r}$, we shall at times consider it as a subgroup of $H(S_r)$ and we shall make it apparent at such times that we are doing so. We establish some further notation required to define the groups $G_{n,r}$ and $\Tnr$.

Let $\dotr:= \{\dot{0},\dot{1},\ldots, \dot{r-1}\}$, and let $X_n := \{0,1,\ldots, n-1\}$. We shall take the ordering $\dot{0} < \dot{1} < \ldots < \dot{r-1}$ and $0< 1 < \ldots < n-1$ on $\dotr$ and $\Xn$ respectively. For $i \in \{0, n-1\}$, if $i= 0$ then set $\bar{i} = n-1$ otherwise set $\bar{i} = 0$. Set $X_{n,r}^{\ast} := \{\dot{a}w: \dot{a} \in \dotr \mbox{ and } w \in X_n^{\ast}\} \sqcup \{\epsilon\}$ where $\epsilon$ denotes the empty word, and $X_{n,r}^{+}:= \{\dot{a}w: \dot{a} \in \dotr \mbox{ and } w \in X_n^{\ast}\}$. For $j \in \N$, let $\Xnr^{j}$, respectively $\Xn^{j}$, denote the set of all elements of $X_{n,r}^{*}$, respectively $X_n^{*}$, of length $j$. Given two words $u$ and $v$ in $X_n^{*}$ of $X_{n,r}^{*}$, we say that  $u$ is a prefix of $v$ and denote this by $u \le v$. We write $u < v$ if $u$ is a proper prefix of $v$. For two words $u, v$ in $X_{n}^{*}$ or $X_{n,r}^{*}$ if $u \nleq v$ and $v \nleq u$ then we say that $u$ is incomparable to $v$ an write $u \perp v$. If instead $ v = u \nu$ for some $\nu \in \Xns \sqcup \Xnrs$ then we set $v-u := \nu$. The relation $\le$ is a partial order on  the set $X_{n,r}^{*}$ and $X_n^{*}$.

For two incomparable words $\nu, \eta$ in $\Xnp$ and $\Xnrp$ we say  \emph{$\nu$ is less than $\eta$ in the lexicographic ordering}, denoted $\nu \lelex \eta$, if there are $wi, wj \in \Xnrp$ or $wi, wj \in  \Xnp$ prefixes of $\nu$ and $\eta$ respectively with $i < j$ in the ordering on $\Xn$ or $\dotr$.  If $i \le j$ then we write $\nu  \leqlex \eta$.

 We may now defined a total order of $\Xnrs$ and $\Xns$ as follows: for $\nu, \eta \in \Xnrs$ or $\nu, \eta \in \Xns$ we say that \emph{$\nu$ is less than $\eta$ in the short-lex ordering} if either $|\nu| \le |\eta|$ or if $|\nu| = |\eta|$ then  $\nu \leqlex \eta$, with strict inequalities if $\nu$ is strictly less than $\eta$. We write $\nu \leqslex \eta$ if $\nu$ is less than $\eta$ in the short-lex ordering, and $\nu \leslex \eta$ is the $\nu$ is strictly less than $\eta$ in the short-lex ordering. 

Now let $\CCn := X_n^{\omega}$ a Cantor space with the usual topology, and let $\CCnr:= \{\dot{a}x | \dot{a} \in \dotr \mbox{ and } x \in \CCn \}$, the disjoint union of $r$ copies of $\CCn$. For a word $\nu \in  X_{n,r}^{+}$, let  $U_{\nu}:= \{\nu\delta \mid \delta \in \CCn\}$, and $U_{\epsilon}:=  \CCnr$. For $\nu \in X_n^{\ast}$, set $U_{\nu}:=   \{ \nu \delta: \delta \in \CCn\}$. Let $\Banr:= \{U_{\nu}\mid \nu \in X_{n,r}^{\ast} \}$ and let $\Ban:= \{U_{\nu} \mid \nu \in \Xn \}$, then $\Banr$ and $\Ban$ are a basis for the topology on $\CCnr$ and $\CCnr$ respectively. For a subset  $Z \subseteq \Xns \sqcup \Xnrs$ we shall denote by $U(Z)$ the set $\{ U_{z} \mid z \in Z \}$.

The lexicographic  extends to a total order on $\CCn$ and $\CCnr$ in the natural way, thus we extend the meanings of the symbols $\leqlex$ and $\lelex$ to this sets also. We further extend this notation to subsets of $\CCn$ and $\CCnr$. Given two subsets $V, W \subset \CCn$ or $V, W \subset \CCnr$ we write $V \leqlex W$ if every element of $U$ is less than every element of $W$ in the lexicographic ordering. The strict inequality $V \lelex W$ is analogously defined.

Let $\nu, \eta \in X_{n,r}^{+}$, then we shall denote by $g_{\nu,\eta}$ the map from $U_{\nu}$ to $U_{\eta}$ that replaces the prefix $\nu$ with $\eta$.

A finite set $\overline{u}:= \{u_0, \ldots, u_l\}$ for $u_i \in X_{n,r}^{+}$ and some $1\le l \in \N$ is called a \emph{complete antichain}, if for any pair $u_i, u_j \in \overline{u}$, $u_i \perp u_j$ and for any word $v \in X_{n,r}^{\ast}$, there is some $u_i \in \overline{u}$ such that $v \le u_i$ or $u_i \le v$. We shall assume through out that any given antichain $\overline{u}$ is ordered according to the lexicographic ordering.  If we have an antichain $\ac{u}:= \{u_0, \ldots, u_l\}$ then we may form another antichain $\ac{u}' = \{u_0, \ldots, u_i0, u_i1, \ldots u_in-1, u_{i+1}, \ldots, u_l\}$ where we have replaced $u_i$ with  $u_i0, \ldots, u_in-1$. We call $\ac{u}'$ a subdivision of $\ac{u}$. Notice that each time we make a subdivision the length of the antichain increases by $n-1$.

We now define the Higman-Thompson group $G_{n,r}$ and the subgroup $T_{n,r}$.

Given two complete antichains $\overline{u} = \{u_0,\ldots,u_{l-1} \}$ and $\overline{v} = \{v_0, \ldots, v_{l-1}\}$ of equal length,  we can  define a homeomorphism from of $\CCnr$ as follows. Let $\pi$ be a bijection from $\overline{u}$ to $\overline{v}$, define a map $g : \CCnr \to \CCnr$ by $ u_ix \mapsto (u_i)\pi x$ for $u_i \in \overline{u}$ and $x \in X_n^{\omega}$. Since $u$ is a complete antichain this map is well-defined on all of $\CCnr$. The group $G_{n,r}$ consists of all homeomorphisms of $\CCnr$ which can be defined in this way. The group $T_{n,r}$ is the subgroup of $G_{n,r}$, consisting of those element  $g$ that are derived from complete antichains $\overline{u}$ and $\ac{v}$ of the same length $l$, and a bijection $\pi$ which maps $u_i$ to $v_{i+j \mod l}$ for some fixed $j \in \{0,1 \ldots, l-1\}$. 

Define an equivalence relation on $\CCnr$ by $x \sim_{t} y$ if and only if there are words $u$ and $v$ in $X_{n,r}^{+}$ such that $x = u z$ and $y = v z$ for some $z \in \CCn$. We write $[x]$ to denote the equivalence class of $x$. Let $H_{n,r,\sim_{t}}$ be the subgroup of $H(\CCnr)$ consisting of those elements of $H(\CCnr)$ that preserve $\sim_{t}$.

We shall need the following transitivity result for $T_{n,r}$

\begin{lemma}\label{Lemma: Tnr preserves tail classes and acts transitively on each tail class}
The group $T_{n,r}$ fixes the equivalence classes of $\sim_{t}$ and acts transitively on each equivalence class.
\end{lemma}
\begin{proof}
The first statement follows from the definition of $T_{n,r}$ as a subgroup of $G_{n,r}$. The second part of the lemma follows by  the following observation:  if $x = \mu z$ and $y = \nu z$ for some $\mu, \nu \in X_{n,r}^{+}$ and $z \in \CCn$, then there is an element of $T_{n,r}$ mapping  $x$ to $y$. To see this observe that the lengths of the complete antichains $\ac{u}:= X_{n,r}^{|\mu|}$ and $\ac{v} := X_{n,r}^{|\nu|}$ are congruent modulo $n-1$. We may assume that the length of  $\ac{u}$ is smaller than the length of $\ac{v}$. Therefore there is a complete antichain $\ac{u}'$ such that $\mu$ is an element of  $\ac{u}'$, $\ac{u}'$ is a repeated subdivision of $\ac{u}$ and the length of $\ac{u}'$ is equal to the length of $\ac{v}$. It is now  straightforward to see that, since $|\ac{v}_1'| = |\ac{v}_2$, there is an element of $T_{n,r}$ that sends any element $\mu z'$ to the element $\nu z'$ for $z' \in \CCn$. In particular such a map sends $x$ to $y$.
\end{proof} 

In the next Section, we see that elements of $T_{n,r}$ induce homeomorphisms of the circle $\mathbb{R}/r\Z$ by thinking of the circle as a quotient of Cantor space.

\section{From Cantor Space to the Circle}\label{Section:fromcantorspacetothecircle}

 \begin{Definition}
 We say that a homeomorphisms $h$ of $\CCnr$ is \emph{orientation preserving} if  whenever $x,y \in \CCnr$ and $x \lelex y$ then $(x)h \lelex (y)h$. We say that $h$ is \emph{orientation reversing} if whenever $x\lelex y$ then $(y)h \lelex (x)h$. We say that $h$ is \emph{locally orientation preserving/  reversing} if there is a neighbourhood of $\CCnr$ on which $h$ is orientation preserving/reversing.
 \end{Definition}
 
We identify $S_r$ with the interval $[0, r]$ with the end points identified. Now, every point in $[0,r]$ can be written as $\dot{a}x$ for $x \in \CCn$ and $\dot{a} \in \dotr$ in $n$-ary expansion. However this representation is not unique. In particular two elements $x, y \in \CCnr$ represents the same element of $[0,r]$ if and only if there is some integer $i$, $0< i \le n-1$ or $\dot{0}< i \le \dot{r-1}$, and some $w \in \Xnr^{*}$ such that $x = wi00\ldots$ and $y= w(i-1)n-1n-1\ldots$ (i.e only elements of $(0,r)\cap \Z[1/n]$ have non-unique $n$-ary representations ).  Let $\simeq$ be the equivalence relation on $\CCnr$ defined by $x \simeq y$ if and only if there is some $0< i \le n-1$ or $\dot{0} < i \le \dot{r-1}$ and some $w \in \Xnr^{+}$ such that $x = wi00\ldots$ and $y= w(i-1)n-1n-1\ldots$ or $x= \dot{0}00\ldots$ and $y = \dot{r}n-1 n-1 \ldots$ then $\CCnr/\simeq$ is homeomorphic to $S_{r}$. Let $\simeq_{{\bf{I}}}$ be the relation on $\CCn$ given by  $x \simeq_{{\bf{I}}} y$ if and only if there is some $0< i \le n-1$ or $\dot{0} < i \le \dot{r-1}$, and some $w \in \Xns$ such that $x = wi00\ldots$ and $y= w(i-1)n-1 n-1\ldots$  then $\CCn/\simeq_{{\bf{I}}}$ is homeomorphic to the interval $[0,r]$.



Let $N$ be the subgroup of $H(\CCnr)$ consisting of those elements $h$ which preserve $\simeq$ and which satisfy the following:  for all points $t \in \CCnr$ there is a neighbourhood of $t$ in $\CCnr$ such that $h$ is orientation preserving ($h$ is orientation reversing), and there is some $w \in \Xnrs$, $0< i \le n-1$ or $\dot{0} < i \le \dot{r-1}$, and points $x= wi00\ldots$ and $y= wi-1n-1n-1\ldots$ so that $(y)h = \dot{r}n-1n-1\ldots$  and $(x)h = \dot{0}00\ldots$ ($(y)h =  \dot{0}00\ldots$ and $(x)h = \dot{r}n-1n-1\ldots$). Observe that all elements of $N$ induce homeomorphisms of $S_{r}$, and $N$ contains $T_{n,r}$. Our aim shall be to show that $N_{H(\CCnr)}(\Tnr) \le N$.

We shall need the following results from \cite{MBrinFGuzman}. The first follows by Rubin's Theorem and the transitivity of $T_{n,r}$ on the circle $S_{r}$ and the second follows by studying the germs of elements of $T_{n,r}$ at a fixed point.

\begin{Theorem}
Let $r < n \in \mathbb{\N}$  and let $n \ge 2$, then $\aut{T_{n,r}} \cong N_{H(S_{r})}(T_{n,r})$.
\end{Theorem}

\begin{lemma}\label{Lemma: normalisers map n-adic rationals to n-adic rationals}
If $h \in N_{H(S_{r})}(\Tnr)$, then $ ( [0,r] \cap \Z[1/n])h = [0,r] \cap \Z[1/n]$. 
\end{lemma}

Now let $g \in N_{H(S_r)}(\Tnr)$. Define $h \in H(\CCnr)$ as follows. For all $t \in S_{r}$ such that $t \notin [0,r] \cap \Z[1/n]$ (notice that this also means $(t)g  \notin [0,r] \cap \Z[1/n]$ by Lemma~\ref{Lemma: normalisers map n-adic rationals to n-adic rationals}),   let $x$ be the unique $n$-ary expansion of $t$, and $y$ be the unique $n$-ary expansion of $(t)g$, set $(x)h = y$. For $t \in [0,r] \cap \Z[1/n]$ let $x$ and $x'$ be the $n$-ary expansions of $t$ such that $x \lelex x'$ in the lexicographic ordering of $\CCnr$ (if  $t =0$, then chose  $x$ and $x'$, satisfying $x \lelex x'$, from the set $\{\dot{0}00\ldots,  \dot{r-1}n-1n-1\ldots\}$). Also let $y$ and $y'$ be the $n$-ary expansions of $(t)g$ in $\CCnr$ such that $y \lelex y'$ (if $(t)g = 0$ then take  $y= \dot{0}00\ldots$ and $y' = \dot{r}n-1n-1\ldots$). If $g$ is an orientation preserving homeomorphism of $S_{r}$,  set $(x)h = y$ and $(x')h = y'$, otherwise $g$ is orientation reversing and we set $(x)h = y'$ and $(x')h = y$. Thus $h$ is now defines a bijection from $\CCnr$ to itself. Moreover it is easy to see that since $g$ is continuous, $h$ is also continuous on $\CCnr$. Furthermore if $h'$ is the homeomorphism obtained from $g^{-1}$ in the same way, it is easy to see that $hh' = h'h = \id \in H(\CCnr)$. Therefore $h  \in N$. Let $\phi: N_{S_r}(\Tnr) \to N$ such that an element $g \in N_{S_r}(\Tnr)$ maps to the element  $h \in H(\CCnr)$, constructed as above, which induces the map $g$ on $S_r$. It follows that $\phi$ is an injective homomorphism. 

Let $N(\Tnr)$ denote the image of $\phi$. We now show that $N_{H(\CCnr)}(\Tnr) = N(\Tnr)$. Observe that as $N(\Tnr) \subseteq N_{H(\CCnr)}(\Tnr)$ it suffices to show only that $N(\Tnr) \supseteq N_{H(\CCnr)}(\Tnr)$. 

Let $\tau, \eta \in X_{n,r}^{+}$ such that $\tau \perp \eta$. We further assume that the points $x=\tau00\ldots$, $y=\eta n-1 n-1\ldots $, $z= \tau n-1 n-1 \ldots$ and $t = \eta 00\ldots$ of $\CCnr$ satisfy $x \not\simeq y$ or $z \not\simeq t$. 
Let $\mathscr{G}$ denote the set of incomparable pairs $(\tau, \eta)$  satisfying the conditions above where $\tau  \lelex \eta$. For a pair $(\tau, \eta)$ in $\mathscr{G}$ let $\ac{u} := \{u_1, \ldots, u_l\}$ be any complete finite antichain of $\Xnrs$ containing $\tau$ and $\eta$ (since $\tau \perp \eta$ such an antichain exists). Let $u_{l_1} = \tau$ and $u_{l_2} = \eta$. Let  $j = l_1-l_2 -1$. Then we say that $j$ is the \emph{node distance between $\tau$ and $\eta$ in  $\ac{u}$} i.e $j$ is the number of elements of $\ac{u}$ which are strictly between  $u_{l_1}$ and $u_{l_2}$. Notice that $j$ is necessarily non-zero by assumption.  Moreover we have, for any $v \in \{0\}^{+}$ and $w \in \{n-1\}^{+}$, that the pair $(\dot{0}v, \dot{r-1}w)$  is not in $\mathscr{G}$.

It is straight-forward to see that if $(\tau, \eta) \in \mathscr{G}$ have node distance $i$ in some complete antichain $\ac{u}_{1}$, then for any other complete antichain $\ac{u}_2$ containing $\tau$ and $\eta$ the node distance between $\tau$ and $\eta $ in $\ac{u}_{2}$ is congruent to $i$ modulo $n-1$. Moreover, for any $\chi \in X_n^{*}$, $(\tau \chi, \eta \chi) \in \mathscr{G}$ and there is a complete antichain $\ac{v}$ containing $\tau\chi$ and $\eta\chi$ such that the node distance between $\tau\chi$ and $\eta \chi$ in $\ac{v}$ is congruent to $i$ modulo $n-1$.  Since modulo $n-1$ the node distance between a pair $(\tau, \eta)$ in $\mathscr{G}$ in a given complete antichain containing $\tau$ and $\eta$ is independent of the complete antichain,  we define the \emph{reduced node distance} between $\tau$ and $\eta$ to be $i \in \{0,1, \ldots, n-2\}$ such that for any complete antichain $\ac{u}$ containing $\tau$ and $\eta$ the node distance between $\tau$ and $\eta$ in $\ac{u}$ is congruent to $i$ modulo $n-1$.

We require the following transitivity result of $\Tnr$.

\begin{lemma}\label{Lemma: can swap things with same node distance}
Let $(\nu_1, \nu_2), (\eta_1, \eta_2) \in \mathscr{G}$ be such that the reduced node distance between $\nu_1$ and $\nu_2$ is equal to the reduced node distance between $\eta_1$ and $\eta_2$. Then there is a $g \in \Tnr$ such that $g\restriction_{U_{\nu_1}}=  g_{\nu_1, \eta_1}$ and $g\restriction_{U_{\nu_2}} = g_{\nu_2, \eta_2}$. 
\end{lemma}
\begin{proof}
Let $\ac{u}$ be a complete antichain containing $\nu_1$ and $\nu_2$, and $\ac{v}$ be a complete antichain containing $\eta_1$ and $\eta_2$. Let $\ac{u}= \{u_1, \ldots, u_{l}\}$ and $\ac{v} = \{v_1, \ldots, v_{m}\}$. Let $u_{l_1} = \nu_1$ and $u_{l_2} = \nu_2$ and $v_{m_1} = \eta_1$ and $v_{m_2} = \eta_2$ where $l_1 < l_2$ and $m_1 < m_2$. By assumption $i:= l_2 - l_1 -1 $ and $j:=m_2 - m_1 -1$ are both non-zero and congruent modulo $n-1$. Without loss of generality we may assume $i< j$, moreover there is an element $u_{l_1+1} \in \ac{u}$ between $\nu_1$ or $\nu_2$. By replacing $u_{l_1 + 1} \in \ac{u}$ with $\{u_{l_1 +1}0, \ldots, u_{l_1 +1} n-1 \}$ we obtain a new complete antichain $\ac{u}'$ containing $\nu_1$ and $\nu_2$ such that the node distance between $\nu_1$ and $\nu_2$ is equal to $i + n-1$. Replace the antichain $\ac{u}$ with $\ac{u}'$

By repeatedly performing this operation we may assume that the complete antichain $\ac{u}$ containing $\nu_1$ and $\nu_2$ is such that the node distance between $\nu_1$ and $\nu_2$ in $\ac{u}$ is equal to $j$.

Observe that as $\ac{u}$ and  $\ac{v}$ are complete antichains the difference $| m-l|$ is congruent to $0$ modulo $n-1$. Without loss of generality (as we may relabel to achieve this) we assume that $l < m$.

Now  observe that $u_1$ must be equal to $\dot{0}w_1$ for some $w_1 \in \{0\}^{*}$ and  $u_{l} = \dot{r-1}w_2$ for some $w_2 \in \{n-1\}^{*}$. Moreover, since the pair  $(\dot{0}v, \dot{r-1}w)$ for $v \in \{0\}^{+}$ and $w \in \{n-1\}^{+}$ is not in $\mathscr{G}$, it follows that either $\nu_{l_1} \ne u_{l} = \dot{0}w_1$ or $\nu_{l_2} \ne \dot{r-1}w_2$. Suppose that $\nu_{l_1} \ne u_{l} = \dot{1}w_1$ (the other case is handled similarly). Then by repeatedly expanding along the node $\dot{1}w_1$, we obtain a complete antichain $\ac{u}''$ of length $m$ containing $\nu_1$ and $\nu_2$ such that the node distance between $\nu_1$ and $\nu_2$ is $j$.

Since the complete antichain $\ac{u}''$ and $\ac{v}$ have the same lengths and the node distance between $\nu_1$ and $\nu_2$ in $\ac{u}''$ is equal to the node distance between $\eta_1$ and $\eta_2$ in $\ac{v}$, it is easy to construct an element $g$ of $\Tnr$  such that $g\restriction_{U_{\nu_1}}=  g_{\nu_1, \eta_1}$ and $g\restriction_{U_{\nu_2}} = g_{\nu_2, \eta_2}$. 
\end{proof}

Using this transitivity result we show below  that $N(\Tnr)$ is equal to $N_{H(\CCnr)}(\Tnr)$.

\begin{lemma}
Let $h \in H(\CCnr)$ be such that $h \in N_{H(\CCnr)}(\Tnr)$. Then $h$ (and so $h^{-1}$) preserves the equivalence relation $\simeq$.
\end{lemma}
\begin{proof}
Suppose there are $x, y \in \CCnr$ such that $x \simeq y$ but $(x)h \not\simeq (y)h$. By relabelling if necessary we may assume that $(x)h < (y)h$. Since $(x)h \not \simeq (y)h$ there are $(\tau, \eta) \in \mathscr{G}$ such that $(x)h \in U_{\tau}$ and $(y)h \in U_{\eta}$. Let $j \in \{0,1 , \ldots, n-2\}$ be the reduced node distance between $\tau$ and $\eta$

Let $(\mu, \nu) \in  \mathscr{G}$ and consider the clopen sets $(U_{\mu})h$ and $(U_{\nu})h$. It is straight-forward to see that there are $\mu'$ and $\nu'$ such that $(\mu', \nu') \in \mathscr{G}$, the reduced node distance between $\mu'$ and $\nu'$ is equal to $j$, and $U_{\mu'} \subset (U_{\mu})h$ and $U_{\nu'} \subset (U_{\nu})h$ (the case where $U_{\mu'} \subset (U_{\nu})h$ and $U_{\nu'} \subset (U_{\mu})h$ is analogous). By Lemma~\ref{Lemma: can swap things with same node distance} there is an element $g \in \Tnr$ such that $g\restriction_{U_{\tau}} = g_{\tau,\mu'}$ and $g \restriction_{U_{\eta}} = g_{\eta, \nu'}$.

Now consider the product $h g h^{-1}$. Notice that $(x) h g h^{-1}$ is contained in the clopen set $U_{\mu}$ since $(x)h \in U_{\tau}$ and $(U_{\tau})g \in U_{\mu'} \subset (U_{\mu})h$, likewise $(y)h g h^{-1} \in U_{\nu}$. Therefore $(x)hgh^{-1} \not\simeq (y)hgh^{-1}$, which is a contradiction since $T_{n,r}$ preserves $\simeq$.
\end{proof}

\begin{lemma}
Let $h \in H(\CCnr)$ be such that $h \in N_{H(\CCnr)}(\Tnr)$, then $h \in N(\Tnr)$.
\end{lemma}
\begin{proof}
Since $h \in N_{H(\CCnr)}(\Tnr)$, then by the previous lemma $h$ preserves $\simeq$. Therefore $h$ induces a continuous function $g$ on  $S_{r}$. However since $h^{-1}$ is also in $N_{(H(\CCnr))}$, $g$ must be a homeomorphism and $(g)\phi = h$, therefore $h \in N(\Tnr)$.
\end{proof}

Thus we have now proved that $N_{H(\CCnr)}(\Tnr) = N(\Tnr) \cong 
\aut{\Tnr}$. 

\section{Automorphisms of \texorpdfstring{$T_{n,r}$}{Lg}}\label{Section:AutTnr}
 In what follows we adapt the results of 
\cite{BCMNO} to show that elements of 
$N_{H(\CCnr)}(\Tnr)$ can be represented by bi-synchronizing 
transducers.  However we shall need to define the group $\T{R}_{n,r}$ introduced in \cite{BCMNO} which is a slight  modification of the Rational group $\mathcal{R}_{n}$ in  \cite{GriNekSus}. The exposition in this section will largely mirror that found in \cite{BCMNO}.

First we introduce transducers generally then we introduce transducers over $\CCnr$.

Let $X_{I}$ and $X_{O}$ be finite sets of symbol. A transducer over the alphabet $X_{I}$ is a tuple $T = \gen{X_{I},X_O, Q_{T} \pi_{T}, \lambda_{T}}$ such that:

\begin{enumerate}[label = (\roman{*})]
\item  $X_{I}$  is the \emph{input alphabet} and $X_{O}$ is the \emph{output alphabet}.
\item $Q_T$ is the set of states of $T$.
\item $\pi_{T}: X_I \times Q_{T} \to Q_{T}$ is the \emph{transition function} and,
\item $\lambda_{T}: X_I \times Q_{T} \to X_{I}^{*}$ is the \emph{output function}.
\end{enumerate}

If $|Q_T| < \infty$ then we say that $T$ is  a finite transducer. If $X_{I} = X_{O} = X$ then we shall write $T = \gen{X, Q_{T}, \pi_{T}, \lambda_{T}}$. If we fix a state $q \in Q_{T}$ from which we begin processing inputs then we say that $T$ is \emph{initialised at $q$} and we denote this by $T_{q}$ and we call $T_{q}$ an \emph{initial transducer}. Given an initial transducer $T_{q_0}$ we shall write $T$ for the underlying transducer with no initialised states.

We inductively extend the domain of the transition and rewrite function to  $X_{I}^{*} \times Q_T $ by the following rules:
for a word $w \in X_{I}^{*}$, $i \in X_{I}$ and any state $q \in Q_T$ we have $\pi_T(wi, q ) = \pi_T(i, \pi_T(w, q))$ and $\lambda_{T}(wi, q) = \lambda_{T}(w, q)\lambda_{T}(i, \pi_{T}(w, q)) $. We then extend the domain of $\pi_{T}$ and $\lambda_{T}$ to  $X_{I}^{\omega} \times Q_{T}$.

In this paper we shall insist that for $\delta \in X_{I}^{\omega}$ and any state $q \in Q_T$ we have $\lambda_{T}(\delta, q) \in X_{O}^{\omega}$. This means that for a state $q \in Q_T$ the initial transducer $T_{q}$ induces a continuous function $h_{T_{q}}$ ($h_{q}$ if it is clear from the context that $q$ is a state of $T$) from $X_I^{\omega}$ to $X_O^{\omega}$. For $q \in Q_{T}$ we denote by $\im(q)$ the image of the map $h_{q}$; if $h_{q}$ is a homeomorphism form $X_{I}^{\omega} \to X_{O}^{\omega}$, then we call $q$ a \emph{homeomorphism state}.

Two states $q_1$ and $q_2$ of $T$ are called $\omega$-equivalent if $h_{q_1} = h_{q_2}$. A state $q$ of $T$ is called a state of incomplete response if for some $i \in X_{I}$ $\lambda_{T}(i, q)$ is not equal to the greatest common prefix of the set $\{(i \delta)h_{q} \mid \delta \in X_{I}^{\omega}\}$. If $T$ is an initial transducer with initial state $q_0$, then $q$ is called \emph{accessible} if  there is a word $w \in X_I^*$ such that $\pi_{T}(w, q_0) = q$. If all the states of $T_{q_0}$ are accessible then $T_{q_0}$ is called \emph{accessible}.

An initial transducer $T_{q_0}$ is called \emph{minimal} if $T_{q_0}$ is accessible, has no states of incomplete response and no pair of $\omega$-equivalent states. The initial transducer $T_{q_0}$ is also called invertible if the state $q_0$ is a homeomorphism state. Given an initial transducer  $T_{q_0}$ there is a unique minimal transducer $S_{p_0}$ $\omega$-equivalent to $T_{q_0}$ (\cite{GriNekSus}).

We give below the method given in \cite{GriNekSus} for constructing the inverse of an invertible, minimal transducer $T_{q_0}$. We first define the following function.

\begin{Definition}
Let  $T_{q_0} = \gen{X_{I}, X_{O}, Q_{T}, \pi_{T}, \lambda_{T}}$ be an invertible minimal transducer and $q$ a state of $T_{q_0}$. Define a function $L_{q}: X_{O}^{+} \to X_{I}^{\ast}$ by $(\nu)L_{q} = \varphi$ where  $\varphi$ is the greatest common prefix of the set $(U_{\nu})h_{q}^{-1}$.
\end{Definition}

Observe that since $T_{q_0}$ is minimal and invertible, then each state $q$ of $T_{q_0}$ induces an injective function from $X_{I}^{\omega}$ to $X_{O}^{\omega}$ with clopen image. From this one can deduce that for each state $q$ of $T_{q_0}$ the set of words $w \in  X_{O}^{\ast}$ such that $(w)L_{q} = \epsilon$ and $U_{w} \subset \im(q)$ is finite (see \cite{GriNekSus}).

We now form a transducer $T_{(\epsilon, q_0)} = \gen{X_{O}, X_{I}, Q'_{T}, \pi'_{T}, \lambda'_{T}}$ where $Q'_{T} = \{ (w,q) \mid  q \in Q_{T}, (w)L_{q} = \epsilon, U_{w} \subset \im(q) \}$ and $\pi'_{T}$ and $\lambda'_{T}$ are defined, for all $i \in X_{O}$ and $(w,q) \in Q'_{T}$, by the rules:
\begin{enumerate}[label=(\roman*)]
\item  $\pi'_{T}(i, (w,q)) = (wi - \lambda_{T}((wi)L_{q},q), \pi_{T}((wi)L_{q}, q)$ and, 
\item $\lambda'_{T}(i, (w,q)) = (wi)L_{q}$.
\end{enumerate} 

The following proposition is a result in \cite{GriNekSus}:

\begin{proposition}\cite{GriNekSus}
For $T_{q_0}$ a minimal invertible transducer, the transducer $T_{(\epsilon, q_0)}$ is well-defined, has no states of incomplete response, and satisfies $h_{(\epsilon, q_0)} = h_{q_0}^{-1}$.
\end{proposition}

We observe that given a minimal invertible transducer $T_{q_0}$, the transducer $T_{(\epsilon,q_0)}$ is  accessible if $T_{q_0}$. This is because if for any word $\Gamma \in \Xnp$, $(\Gamma)L_{q_0} = \epsilon$, then $\pi'_{T}(\Gamma,  (\epsilon, q_0)) = (\Gamma, q_0)$. Furthermore,  for any word $w \in \Xnp$  and state $q \in Q_{T}$, such that  $U_{w} \subset  \im(q)$ and $(w)L_{q} = \epsilon$, picking a word $\Gamma \in \Xnp$ such that  $\pi_{T}(\Gamma,q_0) = q$, we observe that $(\lambda_{T}(\Gamma, q_0)w)L_{q_0} = \Gamma$ and so $\pi'_{A}(\lambda_{T}(\Gamma, q_0)w, (\epsilon,q_0)) = (w,q)$. Note that given a minimal invertible transducer $T_{q_0}$, the transducer $T_{(\epsilon, q_0)}$, even though it has no states of incomplete response and is accessible, might not be minimal. 

Given two  transducers $A = \gen{X, Q_{A}, \pi_{A}, \lambda_{A}}$ and $B = \gen{X, Q_B, \pi_B, \lambda_B}$ then the  product $A*B = \gen{X, Q_{A*B}, \pi_{A*B}, \lambda_{A*B}}$ is the transducer defined as follows. The set of states $Q_{A*B}$ of $A*B$ is equal to the cartesian product $Q_{A} \times Q_{B}$. The transition and output function of $A*B$ are given by the following rules: for states $q \in Q_A$ , $p \in Q_B$ and $i \in X$ we have $\pi_{A*B}(i, (q,p)) = (\pi_{A}(i,q),\pi_B(\lambda_{A}(i,q), p) )$ and $\lambda_{A*B}(i, (q,p)) = \lambda_{B}(\lambda_{A}(i, q), p)$. For two initial transducers  $A_{q_0}$ and $B_{p_0}$, where $A$ and $B$ are the resulting transducers with no initialised states, the product of the initial transducer $A_{q_0} * B_{q_0}$, is the initial transducer $(A*B)_{(q_0, p_0)}$. It is straightforward to see that $h_{(A*B)_{(q_0, p_0)}} = h_{A_{q_0}}\circ h_{B_{p_0}}$.

A transducer (initial or non-initial) $T = \gen{X_I, X_O, Q_T, \pi_T, \lambda_T}$ is said to be \emph{synchronizing at level $k$} if there is a natural number $k \in \N$ and a map $\mathfrak{s}: X_I^{k} \to  Q_T$ such that for a word $\Gamma \in X_I^{{k}}$ and for any state $q \in Q_T$ we have $\pi_{T}(\Gamma, q) = (\Gamma)\mathfrak{s}$. We will denote by $\core(T)$ the sub-transducer of $T$ induced by the states in the image of $\mathfrak{s}$. We call this sub-transducer the \emph{core of $T$}. If $T$ is equal to its core then we say that $T$ \emph{core}. Viewed as a graph $\core(T)$ is a strongly connected transducer. If $T$ is an initial transducer $T_{q_0}$  which is invertible, then we say that $T_{q_0}$ is bi-synchronizing if both $T_{q_0}$ and its inverse are synchronizing. Note that when $T$ is synchronous, then we shall say $T$ is bi-synchronizing if $T$ and its inverse are synchronizing.

Given  a transducer $T$ and states $q_1, q_2$ of $T$ we shall sometimes use the phrase $q_1$ and $q_2$ \emph{transition identically  on a subset $W \subset \Xns$} to mean that the functions $\pi_{T}(\centerdot, q_1): W \to Q_{T}$ and $\pi_{T}(\centerdot, q_2): W \to Q_{T}$ are identical. We might also say that $q_1$ and $q_2$ \emph{read all elements of $W$ to the same location}.

Now we introduce transducers over $\CCnr$. An initial transducer for $\CCnr$ is a tuple $A_{q_0}=(\dotr, X_n, R, S, \pi, \lambda, q_0)$ such that:

\begin{enumerate}[label = (\roman*)]
\item $R$ is a finite set, and the set $Q$ of states of  $A$ is the disjoint union $R \sqcup S$. The state $q_0 \in R$ is the initial state
\item $\pi: \dotr \times \{q_0\} \sqcup X_n \times  Q \to Q\backslash \{q_0\}$ and $\lambda: \dotr \times q_0 \sqcup X_n \times Q \to X_{n,r}^{\ast} \sqcup X_n^{\ast}$
\end{enumerate}

Notice that a letter from $\dotr$ can only be read from the initial state $q_0$ and we can never return to $q_0$ after leaving $q_0$. We Inductively extend the domain of $\pi$ and $\lambda$ to $ \dotr \times \{q_0\} \sqcup  X_n^{*} \times Q$ by the following rules:

$\pi(wx, q) = \pi(x, \pi(w,q))$ and $\lambda(wx, q) = \lambda(w, q)\lambda(x, \pi(w,q))$. Where $w \in X_{n,r}^{+} \sqcup X_n^{+}$ and $x \in X_n$ (if $w \in X_{n,r}^{+}$ then $q = q_0$). We also take the convention that $\pi(\epsilon, q) = q$ and $\lambda(\epsilon, q) = \epsilon$ for any state $q$ of $A$. By transfinite induction we may further extend the domains of $\pi$ and $\lambda$ to $ \dotr \times \{q_0\} \sqcup  X_n^{\omega} \times Q$.

We impose the  following rules on $\pi$ and $\lambda$:

\begin{enumerate}[label = (\arabic*)] 
\item For a state $r \in R $ and for $i$ in $\dotr \sqcup X_n$ such that $\pi(i, r)$ is defined, if $\pi(i, r) \in R$ then $\lambda(i, r) = \epsilon$, otherwise $\lambda(i, r) \in X_{n,r}^{*}$. \label{List: conditions of pi and lambda 1}
\item  For $x \in X_n$ and $q \in S$, $\lambda(x, q) \in X_{n}^{*}$  and $\pi(x, q) \in S$. \label{List: conditions of pi and lambda 2}
\item For a state $s \in S$  and $\delta \in \CCn$ we have that $\lambda(\delta, s) \in \CCn$. \label{List: conditions of pi and lambda 3}
\item If there is a word  $w \in X_n^{+}$ and a state $q \in Q$ such that $\pi(w, q) = q$ then $q \in  S$. \label{List: conditions of pi and lambda 4}
\end{enumerate}

These rules serve the purpose of ensuring that  whenever an element of $\CCnr$ is  is processed through $A_{q_0}$ the output is also in $\CCnr$.

Let $A_{q_0}$ be an initial transducer on $\CCnr$ as above and let $q$ be a state of $A_{q_0}$. Let $A_q$ denote the initial transducer $A_{q_0}$ where we process inputs from the state $q$. Observe that $A_{q_0}$ induces a continuous function $h_{A_{q_0}}$ (or $h_{q_0}$ if it clear that $q_0$ is the initial state of $A$) from $\CCnr$ to itself. Furthermore every non-initial state $q$ of $A_{q_0}$ which is also an element of $R$ induces a continuous function $h_{q}$ from $\CCn$ to $\CCnr$, otherwise it induces a continuous function from $\CCn$ to itself. Once again we denote by $\im(q)$ the image of $q$ and call $q$ a homeomorphism state if $h_{q}$ is a homeomorphism from its domain to its range.

We extend, in the natural way, the definition of accessibility, accessible transducers,and states of incomplete response given in the general setting to the specific setting of transducers over $\CCnr$. We also extend the function $L_{q}$ for a minimal invertible transducer $T_{q_0}$ over $\CCnr$ and $q \in Q_{T}$. Having done this, we may thus define, given $T_{q_0}$ a minimal, invertible transducer, the transducer $T_{(\epsilon, q_0)}$ such that $h_{(\epsilon, q_0)} = h_{q_0}^{-1}$.

We say that a transducer $A_{q_0}$ over $\CCnr$ is synchronizing  if there is a $k \in \mathbb{N}$ such that given any word $\Gamma$ of length $k$ in $\Xnrs \sqcup \Xns$ the active state of $A_{q_0}$ when $\Gamma$ is processed from any \emph{appropriate} state of $A_{q_0}$ is completely determined by $\Gamma$. Thus we may also extend the notions of `core' for synchronizing transducers over $\CCn$. We now introduce the notion of $\omega$-equivalence and minimality for transducers over $\CCnr$. Two initial automata with the same domain and range are said to be \emph{$\omega$-equivalent} if they induce the same continuous function from their domain to their range. 

An initial transducer $A_{q_0}$ is called \emph{minimal} if $A_{q_0}$ is accessible, no states of $A$ are states of incomplete response and for any distinct pair $q_1, q_2$ of states of $A_{q_0}$, $A_{q_1}$ and $A_{q_2}$ are not $\omega$-equivalent.  In \cite{BCMNO} the authors show, by slight modifications of arguments in \cite{GriNekSus}, that for an initial transducer $A_{q_0}$ on $\CCnr$ there is a unique minimal transducer under $\omega$-equivalence.

The product $(A*B)_{(q_0,p_0)} = \gen{\dotr, \Xn, R_{A} \times R_{B} \sqcup  S_{A} \times R_{B}, S_A \times S_B, \pi_{A*B}, \lambda_{A*B}}$ of the initial transducers $A_{q_0}$ and $B_{p_0}$ over $\CCnr$ is defined as follows. The set of states $Q_{A*B}  = R_{A} \times R_{B} \sqcup  S_{A} \times R_{B} \sqcup S_A \times S_B$, and the state $(q_0, p_0) \in  R_{A} \times R_{B}$ is the initial state. The transition and output functions are defined as follows. First for $a \in \dotr$ we have $\pi_{A*B}(a, (q_0, p_0)) = (\pi_{A}(a, q_0), \pi_{B}(\lambda_{A}(a, q_0), p_0 )$, and $\lambda_{A*B}(a, (p_0, q_0)) = \lambda_{B}(\lambda_{A}(a, q_0), p_0)$. Now for any pair $(q, p) \in  Q_{A,B}$ and for any $i \in X_n$ we have $\pi_{A*B}(i, (q,p)) = (\pi_{A}(i, q), \pi_B(\lambda_{A}(1, q),p)$ and $\lambda_{A*B}(i, (q,p)) = \lambda_{B}(\lambda_{A}(i,q),p)$. Now observe that as $A$ and $B$ satisfy condition \ref{List: conditions of pi and lambda 1} to \ref{List: conditions of pi and lambda 4} above then so does the product $(A*B)_{(q_0, p_0)}$. Furthermore, as before, it is a straightforward observation that $h_{(A*B)_{(q_0,p_0)}} = h_{A_{q_0}}\circ h_{B_{p_0}}$.

Below we outline a procedure given in \cite{BCMNO} for constructing from a homeomorphism $h$ of $\CCnr$, an initial transducer $A_{q_0}$ of $\CCnr$ such that $h_{q_0} = h$.

We first need to define local actions.

For an arbitrary homeomorphism $g: \CCnr \to \CCnr$,  fix the notation   $P_{g} \subset  X_{n,r}^{*}$ for the unique maximal subset of $X_{n,r}^{*}$ set satisfying: $(U_{\nu})g \subseteq U_{\dot{a}}$ for $\nu \in P_{h}(a)$, and for any proper prefix $\mu$ of $\nu$ there are elements $\delta_1, \delta_2 \in \CCn$ and $\dot{a}_1, \dot{a}_{2} \in \dotr$ such that $(\mu\delta_1)g \in U_{\dot{a}_{1}}$ and $(\mu\delta_2)g \in U_{\dot{a}_{2}}$. Observe that since $g$ is a homeomorphism  and since $\sqcup_{\dot{a} \in \dotr} (U_{\dot{a}})g^{-1}$ is clopen, $P_{g}$ exists. Moreover maximality of $P_{g}$ implies that $P_{g}$ is a complete antichain for $\Xnrs$.

Let $h : \CCnr \to \CCnr$ be a homeomorphism. Define $\theta_{h}: X_{n,r}^{*} \to X_{n,r}^{*}$ as follows: for $\mu \in X_{n,r}^{*}$ if $\mu$ is a prefix of some $\nu \in P_{h}$ set $(\mu)\theta_{h}:= \epsilon$, otherwise $\mu = \nu \chi$ for some $\nu \in P_{h}$ and $\chi \in X_{n,r}^{*}$, in this case set $(\mu)\theta_{h}$ to be the greatest common prefix of the set $(U_{\mu})h$, since $h$ is a homeomorphism and by choice of $P_{h}$, $(U_{\mu})h \in X_{n,r}^{+}$.

Now for $\mu \in  X_{n,r}^{*}$  we define a map $h_{\mu}$ on  $\CCn$ by $(\delta)h_{\mu} = (\mu\delta)h_{\mu} - (\mu)\theta_{h}$. We call $h_{\mu}$ the \emph{local action of $h$ at $\mu$}. Observe that if $\mu$ is a prefix of an element of $P_{h}$ then the range of $h_{\mu}$ is $\CCnr$ otherwise the range of $\mu$ is $\CCn$, in either case $h_{\mu}$ is continuous. The following fact is straightforward, let $\mu, \nu \in X_{n,r}^{*}$ then $(\mu \nu)\theta_{h} = (\mu)\theta_{h}(\nu)\theta_{h_\mu}$.

Using the function $\theta_{h}$ we now construct an initial transducer $A_{q_0}$ of $\CCnr$ such that $h_{q_0}  = h$.

Let $h: \CCnr \to \CCnr$ be a homeomorphism. Form a transducer 
$A_{\epsilon} = (\dot{r}, X_n, R_{A}, S_{A}, \pi_A, \lambda_{A}, \epsilon 
)$. The set $Q_{A}$ of states of  $A$ is precisely 
$X_{n,r}^{*}$ and $R_A \subset Q_{A}$ is the set of proper 
prefixes of elements of $P_{g}$ and $S_A := \Xns \backslash R$. The 
transition of function and output functions $\pi_A$ and 
$\lambda_{A}$ are defined as follows. For $\dot{a} \in \dotr$ we 
have  that $\pi_{A}(\dot{a}, \epsilon) = \dot{a}$ and $\lambda_A 
(\dot{a}, \epsilon) = (\dot{a})\theta_{h}$; for $\nu \in \Xnrp$ 
and $i \in X_n$ we have $\pi_{A}(i, \nu) =\nu i$ and 
$\lambda_{A}(i, \nu) = (\nu i) \theta_{h}  - (\nu)\theta_{h}$.

Observe that the transducer $A_{\epsilon}$ satisfies the conditions \ref{List: conditions of pi and lambda 1} to \ref{List: conditions of pi and lambda 4}. The following claim is straightforward to prove:

\begin{claim}\label{Claim:tranducerarisingfromhomeo}
Let  $h: \CCnr \to \CCnr$ be a homeomorphism, and $A_{\epsilon} = (\dot{r}, X_n, R_{A}, S_{A}, \pi_A, \lambda_{A}, \epsilon 
)$ be the transducer constructed form $h$ as above. Let $\nu \in \Xnrs$ then for all $w \in \Xns$ if $\nu \ne \epsilon$ or $w \in \Xnrp$ if $\nu = \epsilon$ we have that $(w)\theta_{h_{\nu}} = \lambda_{A}(w, \nu)$.  
\end{claim} 
\begin{proof}
First suppose $\nu \ne \epsilon$ and let $x \in X_n$. Then by 
definition we have $\lambda_{A}(i,\nu) = (\nu i) \theta_{h} - 
(\nu)\theta_{h}$, however by an observation above we have $(\nu 
i) \theta_{h} - (\nu)\theta_{h}) = (i)\theta_{h_{\nu}}$. Now 
assume by that for all $w \in  \Xnp$ we have that $\lambda(w, \nu 
) = (w)\theta_{h_{\nu}}$. Let $i \in \Xn$ and consider 
$\lambda_{A}(w i, \nu)$. We may write $\lambda_{A}(w i, \nu) = 
\lambda_{A}(w, \nu) \lambda_{A}(i, \nu w)$ (since $\pi_A(w, \nu) 
= \nu w$). Therefore $\lambda_{A}(w i, \nu) = (w)\theta_{h_{\nu}} 
(\nu w i)\theta_{h} - (\nu w)\theta_{h}$. Observe that $(\nu w 
i)\theta_{h} - (\nu w)\theta_{h} = (i)\theta_{h_{\nu w}}$, 
therefore $\lambda_{A}(w i, \nu) = (w)\theta_{h_{\nu}} 
(i)\theta_{h_{\nu w}}$ however,  $(w)\theta_{h_{\nu}} 
(i)\theta_{h_{\nu w}} = (wi)\theta_{h_{\nu}}$.

Now suppose that $\nu = \epsilon$. Let $w \in \Xnrp$ and suppose $w = \dot{a}v$ for some $v \in \Xns$ and $\dot{a} \in \dotr$. Consider $\lambda_{A}(w, \epsilon)$, this can be broken up into $\lambda_{A}(\dot{a}w, \epsilon) = \lambda_{A}(\dot{a}, \epsilon)\lambda_{A}(v, \dot{a})$. Now $\lambda_{A}(\dot{a}, \epsilon) = (\dot{a})\theta_{h}$ by definition, and $\lambda_{A}(v, \dot{a}) = (v)\theta_{h_{\dot{a}}}$ by the previous paragraph. Observe that $(\dot{a}v)\theta_{h} = (\dot{a})\theta_{h} (v)\theta_{h_{\dot{a}}}$ therefore $\lambda_{A}(\dot{a}w, \epsilon) = (\dot{a}v)\theta_{h}$ as required.
\end{proof}

\begin{Remark}\label{Remark: finitely many local actions implies finite transducer}
From the claim we deduce that for a homeomorphism $h: \CCnr \to \CCnr$ and for $A_{\epsilon} =  (\dotr, \Xn, R_A, S_A, \pi_A, \lambda_A, \epsilon)$ the transducer constructed from $h$, we have that $h_{A_{\epsilon}} = h$, moreover for any $\nu \in \Xnrs$ and any local action $h_{\nu}$ of $h$ we have that $h_{\nu} = h_{A_{\nu}}$. Therefore if $h$ has finitely many local actions, then it follows that the minimal transducer on $\CCnr$,  under $\omega$-equivalence, representing $h$ has finitely many states. Moreover, since by Claim~\ref{Claim:tranducerarisingfromhomeo}, $A_{\epsilon}$ has no states of incomplete response, it follows that, if $B_{q_0}$ is the minimal transducer representing $A_{\epsilon}$, then for all  $\nu \in \Xnr$ $h_{\nu} = h_{q}$ for some state $q$ of $B$.
\end{Remark}

In what follows we show that all homeomorphisms of $\CCnr$ in the group $N(\Tnr)$ can be represented by a minimal finite transducer. The previous paragraph demonstrates that it suffices to show that such homeomorphisms have finitely many local actions.

Our approach shall essentially mirror that taken in \cite{BCMNO}, the arguments differ only in so far as we need to make modifications to allow for the fact that the action of  $\Tnr$ on $\CCnr$ is not as transitive as the action $\Gnr$ on $\CCnr$.

First we need the following result:

\begin{lemma}\label{Lemma: normalisers in H-tilde}
Let $X$ be a topological space and let $\sim$ be any equivalence relation on $X$. Let $H_{\sim}$ be the subgroup of $H(X)$ consisting of those elements of   $H(X)$ which respect $\sim$. Let $G \le H(X)$ be a subgroup  which fixes each equivalence class of $\sim$ and acts transitively on each  class. Then $N_{H(X)}(G) \le H_{\sim}$.
\end{lemma}

Observe that $T_{n,r} \le H(\CCnr)$, by Lemma ~\ref{Lemma: Tnr preserves tail classes and acts transitively on each tail class}, preserves $\sim_{t}$ and acts transitively on each equivalence class of $\sim_{t}$. Therefore, as  corollary of the lemma above, we have $N(\Tnr) \le H_{\sim_{t}}$. 

The following definitions and lemmas appear in  \cite{BCMNO} and introduce crucial notions and ideas in understanding local actions of homeomorphisms of $\CCnr$.

\begin{Definition}
Let $V \subseteq \CCnr$ be a clopen set. Let $B \subset \Xnrs$ be the minimal antichain such that $U_{B}:=\{U_{\nu} \mid \nu \in B \}$ and for all $\mu$ a prefix of some element of $B$ we have $U_{\mu} \not \subset V$. We call $U_{B}$  the \emph{decomposition of $V$} and denote it by  $\dec(V)$. 
\end{Definition}

\begin{Definition}\label{Def: acting in the same fashion}
Let $h \in H(\CCnr)$, and $U_{\nu}$ and $U_{\eta}$ be elements of $\Banr$ for $\nu, \eta \in \Xnrs$. Then we say that $h$ acts on $U_{\nu}$ and $U_{\eta}$ in \emph{the same fashion} if $h_{\nu} = h_{\eta}$. If $V,W \subset \CCnr$ are a clopen subsets then we say that $h$ acts on $V$ and $W$ \emph{in the same fashion} if for any $U_{\nu} \in \dec(V)$ and $U_{\eta} \in \dec(W)$, $h$ acts on $U_{\nu}$ and $U_{\eta}$ in the same fashion.
\end{Definition}

\begin{Definition}\label{Def: almost in the same fashion}
Let $h \in H(\CCnr)$ and let $V, W \subset \CCnr$ be clopen subsets. Then we say that $h$ acts on $V$ and $W$ \emph{almost in the same fashion} if  there is some $k \in \mathbb{N}$ such that for $U_{\nu} \in \dec(V)$ and $U_{\eta} \in \dec(W)$, and for any $\chi \in  \Xn^{k}$ $h$ acts on $U_{\nu \chi}$ and $U_{\eta\chi}$ in the same fashion. We call the minimal $k$ satisfying this condition the \emph{critical level of $V$ and $W$} and denote it by $\crit_{h}(V, W)$.
\end{Definition}

\begin{Definition}\label{Def: (almost) in the same fashion uniformly}
Let $h \in H(\CCnr)$ and let $V, W \subset \CCnr$ be clopen subsets. We say that $h$ acts  on $V$ and $W$ \emph{in the same fashion uniformly} if for any $\chi \in \Xns$ $h$ and $U_{\nu} \in  \dec(V)$ and  $U_{\eta} \in \dec(W)$, $h$ acts on $U_{\nu}$ and $U_{\eta \chi}$ in the same fashion. We say that $h$ acts \emph{almost in the same fashion uniformly} on $U$ and $W$ if there is some $k \in \N$ such that for $U_{\nu}  \in \dec(V)$ and $U_{\eta} \in \dec(W)$ and for any $\chi, \xi \in \Xn^{k}$, $h$ acts on $U_{\nu\chi}$ and $U_{\eta \xi}$ in the same fashion.
\end{Definition}

\begin{Remark}\label{Remark: Tnr acts almost in the same fashion uniformly}
Let $V$ and $W$ be clopen subsets of cantor space  and $g \in \Tnr$ then $g$ acts on $V$ and $W$ almost in the same fashion uniformly.
\end{Remark}

The following lemmas are crucial in our understanding of local actions of elements of $N(\Tnr)$ and are taken from \cite{BCMNO}:

\begin{lemma}\label{Lemma: g acts in same fashion h acts in same fashion gh acts in same fashion}
 Let $g,h \in H(\CCnr)$ and let $U_{\nu}, U_{\eta} \in \Banr$. Suppose there are $\nu', \eta' \in \Xnrp$ such that $(U_{\nu})g = U_{\nu'}$ and $(U_{\eta})g = U_{\eta'}$, $g$ acts  in the same fashion  on $U_{\nu}$ and $U_{\eta}$ and $h$ acts in the same fashion on $U_{\nu'}$ and $U_{\eta'}$, then $gh$ acts in the same fashion on $U_{\nu}$ and $U_{\eta}$. 
\end{lemma}

\begin{lemma}\label{Lemma: g acts in same fashion, h acts in almost same fashion uniformly, then gh acts in almost the same fashion}
Let $g,h \in H(\CCnr)$ and let $U_{\nu}, U_{\eta} \in \Banr$. Suppose $g$ acts on $U_{\nu}$ and $U_{\eta}$ in the same fashion and $h$ acts on $(U_{\nu})g$ and $(U_{\eta})g$ in almost the same fashion uniformly, then $gh$ acts in almost the same fashion on $U_{\nu}$ and $U_{\eta}$. 
\end{lemma}

We shall also need the following proposition from the same source:

\begin{proposition}\label{Prop: conjugators have same local action at pair of nodes}
Let $h \in H(\CCnr)$ and let $U_{\tau}$ and $U_{\eta}$ be elements of $\Banr$. Suppose that $h_{\tau \chi} \ne h_{\eta \chi}$ holds for every $\chi \in X_n^{\ast}$. Then  $h \notin H_{n,r, \sim_{t}}$.
\end{proposition}

We have the following corollary:

\begin{corollary}\label{Cor:conjugators have same local action at pair of nodes with any node distance}
Let $h \in H_{n,r \sim_{t}}$ and let $i \in \{0, 1, \ldots, n-2\}$. Then there exists $\tau, \eta \in X_{n,r}$ such that $U_{\tau} \cup U_{\eta} \ne \CCnr$, $(\tau, \eta) \in \mathscr{G}$ the reduced node distance between $\tau$ and $\eta$ is  $i$ modulo $n-1$, and $h_{\tau}  = h_{\eta}$.
\end{corollary}
\begin{proof}
This is a direct consequence of Proposition~\ref{Prop: conjugators have same local action at pair of nodes}, since we may chose a pair $(\tau, \eta) \in  \mathscr{G}$ such that there is a complete antichain $\ac{u}$ containing $\tau$ and $\eta$ and the node distance in $\ac{u}$ between $\tau$ and $\eta$ is congruent to $i$ modulo $n-1$. Moreover, by observations above, for any $\chi \in X_n^{*}$, $(\tau \chi, \eta \chi) \in \mathscr{G}$  and any complete antichain $\ac{v}$ containing $\tau\chi$ and $\eta\chi$, the node distance between $\tau\chi$ and $\eta \chi$ in $\ac{v}$ is congruent to $i$ modulo $n-1$.  
\end{proof}

We have the following lemma:

\begin{lemma} \label{Lemma:  automorphisms have few local actions}
Let $h \in H(\CCnr)$ such that $h^{-1}T_{n,r}h \subseteq T_{n,r}$. Then for every $U_{\nu}, U_{\eta} \in \Banr$ such that $U_{\nu} \cup  U_{\eta} \ne \CCnr$, the map $h$ acts on  $U_{\nu}$ and $U_{\eta}$ almost in the same fashion.
\end{lemma}
\begin{proof}

First observe that, by Lemma~\ref{Lemma: normalisers in H-tilde}, 
$h \in H_{\sim_{t}}$. Moreover by Corollary~\ref{Cor:conjugators 
have same local action at pair of nodes with any node distance} 
for any $i \in \{0,1,\ldots, n-2\}$ there are elements $\nu$ and 
$\eta$ of $X_{n,r}^{+}$ such that $(\nu, \eta) \in \mathscr{G}$, 
the reduced node distance between $\nu$ and $\eta$ is $i$, 
$U_{\nu} \cup U_{\eta} \ne \CCnr$, and $h_{\nu} = h_{\eta}$. Fix 
$i \in \{0,1, \ldots, n-2\}$ and let $(\nu, \eta) \in \mathscr{G}$ 
be such that the reduced node distance between $\nu$ and $\eta$ 
is $i$ and $h_{\nu} = h_{\eta}$.

Let $\nu'$ and $\eta'$ be elements of $X_{n,r}^{+}$ be such that $(\nu', \eta') \in \mathscr{G}$ and the reduced node distance between $\nu'$ and $\eta'$ is also equal to $i$. By Lemma ~\ref{Lemma: can swap things with same node distance} there is an element $g \in \Tnr$ such that  $g\restriction_{U_{\nu'}} = g_{\nu', \nu}$ and $g\restriction_{U_{\eta'}} = g_{\eta', \eta}$.

By Lemma~\ref{Lemma: g acts in same fashion h acts in same fashion gh acts in same fashion}, since $h$ acts on  $U_{\nu}$ and $U_{\eta}$ in the same fashion, and $g$ acts on $U_{\nu'}$ and $U_{\eta'}$ in the same fashion, then $gh$ acts on $U_{\nu}$ and $U_{\eta}$ in the same fashion. Let $f = h^{-1}gh$. By assumption $f \in T_{n,r}$ and so $f$ and $f^{-1}$ act on any pair of clopen sets of $\CCnr$ in almost the same fashion. Therefore, since $h = ghf^{-1}$, by Lemma~\ref{Lemma: g acts in same fashion, h acts in almost same fashion uniformly, then gh acts in almost the same fashion} it follows that $h$ acts on $U_{\nu}$ and $U_{\eta}$ in almost the same fashion.

Since $i$ was arbitrarily chosen, we conclude that for any pair $(\tau, \mu) \in \mathscr{G}$ with reduced node distance equal to $i$, $h$ acts on $U_{\tau}$ and $U_{\eta}$ in almost the same fashion.

Now let $\nu'$ and $\eta'$ be arbitrary elements of $X_{n,r}^{*}$ such that $U_{\nu'} \cup U_{\eta'} \ne \CCnr$ . Observe that $\nu', \eta' \in X_{n,r}^{+}$ otherwise  $U_{\nu} = \CCnr$ or $U_{\eta} = \CCnr$. Since $U_{\nu'} \cup U_{\eta'} \ne \CCnr$ there exists an element $\beta \in X_{n,r}^{+}$ such that $\beta \perp \nu'$ and $\beta \perp \eta'$ and such that $(\nu', \beta) \in \mathscr{G}$ and $(\eta', \beta) \in \mathscr{G}$. By arguments in the previous paragraph we therefore have that $h$ acts on $U_{\nu'}$ and $U_{\beta}$ in almost the same fashion, and $h$ acts on $U_{\eta'}$ and $U_{\beta}$ in almost the same fashion. Let $k_1 = \crit_{h}(U_{\eta'}, U_{\beta} )$ and $k_2 = \crit_{h}(U_{\nu'}, U_{\beta} )$. Let $k = \max\{k_1, k_2\}$. Then for any $\chi \in X_n^{k}$ we have that $h_{\nu'\chi} = h_{\beta\chi} = h_{\eta' \chi}$. Therefore we conclude that $h$ acts on $U_{\eta'}$ and $U_{\nu'}$ in almost the same fashion.
\end{proof}

The Corollary below demonstrates that if $h \in H(\CCnr)$ is such that  $h^{-1}\Tnr h \subseteq \Tnr$, then $h$ uses only finitely many types of local action. This Corollary should be compared with Corollary 6.16 of \Cite{BCMNO} and is proved almost identically; we reproduce the proof here for completeness. This is a key result in demonstrating that automorphisms of $\Tnr$ can be represented by bi-synchronizing transducers.

\begin{corollary}\label{Corollary: Automorphisms have finite local actions}
Let $h \in H(\CCnr)$ be such that $h^{-1}\Tnr h \subseteq \Tnr$, then $h$ uses only finitely many types of local action.
\end{corollary}
\begin{proof}
Let $A$ be a complete antichain for  $\Xnrs$ having at least 3 elements. For instance we may take $A = X_{n,r}^{3}$. Notice that for any pair of element $\nu$ and $\eta$ in $A$ we have  by Lemma ~\ref{Lemma:  automorphisms have few local actions} that $h$  acts on $U_{\nu}$ and $U_{\eta}$ almost in the same fashion, and for $i \in X_n$ $h$ acts on $U_{\nu}$ and $U_{\nu i}$ almost in the same fashion. Set 

$$ k= \max \{ \crit_{h}(U_{\nu},U_{\eta}) \mid \nu, \eta \in A, \mbox{ or } \nu \in A \mbox{ and } \eta= \nu i \mbox{ for some } i \in \Xn  \}$$

We now demonstrate that for any $ m\ge  k+3$, $m \in \N$ and for any word $\eta \in X_{n,r}^{m}$ we have that $h_{\eta} = h_{\nu \xi}$ for some $\nu \in A$ and $\xi \in X_n^{k}$. We proceed by induction on $m$.

The base case is trivially satisfied. Let $m \in \N$ be strictly greater than $k+3$, and assume that for all $k+3< l < m$ we have that for any word $\eta \in \Xnr^{l}$ $H_{\eta} = h_{\nu \xi}$ for some $\nu \in A$ and $\xi \in \Xn^{k}$. 

Now let $\tau \in \Xnr^{m}$. We may write $\tau = \nu i \xi$ for $\nu \in A$, some $i \in \Xn$ and $\xi \in X_n^{\ast}$ such that $|\xi| \ge k$ (since $m$ is greater than $k+3$).

Observe that as $h_{\nu} = h_{\nu i}$ we have that $h_{\tau} = h_{\nu \xi}$. Now as $|\nu \xi| < m$, by the inductive assumption we have that there is some $\chi \in X_n^{k}$ and $\mu \in A$ such that $h_{\nu \xi} = h_{\mu \chi}$ and the conclusion follows.
\end{proof}

\begin{Remark}
The above corollary together with Remark~\ref{Remark: finitely many local actions implies finite transducer} means that $N(\Tnr)$ is a subgroup of the group  $\Rnr$ of those homeomorphisms of $\CCnr$ that can be represented by minimal, finite, initial transducers. In what follows we shall demonstrate the element of $N(\Tnr)$ can be represented by minimal, bi-synchronizing transducers in $\Rnr$. 
\end{Remark}

\section{The group \texorpdfstring{$N(\Tnr)$}{Lg} is a subgroup of \texorpdfstring{$\Bnr$}{Lg}} \label{Section: N(Tnr) is a subgroup of Bnr}
Following \cite{BCMNO} we shall use the notation $\Bnr$ for the group consisting of those elements of $\Rnr$ which can be represented by bi-synchronizing transducers. In this section we demonstrate that $N(\Tnr)$ is a subgroup  $\T{T}\Bnr$ of $\Bnr$.

The lemma below is straight-forward and connects the property of a homeomorphism acting in the almost the same fashion (Definition~\ref{Def: almost in the same fashion}), to the synchronizing property of a transducer representing the homeomorphism.

\begin{lemma}\label{Lemma:automorphisms have synch transducers}
Let $h \in H(\CCnr)$ satisfy the following two conditions,
\begin{enumerate}[label = (\alph*)]
\item $h$ has finitely many local actions,\label{Lemma:automorphisms have synch transducers assumption 1}
\item for every pair $U_{\nu}, U_{\eta} \in \Banr$ such that $U_{\nu} \cup U_{\eta} \ne \CCnr$ $h$ acts on $U_{\nu}$ and $U_{\eta}$ almost in the same fashion.\label{Lemma:automorphisms have synch transducers assumption 2}
\end{enumerate}
Then the minimal transducer representing $h$ is synchronizing.
\end{lemma}
\begin{proof}
Let $A_{q_0}$ be the minimal transducer representing $h$. As $h$ has finitely many local actions, by Remark~\ref{Remark: finitely many local actions implies finite transducer}, $A_{q_0}$ is a finite transducer. Applying Remark~\ref{Remark: finitely many local actions implies finite transducer}, there is an $m \in \N$ such that if  $q \in Q_{A}$ is accessible from $q_0$ by a word of length at least $m$, then  $h_{q} = h_{\eta}$ for some $\eta \in \Xnrs$ of length at least $m$. Of course by the definition of $A_{q_0}$, it is the case that $h_{\nu} = h_{p}$, for $p = \pi_{A}(\nu, q_0)$ and $\nu \in \Xnrs$. We may also choose $m$ bigger than $1$ and long enough, so that for any pair $\nu, \eta \in \Xnrs$ of length $m$, $U_{\nu} \cup U_{\eta} \ne \CCnr$. Fix such an $m \in \N$.

By assumption~\ref{Lemma:automorphisms have synch transducers assumption 2} above and Definition~\ref{Def: almost in the same fashion}, it follows that for every pair $\nu, \eta \in \Xnrs$ of length $m$, there is a minimal number $k_{\nu, \eta} \in \N$ such that for all words $\xi \in \Xn^{k_{\nu, \eta}}$, $h_{\nu \xi} = h_{\eta \xi}$. Set $k:= \max_{\nu,\eta \in \Xn^{m}}\{k_{\nu, \eta}\}$. It therefore follows that for any pair $\tau, \mu \in \Xnrs$ of length $m$, and any word $\chi \in \Xn^{k}$, $h_{\tau \chi} = h_{\mu \chi}$.

Now since any state $q$ of $A_{q_0}$ accessible from $q_0$ by a word of length at least $m$ satisfies $h_{q} = h_{\eta}$ for some $\eta \in \Xnrs$ of length $m$. It is therefore the case, by minimality of $A_{q_0}$, that for any $\chi \in \Xn^{k}$  and any pair $\nu, \eta \in \Xnrs$ of length $m$, we have $\pi_{A}(\nu\chi, q_0) = \pi_{A}{\eta\chi, q_0}$. Thus we have that if $Q_{A,\ge m} = \{\pi_{A}(\nu, q_0) \mid \nu \in \Xnrs, |\nu| \ge m \}$, then for any $\xi \in \Xn^{k}$, $|\{\pi_{A}(\chi,p) \mid p \in Q_{A, \ge m} \}| = 1$. From this we deduce that $A_{q_0}$ is synchronizing at level at least $k+m$.

\end{proof}

As a corollary we have the following:

\begin{corollary}\label{Cor: Transducers representing automorphims are synchronizing}
Let $h \in H(\CCnr)$ be such that $h^{-1}T_{n,r}h \subseteq \Tnr$ and $A_{q_0} = \gen{\dotr, \Xn, Q_A, \pi_A, \lambda_A, q_0}$ be an initial transducer such that $h_{q_0} = h$. Then $A_{q_0}$ is  synchronizing.
\end{corollary}
\begin{proof}
By Corollary~\ref{Corollary: Automorphisms have finite local actions} and Lemma~\ref{Lemma:  automorphisms have few local actions} it follows that $h$ and $h^{-1}$ satisfy the conditions of Lemma~\ref{Lemma:automorphisms have synch transducers}. By definition of bi-synchronicity, the minimal transducer $A_{q_0}$ representing $h$ is bi-synchronizing.
\end{proof}

Now we may prove the main result.

\begin{Theorem}
Let $h \in H(\CCnr)$ then $h \in N(\Tnr)$ if and only if there is a bi-synchronizing transducer $A_{q_0}$ such that $h_{q_0} = h$ and preserves $\simeq$.
\end{Theorem}
\begin{proof}
It follows by Corollary~\ref{Cor: Transducers representing automorphims are synchronizing} that if $A_{q_0}$ is a transducer such that $h_{q_0} \in N(\Tnr)$ then $A_{q_0}$ is synchronizing. Since $N(\Tnr)$ is a group it follows that $A_{q_0}$ is bi-synchronizing. This proves the forward implication.

For the reverse implication, suppose that $A_{q_0} = \gen{\dotr, \Xn, Q_A, \pi_A, \lambda_A, q_0}$ is a bi-synchronising transducer such that $h_{q_0}$ preserves $\simeq$. Let $T \in T_{n,r}$ observe that since $h_{q_0}$ preserves $\simeq$ it follows that $h^{-1}_{q_0}T h_{q_0}$ also preserves $\simeq$. Let $C_{p_0} = \gen{\dotr, \Xn, Q_{C}, \pi_{C}, \lambda_{C}, p_0}$ be an initial  transducer such that $h_{p_0} = T$. Observe that $C_{p_0}$ is synchronizing and $\core(C_{p_0})$ is the single state identity transducer, let us denote this state by $\iota$. Let $B_{q_0^{-1}} = \gen{\dotr, \Xn, Q_B, \pi_B, \lambda_B, q_0^{-1}}$ be a minimal initial transducer representing the inverse of $A_{q_0}$. 

Observe that since $A_{q_0}$ is bi-synchronizing, that $B_{q_0}^{-1}$ is also bi-synchronizing. Let $k_1 \in \N$ be minimal so that $A_{q_0}$ and $B_{q_0}^{-1}$ are synchronizing at level $k_1$. let $k_2 \in \N$ be minimal so that for any word $\Gamma \in \Xnr^{k_2}$  we have $\pi_{C}(\Gamma, p_0) = \iota$. Let $k = \max\{k_1, k_2\}$. Let $j \in \N$ be the minimal number greater than $k$ such that for any word $\Delta \in \Xnr^{j}$  we have that $|\lambda_{B}(\Delta, q_0^{-1})|> 2k$. Therefore after processing any word of length $j$, the active state of the product transducer $B_{q_0}^{-1} \ast C_{p_0} A_{q_0}$ is of the for $(r, \iota, q )$ for some state $r$ of $\core(B_{q_0^{-1}})$ and some state $q$ of $A_{q_0}$. Notice that as $\iota$ is the single state identity transducer we may identify $(r, \iota, q)$ with the state $(r, q)$ of $B_{q_0^{-1}} \ast A_{q_0}$.

Now let $\Lambda \in \Xnr^{*}$ be a word such that $\pi_{B}(\Lambda, q_0^{-1}) = r$ and let $\Delta \in \Xn$ be such that $\lambda_{B}(\Delta, r)$ is greater than $k$ (recall that states of $\core(B_{q_0^{-1}}$ only process elements from $\Xn$). Let $q'$ be the state of  $A_{q_0}$ forced by  $\lambda_{B}(\Delta,r)$ and let $r'$  be the state $\pi_{B}(\Delta, r)$. Then observe that after processing the word $\Lambda\Delta$ the active state of $B_{q_0^{-1}} \ast A_{q_0}$ is $(r', q')$. However after processing $\Delta$ from the state $(r,q)$ the active state is also $(r', q')$. However since $B_{q_0^{-1}}\ast A_{q_0}$ is $\omega$-equivalent to the single state identity transducer, we see that $(r',q')$ which is a state of $\core(B_{q_0^{-1}}\ast A_{q_0})$ is $\omega$-equivalent to a map which produces some finite initial prefix and then acts as the identity.

From this we may conclude that there some $N \in \N$ such that $B_{q_{0}^{-1}} \ast C_{p_0} \ast A_{q_0}$ acts as the identity after processing a word of length $N$. This means that since $h_{q_0}^{-1} T h_{q_0}$ preserves $\simeq$ and is a homeomorphism that  $h_{q_0}^{-1} T h_{q_0} \in \Tnr$.
\end{proof}

Thus we have the following:

\begin{Theorem}
 The group $\aut{\Tnr}$ is isomorphic to the subgroup $\mathcal{T}\Bnr \le \Bnr$ consisting of those elements of $\Rnr$ which may be represented by finite, initial, bi-synchronizing transducers which preserve $\simeq$. 
\end{Theorem}

\begin{Remark}\label{Remark:elements of Tnr correspond to bisynch transducers with trivial cores}
Notice that since all elements of $\Tnr$ eventually act as the identity, then a minimal transducer $T_{q_0}$ representing an element $t \in \Tnr$  satisfies $\core(T_{q_0})$ is the single states identity transducer over  $\CCn$ which we denote $\id$. Furthermore any element of $\TBnr$ which may be represented by a bi-synchronizing transducer with trivial core, is an element of $\Tnr$. This is because $\Tnr \le \TBnr$ is uniquely defined by the fact that all elements act as the identity after modifying an  finite initial prefix.
\end{Remark}

\section{\texorpdfstring{$\out{\Tnr}$}{Lg}}\label{Section: out(Tnr)}

The paper \cite{BCMNO} introduces a group $\On$ which contains the group $\Onr\cong \out{G_{n,r}}$ for all valid $r$. In this section we shall demonstrate that there is a corresponding subgroup $\TOn \le \On$ containing the groups $\TOnr:= \out{\Tnr}$ for all valid $r$. 

Let $\widetilde{\Bnr}$ be the set of minimal transducer $A_{q_0}$ such that $h_{q_0} \in \Bnr$. The set $\widetilde{\Bnr}$ of transducers becomes a group isomorphic to $\Bnr$ with the product defined by for elements $A_{q_0}, B_{p_0}  \in   \widetilde{\Bnr}$, $AB_{(q_0,p_0)}$ is the minimal initial transducer representing the product $A_{q_0} \ast B_{p_0}$. Notice that since $h_{(q_0, p_0)} = h_{q_0} \circ h_{p_0}$ then we have that $AB_{(q_0,p_0)} \in \widetilde{\Bnr}$. To simplify the discussion below we shall identify $\Bnr$ with the group $\widetilde{\Bnr}$ of transducers. In particular we shall no longer distinguish between an element $h \in \Bnr$ and the minimal transducer $A_{q_0} \in \widetilde{\Bnr}$ representing $h$, thus we shall use the symbol $\Bnr$ for both $\widetilde{\Bnr}$ and $\Bnr$. Likewise we shall longer distinguish between elements of $\TBnr$ and the minimal initial transducers representing them.

\begin{lemma}
Let $A_{q_0}, B_{p_0} \in \TBnr$ such that $\core(A_{q_0}) = \core(B_{p_0})$. Then $A_{q_0}B_{p_0}^{-1} \in \Tnr$.
\end{lemma}
\begin{proof}
By Remark~\ref{Remark:elements of Tnr correspond to bisynch transducers with trivial cores} it suffices to show that $\core(A_{q_0}B_{p_0}^{-1})$ is equal to the single state identity transducer $\id$.
However it is a consequence of a result in \cite{BCMNO}, that $\core(A_{q_0}B_{p_0}^{-1})$ is trivial.
\end{proof}

As a corollary we have that an element of  $\out{\Tnr}$ corresponds to a subset of elements of $\TBnr$ which all have the same core. Therefore we may identify an element $[A_{q_0}] \in \out{\Tnr}$, for $A_{q_0}\in \TBnr$, with the element $\core(A_{q_0})$. 

Let $\Onr:= \{ \core(A_{q_0}) \mid A_{q_0} \in \Bnr \}$. This set inherits a product from $\Bnr$ as follows. For $g_1, g_2 \in \Onr$, let $A_{q_0}, B_{p_0} \in \Bnr$ be elements such that $\core(A_{q_0}) = g_1$ and $\core(B_{q_0}) = g_2$, and set $g_1g_2 := \core( AB_{(q_0, p_0)})$. Now observe that $\core( AB_{(q_0, p_0)})$ depends only on the states of $\core(g_1 \ast g_2)$ by a result in \cite{BCMNO} which states that $\core (A_{q_0}* B_{q_0}) = \core(\core(A_{q_0})* \core(B_{q_0}))$ and since removing states of incomplete response and identifying $\omega$-equivalent depends only on $\core(g_1 \ast g_2)$ (see \cite{GriNekSus}).  Therefore it follows that for any other element  $A'_{q'_0}$ and $B'_{p'_0}$ with cores $g_1$ and $g_2$ respectively we have $\core(A'B'_{(q'_{0}p'_{0})}) = \core(AB_{(q_0,p_0)})$. Thus setting the  product $g_1g_2 := \core( AB_{(q_0, p_0)})$ for elements $A_{q_0}, B_{q_0} \in \Bnr$ with cores $g_1$ and $g_2$ respectively, results in a well-defined product.

The authors of \cite{BCMNO} show that this product is equivalent to the following. Let $g, h \in \Onr$,  observe that  $g$ and $h$ are core non-initial transducers. Fix a state $q_1$ and $q_2$ of $g$ and $h$ respectively. Let  $gh := \core( (gh)_{(q_{1},q_{2})})$ where $(gh)_{(q_1,q_2)}$ is the minimal transducer representing the product ${g}_{q_1} \ast {h}_{q_2}$. The set $\Onr$ is isomorphic to $\out{G_{n,r}}$ under this product.

The inverse of an element of $\Onr$ can likewise be defined in two equivalent ways as demonstrated in \cite{BCMNO}. The first approach is as follows. Given an element $g \in \Onr$ let $A_{q_0} \in \Bnr$ be such that $\core(A_{q_0}) = g$. Let $B_{p_0}$ be the minimal transducer representing the inverse of $A_{q_0} $ then set $g^{-1} = \core(B_{q_0})$. It turns out (see \cite{BCMNO}) that this inverse is independent of the the choice of transducer $A_{q_0} \in \Bnr$ with core equal to $g$ and so is well-defined. The second approach makes use only of the states of $g$ to construct the inverse of $g$, thus removing the need of finding an element of $\Bnr$ with core equal to $g$.

Thus having removed the need to find elements of $\Bnr$ with cores equal to the relevant transducers in order to compute products and inverses, the paper \cite{BCMNO} introduces the group $\On:= \cup_{1 \le r < n} \Onr$ with product of two elements constructed as in the second approach, likewise for inverses. The groups $\Onr$ are moreover subgroups of $\On$. Another way of defining the group $\On$ (see \cite{BCMNO}) is as the group of all  bi-synchronizing transducers (note that as we have a method for computing inverses, our use of the word bi-synchronizing makes sense) all of whose states induce injective maps of $\CCn$ with clopen images. 

Let $\TOnr := \{\core(A_{q_0}) \mid A_{q_0} \in \TBnr \} \subset  \Onr$. Then $\TOnr$ is a subgroup of $ \Onr$, under the product inherited from $\On$, moreover $\TOnr \cong \out{\Tnr}$ under this product.

The following lemma essentially solves one half of the membership problem of $\TOnr$ in $\Onr$.

\begin{lemma}\label{Lemma: states of TBnr induces continuous functions from the interval to itsel }
Let $A_{q_0} \in \TBnr$, then for all states $q$ of $\core(A_{q_0})$ the maps $h_{q} : \CCn \to \CCn$ either all preserve or all reverse the lexicographic ordering of $\CCn$ and preserve the relation $\simeq_{{\bf{I}}}$ on $\CCn$.
\end{lemma}
\begin{proof}
We consider only elements $A_{q_0} \in \TBnr$ such that $h_{q_0}$ induces an orientation preserving homeomorphism of $S_{r}$. The proof is similar in the other case.

Observe that since $A_{q_0} \in \TBnr$ this means that $h_{q_0}: \CCnr \to \CCnr$ is a homeomorphism which preserve $\simeq$. Hence there is a $\nu \in \Xnrs$ such that $h_{q_0}$  preserve the lexicographic ordering  on $U_{\nu}$. By restricting to a small enough open set contained in $\nu$ we may assume that assume that $\nu$ is longer than the synchronizing level of $A_{q_0}$. Let $\gamma \in \Xns$ be longer than the synchronizing level of $A_{q_0}$ such that the state of $A_{q_0}$ forced by $\gamma$ is a state $p \in \core(A_{q_0})$. Let $q = \pi_{A}(\nu, q_0) \in \core(A_{q_0})$. Observe that as $h_{q_0}$ preserves the lexicographic ordering of $U_{\nu}$, it must be the case that $h_{q}: \CCn \to \CCn$ preserves the lexicographic ordering on $U_{\nu}$. Thus $h_{p}$ must also preserve the lexicographic ordering on $\CCn$ otherwise $h_{q}$ does not preserve the lexicographic ordering. Moreover since $h_{q_0}$ preserves $\simeq$ we must also have that $h_{p}$ preserves $\simeq_{{\bf{I}}}$ otherwise $h_{q_0}$ would not preserve $\simeq$. Now since $A_{q_0}$ is synchronizing for all states $p'$ of $\core(A_{q_0})$ there is a  word $\gamma'$ in $\Xns$ longer than the synchronizing length of $A_{q_0}$ such that the state of $A_{q_0}$ forced by $\gamma'$ is $p'$. This concludes the proof.
\end{proof}

This prompts the following definition:

\begin{Definition}\label{Def: noninitial transducers preserving relations}
Let $A$ be a non-initial transducer of over $\CCn$, and let $P$ be some relation on $\CCn$. Then we say that \emph{$A$ preserves $P$} if for all states $q$ of $A$ the map $h_{q}$ preserves the relation $P$.
\end{Definition}

\begin{lemma}\label{Lemma: images of states of core of elements of TBnr are connected}
Let $A_{q_0} \in \TBnr$ then for all states $q \in \core(A_{q_0})$ there is a finite antichain  $\ac{u} = \{u_1, u_2, \ldots, u_{m}\}$ such that $\im(q) =  \cup_{1 \le i \le m}U_{u_i}$. Moreover for $1 \le i \le m-1$ we have that $u_in-1 n-1\ldots \simeqI u_{i+1}00\ldots$.
\end{lemma}
\begin{proof}
Observe that by  Lemma \ref{Lemma: states of TBnr induces continuous functions from the interval to itsel } it follows that for every state $q \in \core(A_{q_0}$ $h_{q}$ induces a continuous function from the interval to itself. However since the interval is connected, and the continuous image of a connected topological space is connected, it follows that $\im(q)/\simeqI$ is a connected. The lemma is now an easy consequence of this.
\end{proof}

 The result below was demonstrated first by Brin in his seminal paper \cite{MBrin2} describing the automorphisms of R. Thompson's groups $F$ and $T$ and we have simple translated it into the language of transducers. 

\begin{Theorem}
Let $n = 2$, and let $A_{q_0} \in \TBnr$, then  $\core(A_{q_0})$ is either the single state identity transducer or the single state transducer mapping induce the transposition $(0 \ 1)$ on $X_{2}$. More specifically we have $\TOn \cong C_2$.
\end{Theorem} 

In order to show the subgroup $\TOnr \le \Onr$ is completely characterised by the fact that all of its states  preserve the relation $\simeqI$ and either all preserve or  all reverse the lexicographic ordering of $\CCn$ we need the following notions from  \cite{BCMNO}. Recall that $\On$ is the group of core, bi-synchronizing transducers all of whose states induce injective maps of $\CCn$ with clopen images. 

\begin{Definition}[\cite{BCMNO}]\label{Def: Viable combination}
Let $g \in \On$. A \emph{viable combination} $\mathfrak{v}_{j}$ for $g$ is a pair of tuples $$((\rho_1, \ldots, \rho_j), (p_1, \ldots, p_j))$$ satisfying 
\begin{enumerate}
\item $\bigcup_{1\le i \le j} \rho_i\im(p_i) = \CCn$ \label{Def: Viable combination union fills cantor space}
\item  for $1 \le i < l \le j$ we have $\rho_i \im(p_i) \cap \rho_l \im(p_l) = \emptyset$ \label{Def: Viable combination unions are disjoint}.
\end{enumerate}
 where the $\rho_{i} \in \Xns$, $1 \le i \le j$ are not necessarily incomparable or distinct, and  $p_i$, $1 \le i \le j$ are not necessarily distinct states of $g$.
\end{Definition} 

\begin{Definition}[Single expansions of viable combinations \cite{BCMNO}] \label{Def: expansions of viable combinations}
Let $T = (X_n, Q, \pi, \lambda) \in \On$ and let 
$\mathfrak{v}_{j} = ((\rho_1, \ldots, \rho_j), (p_1, \ldots, 
p_j))$ be  a viable combination for $T$. Fix an $i$ such that $1 
\le i \le j$, and let $\rho_{i,l} := \rho_i\lambda(l, 
p_i)$ and $p_{i,l} := \pi(l, p_i)$ for $l \in X_n$, then 
$$\mathfrak{v}_{j+n-1} := ((\rho_1, \ldots, \rho_{i-1}, 
\rho_{i,0}, \ldots,\rho_{i,(n-1)}, \rho_{i+1}, \ldots, \rho_{j} 
), (p_1, \ldots, p_{i-1}, p_{i,0}, \ldots, p_{i, (n-1)}, 
p_{i+1},\ldots, p_{i,k}))$$ is called a \emph{single expansion of 
$\mathfrak{v}_{j}$}. 
\end{Definition}

\begin{Remark}\label{Remark: sequence of expansions of viable combinations give rise to viable combinations}
Clearly a single expansion of a viable combination for an element $T \in \On$ results in a new viable combination for $T$. Therefore a sequence of single expansions applied to a viable combination for $T$ also results in a new viable combination for $T$. See \cite{BCMNO} for more detail.
\end{Remark}

The following lemma is proved in \cite{BCMNO}:

\begin{lemma}\label{Lemma: g is in Onr if and only if there is a viable combination with the right modulo arithmetic properties}
Let $g \in \On$ and let $\mathfrak{v}_{g}$ denote the set of all viable combinations of $g$. Let $1 \le r \le n-1$. There is a transducer $A_{q_0} \in \Bnr$ whose core is equal to $g$ if and only if there is a sequence $\mathfrak{v}_{j_1}, \mathfrak{v}_{j_2} \ldots, \mathfrak{v}_{j_m}$ of elements of $\mathfrak{v}_{g}$ such that $r \equiv \sum_{1 \le i \le m} j_i \equiv m \mod{n-1}$.
\end{lemma}

Now we have the following lemma which solves the other half of the membership problem of $\TOnr \in \Onr$. However we make the following definition first.

\begin{Definition}\label{Def: lexicographic viable combinations}
Let $g \in \On$. A \emph{lexicographic viable combination} $\mathfrak{v}_{j}$ for $g$ is a viable combination $$\mathfrak{v}_{j}:=((\rho_1, \ldots, \rho_j), (p_1, \ldots, p_j))$$ satisfying: for $1 \le a < b \le j$  $\rho_{a}\im(p_a) \lelex \rho_{b}\im(p_b)$.
\end{Definition}

\begin{lemma}\label{Lemma: g in Onr induces map on the line iff it has lexicographic viable combination}
Let $g \in \Onr$ be such that $g$ preserves  $\simeqI$ and preserves (reverses) the lexicographic ordering on $\CCn$ then any viable combination of   $g$ may be re-ordered to get a lexicographic viable combination.
\end{lemma}
\begin{proof}

Since $g \in \Onr$ there is an element $A_{q_0} \in \TBnr$ such that $\core(A_{q_0}) = g$. Now by Lemma~\ref{Lemma: g is in Onr if and only if there is a viable combination with the right modulo arithmetic properties} it follows that there are viable combinations  $\mathfrak{v}_{j_1}, \mathfrak{v}_{j_2}, \ldots, \mathfrak{v}_{j_{m}}$ of $g$. Thus the set  $\mathfrak{v}_{g}$ of viable combinations of $g$ is non-empty. Let $\mathfrak{v}_{l} \in \mathfrak{v}_{g}$. We shall now make use of Lemma~\ref{Lemma: images of states of core of elements of TBnr are connected} to re-order $\mathfrak{v}_{l}$ to obtain a  lexicographic viable combination of $g$. 

Suppose $\mathfrak{v}_{l} = ((\rho_1, \ldots, \rho_l),(p_1, \ldots, p_l))$. Let $R_{\mathfrak{v}_{l}} := \{ \rho_k \mid 1 \le k \le l \mbox{ such that }\rho_{t} \not\leq \rho_{k} \forall  1\le t \le l \}$. Observe that $R_{\mathfrak{v}_{l}}$ is an antichain by construction, thus we may assume that the set $R_{\mathfrak{v}_{l}}$ is totally ordered in the lexicographic order. Fix $\rho_{k} \in  R_{\mathfrak{v}_{l}}$. Let $D(\rho_{k}) = \{ \rho_{t} \mid \rho_{k} \le \rho_{t} \}$. We assume that $D(\rho_{k})$ is ordered according to the short-lex ordering. Observe that $\rho_{k} \in D(\rho_{k})$ and in particular is the smallest element  of $\rho_{k}$. Let $d_{\rho_{k}} = | D(\rho_{k})|$. Now for each $\rho_{t} \in D(\rho_{k})$ let $\mu_{\rho_{t}}$ be the number of times it occurs in the tuple $(\rho_1, \ldots, \rho_k)$.  Let  $\bar{l}:= |R_{\mathfrak{v}_{l}}|$. 

Given a tuple $(y_1, y_2, \ldots, y_t)$ we shall use the notation $(y_1, \ldots, y_s^{p}, \ldots, y_t)$ to mean that $y_s, \ldots, y_{s+p-1}$ are all equal to $y_s$. Now let  $\xi_1, \ldots \xi_{\bar{l}}$ be the elements of $R_{\mathfrak{v}_{l}}$ in lexicographic ordering. For $1 \le s \le \bar{l}$ let $\xi_{(s,1)}, \ldots \xi_{(s, d_{\xi_{s}})}$ be the elements of $D(\xi_{s})$ in short-lex ordering. Then we have 
\begin{equation}\label{Equation: first rearrangement of tuple}
\left(\xi_1^{\mu_{\xi_1}}, \xi_{(1,1)}^{\mu_{\xi_{(1,1)}}}, \ldots, \xi_{(1, d_{\xi_1})}^{\mu_{\xi_{(1, d_{\xi_1})}}}, \ldots, \xi_{\bar{l}}^{\mu_{\xi_{\bar{l}}}}, \xi_{(\bar{l},1)}^{\mu_{\xi_{(\bar{l},1)}}}, \ldots, \xi_{(\bar{l},d_{\xi_{\bar{l}}})}^{\mu_{\xi_{(\bar{l},d_{\xi_{\bar{l}}})}}}\right)
\end{equation}
is a reordering of $(\rho_1, \ldots, \rho_l)$. Let $(q_1, \ldots, q_l)$ be the induced reordering on the tuple of states $(p_1,\ldots, p_l)$.

Now fix $1 \le s \le \bar{l}$ and consider the sequence 
$\xi_s^{\mu_{\xi_s}}, \xi_{(s,1)}^{\mu_{\xi_{(s,1)}}}, \ldots, 
\xi_{(s, d_{\xi_1})}^{\mu_{\xi_{(s, d_{\xi_s})}}}$  let 
$$q_{(s_0,1)}, \ldots, q_{(s_0, \mu_{\xi_s})}, \ldots, 
q_{(s_{d_{\xi_s}},1)} ,\ldots , q_{(s_{d_{\xi_s}}, \mu_{\xi_{(s, 
d_{\xi_{s}})}})}$$ be the corresponding states. Then observe that $$ 
\xi_{s}\im(q_{(s_0,1)} \cup \ldots \cup \xi_s \im(q_{(s_0, \mu_{\xi_s})}) \cup \bigcup_{1 \le k \le 
d_{\xi_s}} \xi_{(s,k)} \im(q_{(s_{k},1)}) \cup \ldots \cup \xi_{(s,k)} \im(q_{(s_{k}, \mu_{\xi_{(s, k)})}})  =  
U_{\xi_{1}}$$ by construction and the definition of viable 
combinations. Now let  $0 \le k < d_{\xi_s}$ and let $k \le k' \le  d_{\xi_s}$, for $1\le t_1 \le \mu_{\xi_{(s,l)}}$ and $1 \le t_2 \le \mu_{\xi_{(s, k')}}$ consider $\xi_{(s,k)} \im(q_{(s_{k}, t_1)})$  and $\xi_{(s,k')} \im(q_{(s_{k'},t_2)})$ (note that we set $\xi_{(s,0)} \im(q_{(s_{0},t_1)}):= \xi_{s}\im(q_{s_{0}, t_1})$). Observe that by the ordering of the set $D(\xi_{s})$, by Lemma~\ref{Lemma: images of states of core of elements of TBnr are connected} and by  Definition~\ref{Def: Viable combination} part~\ref{Def: Viable combination unions are disjoint} of viable combinations, we must have that exactly one of the following holds: 
\begin{enumerate}[label = (\roman*)]
\item $\xi_{(s,k)} \im(q_{(s_{k}, t_1)}) \lelex \xi_{(s,k')} \im(q_{(s_{k'},t_2)})$ or,
\item $\xi_{(s,k')} \im(q_{(s_{k'},t_2)}) \lelex \xi_{(s,k)} \im(q_{(s_{k}, t_1)})$ or,
\item $k = k'$ and $t_1 = t_2$.
\end{enumerate}

Thus we may reorder  the tuple \eqref{Equation: first rearrangement of tuple} to obtain a new tuple

\begin{equation}\label{Equation: second rearrangement of tuple}
\left(\chi_1^{\mu_{\chi_1}}, \chi_{(1,1)}^{\mu_{\chi_{(1,1)}}}, \ldots, \chi_{(1, d_{\xi_1})}^{\mu_{\chi_{(1, d_{\xi_1})}}}, \ldots, \chi_{\bar{l}}^{\mu_{\chi_{\bar{l}}}}, \chi_{(\bar{l},1)}^{\mu_{\chi_{(\bar{l},1)}}}, \ldots, \chi_{(\bar{l},d_{\xi_{\bar{l}}})}^{\mu_{\chi_{(\bar{l},d_{\chi_{\bar{l}}})}}}\right)
\end{equation} 

satisfying the following conditions. Let $1 \le s \le  \bar{l}$ and consider the subsequence  $$\chi_s^{\mu_{\chi_s}}, \chi_{(s,1)}^{\mu_{\chi_{(s,1)}}}, \ldots, 
\chi_{(s, d_{\xi_s})}^{\mu_{\chi_{(s, d_{\xi_s})}}}$$ then:
\begin{enumerate}[label = (\roman*)]
\item  Each $\chi_{(s, r)}^{\mu_{\chi_{(s,r)}}}$ for $0 \le r \le d_{\xi_{s}}$ (where $\chi_{(s,0)} = \chi_{s}$) corresponds to precisely one element  $\xi_{(s,r')}^{\mu_{\xi_{(s,r')}}}$ for $0 \le r' \le d_{\xi_{s}}$ (where $\xi_{(s,0)} = \xi_{s}$). \label{List: rearrangement 2 cond 1}
\item The subsequence $\chi_s^{\mu_{\chi_s}}, \chi_{(s,1)}^{\mu_{\chi_{(s,1)}}}, \ldots, 
\chi_{(s, d_{\xi_s})}^{\mu_{\chi_{(s, d_{\xi_s})}}}$   is ordered so that  the following is true. If $q_{(s_0,1)}, \ldots, q_{(s_0, \mu_{\chi_s})}, \ldots, 
q_{(s_{d_{\xi_s}},1)} ,\ldots , q_{(s_{d_{\xi_s}}, \mu_{\chi_{(s, 
d_{\xi_{s}})}})}$ are the set of states corresponding to the sub-sequence $\chi_s^{\mu_{\chi_s}}, \chi_{(s,1)}^{\mu_{\chi_{(s,1)}}}, \ldots, 
\chi_{(s, d_{\xi_1})}^{\mu_{\chi_{(s, d_{\xi_s})}}}$ then for $0 \le k < d_{\xi_{s}}$, for $k \le k' \le d_{\xi_{s}}$, for $1 \le t_1 \le \mu_{\chi_{(s,k)}}$ and $1 \le t_2 \le \mu_{\chi_{(s,k')}}$ such that if $k = k'$ and $\mu_{\chi_{(s,k)}} >1$ then $t_1 \ne t_2$, we have  $\chi_{(s,k)} \im(q_{(s_{k}, t_1)}) \lelex \chi_{(s,k')} \im(q_{(s_{k'},t_2)})$.  
\end{enumerate}
 
 Notice condition~\ref{List: rearrangement 2 cond 1} above means all elements of the subsequence $\chi_s^{\mu_{\chi_s}}, \chi_{(s,1)}^{\mu_{\chi_{(s,1)}}}, \ldots, 
 \chi_{(s, d_{\xi_s})}^{\mu_{\chi_{(s, d_{\xi_s})}}}$ have prefix $\xi_{s}$.    

Now observe that for $1 \le s < s' \le \bar{l}$, for $ 0 \le k \le d_{\xi_{s}}$, for  $0 \le k' \le d_{\xi_{s'}}$, for $1 \le t \le \mu_{\chi_{(s, k)}}$ and for $1 \le t' \le \mu_{\chi_{(s', k')}}$ so that if  the $t$\textsuperscript{th} copy of $\chi_{(s, k)}$ ($\chi_{s}$ if $k=0$) and the $t'$\textsuperscript{th} copy of $\chi_{(s', k')}$ ($\chi_{s'}$ if $k'=0$) in the sequence \eqref{Equation: second rearrangement of tuple} correspond to the states $q_{(s_{k}, t_1)}$ and $q_{(s'_{k}, t_2)}$ respectively, then $\chi_{(s,k)} \im(q_{(s_{k}, t_1)}) \lelex \chi_{(s',k')} \im(q_{(s'_{k'},t_2)})$ in the lexicographic ordering on $\CCn$. This follows from the observation in the previous paragraph and because $\xi_{s} \lelex \xi_{s'}$.

Let $$(\zeta_1, \ldots, \zeta_{l}) := \left(\chi_1^{\mu_{\chi_1}}, \chi_{(1,1)}^{\mu_{\chi_{(1,1)}}}, \ldots, \chi_{(1, d_{\xi_1})}^{\mu_{\chi_{(1, d_{\xi_1})}}}, \ldots, \chi_{\bar{l}}^{\mu_{\chi_{\bar{l}}}}, \chi_{(\bar{l},1)}^{\mu_{\chi_{(\bar{l},1)}}}, \ldots, \chi_{(\bar{l},d_{\xi_{\bar{l}}})}^{\mu_{\chi_{(\bar{l},d_{\chi_{\bar{l}}})}}}\right)$$ and let 
\begin{IEEEeqnarray*}{rCl}
(q_1', \ldots q_l') := (& q_{(1_0,1)}& , \ldots, q_{(1_0, \mu_{\chi_{1}})}, \ldots, q_{(1_{d_{\xi_{1}}},1)},\ldots, q_{(1_{d_{\xi_{1}}}, \mu_{\chi_{(1, d_{\xi_1})}})}, \ldots,  \\
  & q_{(\bar{l}_0,1)},&   \ldots, q_{(\bar{l}_0, 
 \mu_{\chi_{\bar{l}}})}, \ldots, 
 q_{(\bar{l}_{d_{\xi_{1}}},1)},\ldots, 
 q_{(\bar{l}_{d_{\xi_{\bar{l}}}}, \mu_{\chi_{(\bar{l}, 
 d_{\xi_{\bar{l}}})}})})
\end{IEEEeqnarray*}

Observe that  $((\zeta_1, \ldots, \zeta_l), (q_1', \ldots, q_l'))$ is a viable combination for $g$. Moreover, by discussion above, for $1 \le a < b \le l$ it satisfies   $\zeta_{a} \im(q_a) \lelex \zeta_{b}\im(q_b)$ in the lexicographic ordering of $\CCn$ i.e  $((\zeta_1, \ldots, \zeta_l), (q_1', \ldots, q_l'))$ is a lexicographic  viable combination of $g$.  
\end{proof}

Let $ \pi_{R,n} \in \On$ be the single state transducer which induces the permutation $i \mapsto n-1- i$ on $\Xn$. Then  $\pi_{R,n}$ has order $2$, preserves $\simeqI$, and reverses the lexicographic ordering on  $\CCn$. It is also not hard to see that $\pi_{R,n} \in \TOnr$ for all $1 \le r < n-1$. For instance the map $ f_{\pi_{R,n}}: \CCnr \to \CCnr$ given by $ \dot{a} \delta \mapsto \dot{r-a +1} (\delta)h_{\pi_{R,n}}$ is induces by an element of $\TBnr$ with core equal to $\pi_{R,n}$ since it replaces some finite prefix before acting by $\pi_{R,n}$.

The following lemma follows immediately from the fact that $\pi_{R,n}$ has order $2$ and reverses the lexicographic ordering on $\CCn$.

\begin{lemma}\label{Lemma: multiplication by piR bijection from orientation preserving to reversing and vice versa}
Multiplication by $\pi_{R,n}$ induces a bijection from the subset of $\Onr$ preserving (reversing) the lexicographic ordering, to the subset of $\Onr$ reversing (preserving) the lexicographic ordering.
\end{lemma}

\begin{lemma}\label{Lemma: g in Onr is in TOnr if g prserve  simeqI and lex order}
Let $g \in \Onr$ be such that $g$ preserves $\simeqI$ and preserves (reverses) the lexicographic ordering on $\CCn$ then $g \in \TOnr$.
\end{lemma}
\begin{proof}
By Lemma~\ref{Lemma: multiplication by piR bijection from orientation preserving to reversing and vice versa} it suffices to prove this for elements $g \in \Onr$ which preserves $\simeqI$ and the lexicographic ordering on $\CCn$. Since if $g \pi \in \TOnr$, then as $\TOnr$ is a group containing $\pi_{R,n}$ by an observation above, then $g \pi_{R,n} \pi_{R,n} = g \in \TOnr$ also. Thus fix $g \in \Onr$ an element which preserves $\simeqI$ and the lexicographic ordering on $\CCn$.

Since $g \in \Onr$ by Lemma~\ref{Lemma: g is in Onr if and only if there is a viable combination with the right modulo arithmetic properties} there are viable combinations $\viable{j_1}, \viable{j_2},\ldots,\viable{j_m}$ of $g$ such that $r \equiv \sum_{1 \le i \le m} j_{i} \equiv m \mod{n-1}$. Let $j= \sum_{1 \le i \le m} j_i$. By Remark \ref{Remark: sequence of expansions of viable combinations give rise to viable combinations} we may assume that $j_i > 1$ for $1 \le i \le m$. Now by Lemma~\ref{Lemma: g in Onr induces map on the line iff it has lexicographic viable combination}  we may assume that all the  $\viable{j_{i}}$ for $1\le i \le m$ are lexicographic viable combinations for $g$.

Let $1 \le i \le m$, and consider the viable combination $\viable{j_i}$ of $g$. Suppose $\viable{j_{i}} = ((\rho_1, \ldots, \rho_{j_i}), (p_1, \ldots, p_{j_i}))$. Let $\ac{u}= \{ u_1, u_2, \ldots, u_{j_i}\}$ be an antichain of $\Xnrs$  of length $j_i$, and let $w \in \Xnrp$. Recall that all antichains are assumed to be ordered in the lexicographic ordering. Now if $g$  preserves the lexicographic ordering of $\CCn$ then define a map $f_{\mathfrak{v}_{j_i},w}: U_{u_1} \sqcup \ldots \sqcup U_{u_{j_i}} \to  U_{w}$ by  $u_a \delta \mapsto w\rho_a (\delta)h_{p_{a}}$ for $1 \le a \le j_i$ and $\delta \in \CCn$. O Then $f_{\mathfrak{v}_{j_i},w}$ is homeomorphism from its domain to its range (by the definition of a viable combination), moreover  $f_{\mathfrak{v}_{j_i},w}$ preserves the lexicographic ordering on $U_{u_1} \sqcup \ldots \sqcup U_{u_l}$ since $\viable{j_i}$ is a lexicographic viable combination and for all states $q$ of $g$ $h_{q}$ preserves the lexicographic ordering on $\CCn$ and the relation $\simeqI$  (Lemma~\ref{Lemma: states of TBnr induces continuous functions from the interval to itsel }). 

Now let $\ac{v}$ be a complete antichain of $\Xnrs$ of length $j$ and let $\ac{w}$ be a complete antichain of $\Xnrs$ of length $m$. These antichains exist since $j \equiv m \equiv r \mod{n-1}$. Recall that as $ 1 \le r < n-1$ we must have that $j$ and $m$ are both non-zero. Let $\ac{v}_1, \ldots, \ac{v}_m$ be disjoint subsets of $\ac{v}$ such that each $\ac{v}_i$, $1 \le i \le m$ is an antichain (ordered lexicographically) of length $j_i$ and for $1 \le a < b \le m$ we have $\ac{v}_{a} \lelex \ac{v}_{b}$. Note that the stipulation that $|\ac{v}_{i}| = j_i$ for $1 \le i \le m$ means that $\sqcup_{1 \le i \le m} \ac{v}_{i} = \ac{v}$. Suppose that $\ac{w} = \{ w_1, w_2, \ldots, w_m\}$. 

Recall that for a subset  $Z \subseteq \Xns \sqcup \Xnrs$ we  denote by $U(Z)$ the set $\{ U_{z} \mid z \in Z \}$. Let $f: \CCnr \to \CCnr$ be defined such that $f \restriction _{U(\ac{v}_{i})} = f_{\ac{v}_i,w_i}$ for $1 \le i \le m$. Then clearly $f$ is a homeomorphism, and since each $f_{\ac{v}_i,w_i}$, $1 \le i \le m$, preserves the lexicographic ordering of  $U(\ac{v_i})$ and $\simeqI$ then, $f$ preserves the lexicographic ordering and the relation $\simeqI$ on $\CCnr$. Moreover since $f$ replaces some initial prefix before acting as a state of  $g$, it follows that there is a transducer $A_{q_0} \in \TBnr$ with $h_{q_0} = f$.

\end{proof}

\begin{Remark}\label{Remark: direct proof that g in Onr preseving simeq and reversing lex is in TOnr}
One may also prove Lemma~\ref{Lemma: g in Onr is in TOnr if g prserve  simeqI and lex order} above for $g \in \Onr$ which preserves $\simeqI$ and reverses the lexicographic ordering by directly constructing a transducer $A_{q_0} \in \TBnr$ which induce an orientation reversing homeomorphism of the line. As in the proof above, we make use of the lexicographic viable combinations to induce maps on $\CCnr$ which preserve $\simeqI$ and reverse the lexicographic ordering.
\end{Remark}

\begin{lemma}\label{Lemma:preservinglexsuffices}
Let $g \in \Onr$ be such that $g$ preserves or reverses the lexicographic ordering, then $g$ also preserves $\simeqI$.
\end{lemma}
\begin{proof}
Fix $q$ any state of $g$. We observe that, by definition $q$ induces a continuous injection function from $\CCn$ to itself with clopen image. Let $x, y \in \CCn$ be such that  $x \ne y$, $x \lelex y$, $x \simeqI y$ and there is a word $\nu \in \Xnrp$ such that $(x)h, (y)h \in  U_{\nu}$. As there is no point $y' \in \CCn$  not equal to $x$ or $y$ satisfying $x \lelex y' \lelex y$, then it must be the case that $(x)h \simeq (y')h$.  Now as $(U_{\nu})h_{q} ^{-1}$ is open, there is a $\mu \in \Xnrp$ such that $(U_{\mu})h_{q} = U_{\nu}$ and so $h_{q}$ preserves the relation $\simeqI$ on $U_{\mu}$. Finally, observe that as $g$ is synchronizing for any other state $p$ of $g$, there is a word $\Gamma \in \Xnrp$ such that $\pi_{g}(\Gamma, q) = p$. Therefore, we deduce that  all states of $g$ preserve $\simeqI$ as required.
\end{proof}

Putting together Lemmas~\ref{Lemma: g is in Onr if and only if there is a viable combination with the right modulo arithmetic properties}, \ref{Lemma: g in Onr induces map on the line iff it has lexicographic viable combination}, Lemma~\ref{Lemma:preservinglexsuffices}, and \ref{Lemma: g in Onr is in TOnr if g prserve  simeqI and lex order}, we obtain the following result.

\begin{Theorem}\label{Thm: Equivalent conditions for an element of $On$ to belong to TOnr}
Let $g \in \On$. The following are equivalent:
\begin{enumerate}[label =(\alph*)]
\item $g \in \TOnr$
\item  $g$ preserves or reverses the lexicographic ordering, and there are lexicographic viable combinations $\viable{{j_1}}, \ldots, \viable{{j_m}}$ such that $r \equiv \sum_{1 \le i \le m} j_{i} \equiv m \mod{n-1}$.
\item  $g$ preserves or reverses the lexicographic ordering, and there are viable combinations $\viable{{j_1}}, \ldots, \viable{{j_m}}$ such that $r \equiv \sum_{1 \le i \le m} j_{i} \equiv m \mod{n-1}$.
\item $g \in \Onr$ and preserves or reverses the lexicographic ordering  on $\CCn$.
\end{enumerate}

\end{Theorem}

\section{ The enveloping group $\TOn$}\label{Section:nestingproperties1}

We make the following definition.

\begin{Definition}\label{Def: External description of TOn.}
Let $\TOn \subset \On$ consists of those element which reverse or preserve the lexicographic ordering on $\CCn$.
\end{Definition}

\begin{Remark}\label{Remark: TOn is the union of the TOnr's}By an observation in \cite{BCMNO}, for an element $g \in \On$ the set $\viable{g}$ of viable combinations of $g$ is non-empty. From this it follows that $\TOn = \cup_{1\le  r \le n-1}  \TOnr$.
\end{Remark}

\begin{proposition} \label{Proposition:TOn is a subgroup of On}
The set $\TOn$ is a subgroup of $\On$.
\end{proposition}
\begin{proof}
It suffices to show that $\TOn$ is closed under inverses and products. It is clear that the single state identity transducer is an element of $\TOn$.

Let $g \in \TOn$. Then there is an $r \in \{1,2, \ldots, n-1\}$ such that there is an $A_{q_0} \in \TBnr$ with $\core(A_{q_0}) = g$. Let $B_{p_0}$ be the minimal transducer representing the inverse of $A_{q_0}$, then  $B_{p_0} \in \TBnr$ and $g^{-1} = \core(B_{p_0}) \in \TOnr \subset \TOn$.

Now let $h \in \TOn$. Let $q$ be a state of $g$ and $p$ be a state of $h$. Let $gh_{(p,q)}$ be the minimal transducer representing the product $g_{p} \ast h_{q}$. Then observe that  $gh_{(p,q)}$ induces the function $g \circ h : \CCn \to \CCn$. Thus since $g$ and $h$ either preserve or reverse the lexicographic ordering on $\CCn$ it follows that the states of $gh_{(p,q)}$ either all reverse or all preserve the lexicographic ordering on $\CCn$. Therefore $gh = \core(gh_{(p,q)}) \in \TOn$.
\end{proof}

\begin{Definition}
Let $\widetilde{\TOn}$ be the subset of $\TOn$ consisting of all those elements which preserve the lexicographic ordering on $\CCn$. Set $\WTOnr:=  \WTOn \cap  \TOnr$. Then we call elements of $\WTOn$ \emph{orientation preserving}, and elements of $\TOn \backslash \WTOn$ \emph{orientation reversing}. Likewise set $\widetilde{\TBnr}$ to be those elements of $\TBnr$ with core in $\widetilde{\TOn}$, then $\widetilde{\TBnr}$ are precisely those elements of $\TBnr$ which induce orientation preserving maps of $S_{r}$. We will also call elements of $\widetilde{\TBnr}$ \emph{orientation preserving} and the elements of $\TBnr \backslash \widetilde{\TBnr}$ \emph{orientation reversing}.
\end{Definition}

The following proposition is straightforward and so we omit its proof.

\begin{proposition}
The set $\WTOn$ is an index 2 subgroup of $\TOn$ and so a normal subgroup of $\TOn$. The set $\WTOnr$ is an index 2 subgroup of $\TOnr$ and so a normal subgroup of $\TOnr$. 
\end{proposition}

We now investigate how the groups $\TOn$ intersect each other. For this we require the following definition, which in fact applies to the elements of the group $\On$.

\begin{Definition}[Signature]\label{Definition:signature}
Let $T \in \On$, for each state $q \in Q_{T}$ let $m_q$ be the size of the smallest subset $V$ of $\Xns$ such that $U(V) = \{ U_v \mid v \in V\}$  is a clopen cover of $\im(q)$ and $U_{v} \subset \im(q)$ for all $v \in V$. Let $k \in \N$ be the minimal synchronizing level of $T$ an order the elements of the set $\Xn^{k}$ lexicographically as follows: $x_1 < x_2 < \ldots < x_{n^{k}}$. Let $(q_{x_1}, q_{x_2}, \ldots, q_{x_{n^{k}}}) \in Q_{T}^{n^{k}}$ be such that, for all $1 \le i \le n^{k}$, $q_{x_i}$ is the unique state of $T$ forced by $x_i$. Set $(T)\sig = \sum_{1 \le i \le n^{k}} m_{q_{x_i}}$, we call $(T)\sig$ the \emph{signature} of $T$; set $(T)\rsig = (T)\sig \mod{n-1}$, we call $(T)\rsig$ the \emph{reduced signature of $T$}.
\end{Definition}

We have the following proposition. We prove the proposition below for the group $\TOn$ noting that a similar result holds, with almost identical proof, in the  group $\On$.

\begin{proposition}\label{Proposition:sigdeterminesmembership}
Let $T \in \TOn [T \in \On]$, and $1 \le r < n$, then $T \in \TOnr [T \in \Onr]$ if and only if $r (T)\sig \equiv r \mod{n-1}$ (equivalently $r((T)\sig-1) \equiv 0 \mod {n-1}$.
\end{proposition}
\begin{proof}
We begin with the forward implication. First suppose that $T \in \TOnr$ and $T$ has minimal synchronizing level $k$. Since $T \in \TOnr$, there is an element $A_{q_0} \in \TBnr$ with $\core(A_{q_0}) = T$. Let $j \in \N$ be minimal such that after reading a word of length $j$ from the state $q_0$ of $A$, the resulting state is a state of $T$. Let $\{\mu_i \mid 1 \le i \le rn^{k+j-1}\}$ be the set of all words of length $j+k$ in $\Xnrs$ ordered lexicographically. For $1 \le i \le rn^{k+j-1}$, set  $\nu_i  = \lambda_{A}(\mu_i, q_0)$ and  $q_{\mu_i}$ to be the state of $T$ forced by $\mu_i$. Observe that the state  $q_{\mu_i}$ depends only on the last $k$ letters of $\mu_i$, hence if the elements of $\Xn^{k}$ are ordered lexicographically as  $x_1 < x_2< \ldots < x_{n^{k}}$, the sequence ${(q_{\mu_i})}_{1 \le i \le rn^{j+k-1}}$, where $q_{x_i}$, $1 \le i \le n^{k}$, is the state of $T$ forced by $x_i$, is precisely the sequence $q_{x_1}, \ldots, q_{x_{n^{k}}}$ repeated $rn^{j-1}$ times.  Since $\bigcup_{1 \le i \le rn^{j+k-1}} \nu_i \im(q_{\mu_i}) = \CCnr$, it must therefore be the case that $rn^{j-1}(\sum_{1\le i \le n^{k}} m_{q_{x_i}}) \equiv r \mod n-1$. This is because if, for each $q_{\mu_i}$, $1 \le i \le rn^{k+j-1}$, $V_{q_{\mu_i}}$ is the smallest subset of $\Xns$ such that $U(V_{q_{\mu_i}})$  is a cover of $\im(q)$ and $U_{v} \subset \im(q)$ for all $v \in V_{q_{\mu_i}}$, then $ \bigcup_{1 \le i \le rn^{j+k-1}}\{ \nu_i v \mid v \in V(q_{\mu_i}) \}$ must be a complete antichain of $\CCnr$ (otherwise $A_{q_0}$ is not a homeomorphism). Therefore, setting $m_{q_{x_i}} =  |V_{q_{x_i}}|$ we have, $$\left| \bigcup_{1 \le i \le rn^{j+k-1}}\{ \nu_i v \mid v \in V(q_{\mu_i}) \}\right |  = rn^{j-1}(\sum_{1\le i \le n^{k}} m_{q_{x_i}}) \equiv r \mod n-1.$$
Since $n \equiv 1 \mod{n-1}$ we therefore have that $r(T)\sig \equiv r \mod n-1$ as required.

For the reverse implication  let $T \in \TOn$ and $1 \le r < n-1$ be such that $r (T)\sig \equiv r \mod{n-1}$. Let $k \in \N$ be such that $T$ is synchronizing at level $k$, once more assume that the set $\Xn^{k}$ is  ordered lexicographically as  $x_1 < x_2< \ldots < x_{n^{k}}$. For $1 \le i \le n^{k}$, let $q_{x_i}$ be the state of $T$ forced by $x_i$. For each $1 \le i \le n^{k}$, let $V_{q_{x_i}}$ be the smallest subset of $\Xns$, with size $m_{q_{x_i}}$, such that $U(V_{q_{x_i}})$ is a clopen cover of $\im(q)$ consisting of clopen subsets of $\im(q)$; let $M_{i} = \max\{\left|v\right| \mid v \in V_{q_{x_i}}\}$ and set $M = \max_{1 \le i \le n^{k}} M_{i}$. Let $j \in \N$ be minimal such that for any word $\Gamma \in \Xn^{j}$ and any state $q$ of $T$, $|\lambda_{T}(\Gamma, q)|> M$. Order the set $\Xn^{j}$  lexicographically  as $y_1 < y_2 < \ldots < y_{n^{j}}$. For each state $q_{x_i}$, $1 \le i \le n^{k}$, order the set $V_{q_{x_i}}$ lexicographically as $\nu_{i,1} < \nu_{i,2}< \ldots < \nu_{i,m_{q_{x_i}}}$, and for all $1 \le l \le n^{j}$ let $\mu_{i,l}\varphi_{i,l} = \lambda_{T}(y_l, q_{x_i})$, for some $\mu_{i,l} \in V_{q_{x_i}}$, $\varphi_{i,l} \in \Xnp$, and $p_{i,l} = \pi_{T}(y_l, q_{x_i})$. Now since, for $1 \le i \le n^{k}$, $|V_{q_{x_i}}| = m_{q_{x_i}}$, each $\mu_{i,l} = \nu_{i,a}$ for some $1 \le a \le m_{q_{x_i}}$ and we may write, for all $1 \le l \le n^{j}$, $\mu_{i,l}\varphi_{i,l} = \nu_{i,a}\rho_{i,l_a}$ where $\nu_{i,a} = \mu_{i,l}$ and  $\rho_{i,l_a} = \varphi_{i,l}$ for some $1 \le a \le m_{q_{x_i}}$, we also adopt the same notation for the set of $p_{i,l}$ and write $p_{i,l_a}$, for $1 \le a \le m_{q_{x_i}}$, where $\nu_{i,a} \rho_{i,l_a} = \mu_{i,l}\varphi_{i,l}$.

Now let  $\ac{u}$ be a maximal antichain of $\CCnr$ of length $r n^{j+k}$ and let $\ac{v}$ be a maximal antichain of $\CCnr$ of length $r (T)\sig$. Write $\ac{u} = \cup_{1 \le t \le r}\ac{u}_t$ where each $\ac{u}_{t}$ is ordered lexicographically, $| \ac{u}_t| = n^{j+k}$ and for $1 \le t_1 < t_2 \le r$ all elements  of $\ac{u}_{t_1}$ are strictly less than all elements of $\ac{u}_{t_2}$ in the lexicographic ordering on $\Xns$. Let $\ac{v} = \cup_{1 \le t \le r}\ac{v}_t$ where each $\ac{v}_{t}$ is ordered lexicographically, $| \ac{v}_{t}| = (T)\sig$ and for $1 \le t_1 < t_2 \le r$ all elements  of $\ac{v}_{t_1}$ are strictly less than all elements of $\ac{v}_{t_2}$ in the lexicographic ordering on $\Xns$.

Fix $1 \le t \le r$ and consider $\ac{u}_{t}$. Write $\ac{u}_{t}= \{u_{t,i,l} \mid 1 \le i \le n^{k}, 1 \le l \le n^{j}\}$. We further assume that, for a fixed $1 \le i \le n^{k}$, the set $\{u_{t,i,l} \mid 1 \le l \le n^{j}\}$ is ordered lexicographically and, for $1 \le i_1 < i_2 \le n^{k}$ and for all $1 \le l_1, l_2 \le n^{j}$,  $u_{t,i_1, l_1} \lelex  u_{t,i_2, l_2}$. Likewise write $\ac{v}_{t} = \{ v_{t,i,a} \mid 1 \le i \le n^{k}, 1 \le a \le m_{q_{x_i}} \}$.  We further assume that, for a fixed $1 \le i \le n^{k}$, the set $\{v_{t,i,a} \mid 1 \le a \le m_{q_{x_i}}\}$ is ordered lexicographically and, for $1 \le i_1 < i_2 \le n^{k}$,  $v_{t,i_1,a} \lelex  v_{t,i_2,b}$ for all $1 \le a \le m_{q_{x_{i_1}}}$ and $1 \le b \le m_{q_{x_{i_2}}}$. Furthermore, whenever, for $1 \le a \le m_{q_{x_i}}$ and $1 \le l \le n^{j}$, we have $\nu_{i,a}\rho_{i,l_a} =  \mu_{i,l}\varphi_{i,l}$ we set $\eta_{t,i,l} := v_{t,i,a}$.

Define a map $f$ from  $\CCnr$ to itself as follows. For $1 \le t \le r$, $f$ acts on elements with prefix in the set $\ac{u}_{t}$ as follows. For $1 \le i \le n^{k}$ and $1 \le l \le n^{j}$, and $\Gamma \in \CCn$, $u_{t,i,l}\Gamma \mapsto \eta_{t,i,l}\varphi_{i,l}(\Gamma)p_{i,l}$. Observe that for a fixed $1 \le i \le n^{k}$, and a fixed $1 \le a \le m_{q_{x_i}}$, since we have $\cup \{ \mu_{i,l}\varphi_{i,l}\im(p_{i,l}) \mid \mu_{i,l} = \nu_{i,a} \} = U_{\nu_{i,a}}$ then it is the case that, for a fixed $1 \le t \le r$,  $\cup \{ \eta_{t,i,l}\varphi_{i,l}\im(p_{i,l}) \mid \eta_{i,l} = v_{t,i,a} \} = U_{v_{t,i,a}}$.  Now since, for a given $1 \le i \le n^{k}$,  $\bigcup_{1 \le l \le n^{j}} \mu_{i,l} \varphi_{i,l} \im(p_{i,l}) = \im(q)$ and as $q$ is injective and preserves the lexicographic ordering of $\CCn$, we see that $f \restriction_{U(\ac{u}_{t})}$ is a bijection unto the set $U(\ac{v}_{t})$ which preserves the lexicographic ordering of $\CCn$. More specifically, since $f$ acts by a state of $T$ after a finite depth, we see that $f$ is a homeomorphism of $\CCnr$ which is in fact an element of $\TBnr$.

The proof of the other reading proceeds in an analogous fashion only here we do not have to worry about preserving the lexicographic ordering on $\CCnr$.
\end{proof}

\begin{Remark}
It follows from results in \cite{GriNekSus} that given an element $T \in \TOn [T \in\On]$ and $1 \le r < n$, then it is possible to decide in finite time if $T \in \TOnr [T \in\Onr]$. Furthermore, by the above proposition,  $T \in \TOns{1} [T \in \Ons{1}]$ if and only if  $(T)\sig \equiv 1 \mod{n-1}$.
\end{Remark}

The following result is a consequence of Proposition~\ref{Proposition:sigdeterminesmembership}:

\begin{lemma}\label{Lemma:TOni is a subset of TOnj if mi is congruent to j mod n-1}
Let $n \in \N$ and suppose that $n\ge 2$. Let $i, j, m \in \Z_{n}$ such that, $i$ and $j$ are non-zero and $m i \equiv j \mod n-1$ then $\TOns{i} \subseteq \TOns{j}$ [$\Ons{i} \subseteq \Ons{j}$].
\end{lemma}
\begin{proof}
Let $i,j, m$ be as in the statement of the lemma. Let $T \in \TOns{i} [T \in \Ons{i}]$, then by Proposition~\ref{Proposition:sigdeterminesmembership} $i (T)\sig \equiv i \mod{n-1}$, therefore $mi (T)\sig \equiv mi \mod{n-1}$ and so $j (T)\sig \equiv j \mod{n-1}$ and $T \in \TOns{j} [T \in \Ons{j}]$  again by  Proposition~\ref{Proposition:sigdeterminesmembership}.
\end{proof}

\begin{Remark}
Observe that the lemma above implies that whenever $j \in \Z_{n}$ is co-prime to $n-1$ then $\TOns{j} = \TOns{1}$ [$\Ons{j} = \Ons{1}$]. Thus it follows that for $n$ a natural number bigger than $2$ such that $n-1$ is prime, if $T \in \TOn$ [$T \in \On$],  satisfies $(T)\sig \not\equiv 1 \mod{n-1}$ then $T \in \TOn \backslash \{\TOns{1}\}$[$T \in \On \backslash \{\Ons{1}\}$]. The following result generalises this observation.
\end{Remark}

An immediate corollary of the Lemma~\ref{Lemma:TOni is a subset of TOnj if mi is congruent to j mod n-1} is the following result, an analogous result appears in \cite{BCMNO} for the groups $\Onr$:

\begin{Theorem}\label{Theorem:TOnisubsetofTOnjifidividesjinZnminus1}
Let $n \in \N$, $n \ge 2$, then for all non-zero $i \in \Z_{n}$ we have $\TOns{i} \subseteq \TOns{n-1}$. Hence $\TOns{n-1} = \TOn$ and   $\TOns{1} \subseteq \TOns{i}$ for all $1 \le i \le n-1$, hence $\cap_{1 \le r \le n-1} \TOns{r} = \TOns{1}$.
\end{Theorem}
\begin{proof}
Let $i$ be non-zero in $\Z_{n}$, then observe that $(n-1)\ast i \equiv n-1 \mod{n-1}$ and $i*1 \equiv i \mod{n-1}$. Therefore by Lemma~\ref{Lemma:TOni is a subset of TOnj if mi is congruent to j mod n-1} above, we have $\TOns{i} \subseteq \TOns{n-1}$ and $\TOns{1} \subseteq \TOns{i}$. Thus $\TOns{n-1} = \TOn$ and $\cap_{1 \le r \le n-1} \TOns{r} = \TOns{1}$.
\end{proof}

The following result is again a corollary of Lemma~\ref{Lemma:TOni is a subset of TOnj if mi is congruent to j mod n-1}, we observe that the result for $\Onr$ was proved in \cite{BCMNO}:

\begin{corollary}\label{Corollary:TOniisequaltoTOnjforsomejdividingnminus1}
Let $n$ be a natural number bigger than $2$, $j \in \Z_{n}\backslash \{0\}$ and $d$ be the greatest common divisor of $n-1$ and $j$, then $\TOns{j} = \TOns{d}$ [$\Ons{j} = \Ons{d}$]. 
\end{corollary}
\begin{proof}
Let $n,j,d$ be as in the statement of the corollary. Since $d$ is the greatest common divisor of $n-1$ and $j$, there are co-prime numbers $a, b \in \Z_{n}\backslash\{0\}$ such that $j = da$ and $n-1 = db$. Since $a$ and $b$ are co-prime, there are numbers $u, v \in \Z$ such that $ua = 1 + vb$. Multiplying both sides of the equation by $d$ it follows that $uad = d + vbd$ and so $uj = d + v(n-1)$. Therefore there is some $m_1 \in \Z_{n-1}$ such that $m_1j \equiv d \mod {n-1}$. Moreover, since $d$ divides $j$, there is some $m_2 \in \Z_{n-1}$ such that $m_2 d = j \mod{n-1}$. It therefore follows by Lemma~\ref{Lemma:TOni is a subset of TOnj if mi is congruent to j mod n-1} that $\TOns{j} = \TOns{d}$[$\Ons{j} = \Ons{d}$].
\end{proof}

Corollary~\ref{Corollary:TOniisequaltoTOnjforsomejdividingnminus1} should be compared with the result of Pardo \cite{EPardo} showing that $G_{n,r} \cong G_{m,s}$ if and only if $n=m$ and $\gcd(n-1,r) = \gcd(n-1,s)$. It is a question in \cite{BCMNO} whether or not $\Ons{r} \cong \Ons{s}$ if and only if $\gcd(n-1,r) = \gcd(n-1,s)$. Below (Remark~\ref{Remark:negativesolutiontoquestionofCollinetal}) we show that this question has a negative answer.  

The following Lemma can be thought of as a partial converse to Lemma~\ref{Lemma:TOni is a subset of TOnj if mi is congruent to j mod n-1}.

\begin{lemma}\label{Lemma:partialconverselemma}
Let $T \in \TOn$ [$T \in \On$] and let $r$ be minimal such that $T \in \TOnr$ [$T \in \Onr$], then $T \in \TOns{j}$ [$T \in \Ons{j}$] for some $1 \le j \le n-1$ if and only if  there is some $m \in \Z_{n}$ such that $mr = j$.
\end{lemma}
\begin{proof}
The forward implication is a consequence of Lemma~\ref{Lemma:TOni is a subset of TOnj if mi is congruent to j mod n-1}. Therefore let $T \in \TOn$ and $r$ be minimal such that $T \in \TOnr$ Let $1 \le j \le n-1$ be such that $T \in \TOns{j}$. By assumption we must have that $r < j$. Since $T \in \TOnr$ then by Proposition~\ref{Proposition:sigdeterminesmembership}, $r(T)\sig \equiv r \mod{n-1}$ and since $T \in \TOns{j}$, then $j (T)\sig \equiv j \mod{n-1}$. Thus we deduce that $(j-r)(T)\sig \equiv (j-r) \mod{n-1}$. If $(j-r) > r$, then we may repeat the process otherwise $j-r = r$, by the minimality assumption on $r$ and Proposition~\ref{Proposition:sigdeterminesmembership}, in which case $j = 2r$. Inductively there is some $k \in \N$ such that $j-kr = r$ and so $j = (k+1)r$ which concludes the proof. 

The other reading of the lemma is proved analogously, simply replace $\TOn$ with $\On$ in the paragraph above.
\end{proof}

We now show that the map $\rsig: \On \to \Z_{n-1}$ is a homomorphism from $\On$ to the group of units of $\Z_{n-1}$ with kernel $\Ons{1}$.

We begin with the following result:

\begin{proposition}\label{Proposition:signatureisequaltosizeofimagecovermodnminus1}
Let $T \in \On$, $q$ be any state of $T$ and $m_q$ be the size of the smallest subset $V$ of $\Xns$ such that $U(V) = \{ U_v \mid v \in V\}$  is a clopen cover of $\im(q)$ and $U_{v} \subset \im(q)$ for all $v \in V$, then  $m_{q} \mod{n-1} = (T)\rsig$.
\end{proposition}
\begin{proof}

For any state  $p$ of $T$ let $V(p)= \{v_{1,p}, v_{2,p}, \ldots
, v_{m_p,p}\} \subset \Xns$ be such that $\im(p) = \cup_{1 \le i \le m_{p}} U_{v_{i,p}}$. Now fix a state $q$ of $T$, let $k$ be the minimal synchronizing level of $T$ and $j \in \N_{k}$ be such that for any word $\Gamma \in \Xn^{j}$ and any state $p$ of $T$, $|\lambda_{T}(\Gamma, p)| \ge  \max\{ |v_{1,q}| \mid v_{1,q} \in V(q) \}$. Order the set  $X_{n}^{j}$ in the lexicographic ordering as follows $x_1 < x_2 < x_3 \ldots < x_{n^{j}}$, for $1 \le a \le n^{j}$ let $q_{x_a}$ be the unique state of $T$ forced by $x_a$ and $\rho_{x_a} = \lambda_{A}(x_a, q)$. Observe that for all $1 \le a \le n^{j}$, $|\rho_{x_a}| \ge \max\{ |v_{1,q}| \mid v_{1,q} \in V(q) \}$ and so $\rho_{x_a}$ has a prefix in $V(q)$. Now since the state $q$ is induces a homeomorphism from $\CCn$ unto its image, it follows that the set $\cup_{1 \le a \le n^{j}} \{\rho_{x_a}v_{i,q_{x_a}} \mid 1 \le i \le m_{q_{x_a}} \}$ is  a complete antichain for the subset of $\Xns$ consisting of all elements with prefix in $V(q)$. From this we deduce that  $|\cup_{1 \le a \le n^{j}} \{\rho_{x_a}v_{i,q_{x_a}} \mid 1 \le i \le m_{q_{x_a}} \}| \equiv m_q \mod{n-1}$.

On the other hand 
\[
|\cup_{1 \le a \le n^{j}} \{\rho_{x_a}v_{i,q_{x_a}} \mid 1 \le i \le m_{q_{x_a}} \}| = \sum_{1 \le a \le n^{j}} m_{q_{x_a}}.
\]
However, since $T$ is synchronizing at level $k$, the sequence $(q_{x_a})_{1 \le a \le n^j}$ is in fact equal to the sequence $(q_{x_a})_{(1 \le a \le n^{k})}$ repeated $n^{j-k}$ times. Therefore we have,
\[
|\cup_{1 \le a \le n^{j}} \{\rho_{x_a}v_{i,q_{x_a}} \mid 1 \le i \le m_{q_{x_a}} \}| = n^{j-k}\sum_{1 \le a \le n^{k}} m_{q_{x_a}} 
\]
 and so we conclude that $m_{q} \mod {n-1} = n^{j-k}\sum_{1 \le a \le n^{k}} m_{q_{x_a}} \mod{n-1} = (T)\rsig$. 
\end{proof}

\begin{Theorem}\label{Theorem:rsigisahomomorphism}
The map $\rsig: \On \to \Z_{n-1}$ is a homomorphism from $\On$ into the group of units of $\Z_{n-1}$ with kernel $\Ons{1}$.
\end{Theorem}
\begin{proof}
Let $T, U \in \On$,  $j, k \in \N$ be the minimal synchronizing levels of $T$ and $U$. Consider the transducer product $T \ast U$. Let $(p,q)$ be any state in the core of  $T \ast U$, and for $\sharp \in \{p,q\}$ let $V(\sharp)= \{v_{1,\sharp}, v_{2,\sharp}, \ldots
, v_{m_p,\sharp}\} \subset \Xns$ be such that $\im(\sharp) = \cup_{1 \le i \le m_{\sharp}} U_{v_{i,\sharp}}$. We may assume that $V_{q}$ is the smallest subset of $\Xns$ with this property. Let  $m = \max\{ |v_{i,q} | \mid v_{i,q} \in V(q) \}$ and $l \in \N$ be such that for any state  $p'$ of $T$ and any word $\Gamma \in \Xn^{l}$, $|\lambda_{T}(\Gamma, p')| \ge m$. We may further assume that $V(p)$ is the smallest subset of $\Xn^{l}$ satisfying $\im(p) = \cup_{1 \le i \le m_{p}} U_{v_{i,p}}$. For each $1 \le i \le m_{p}$ let $q_i = \pi_{U}(v_{i,p}, q)$ and $\rho_i = \lambda_{U}(v_{i,p}, q)$. Since $\im(p) = \cup_{1 \le i \le m_{p}} U_{v_{i,p}}$, it follows that $\im((p,q)) = \sqcup_{1 \le i \le m_{p}} \{ \rho_i \im(q_i)\}$. For each $1 \le i \le m_{p}$, let $m_{q_i}$ be the size of the smallest subset $V(q_i) \subset \Xns$ such that $\cup_{v \in V(q_i)} U_{v} = \im(q_i)$, it follows that $V:=\cup_{1 \le i \le m_p} \{ \rho_i v \mid v \in V(q_i) \}$ satisfies $\cup_{v \in V}(U_{v}) = \im((p,q))$. Notice that since both $p$ and $q$ are injective, then $|V| = \sum_{1 \le i \le m_{p}} m_{q_i}$.

Observe for any other set $V' \subset \Xns$ such that $\cup_{v \in V'} U_{v} =  \im((p,q))$ it must be the case that $|V'| \equiv |V| \mod{n-1}$. Thus if $V'$ is the smallest subset of $\Xns$ with this property, we have $|V'| \equiv |V| \mod{n-1}$. Furthermore we observe that removing incomplete response from the states of $T\ast U$ simply removes the greatest common prefix of $\im(p,q)$. Therefore if the state $s$ of $TU$ is equal to the state $(p,q)$ after removing the incomplete response, then for $V'' \subset \Xns$ minimal such that $\cup_{v \in V''} U_{v} =  \im(s)$, we have $|V''| \equiv |V| \mod{n-1}$. By Proposition~\ref{Proposition:signatureisequaltosizeofimagecovermodnminus1} it follows that
\[
 (TU)\rsig \equiv |V''| \equiv \sum_{1 \le i \le m_{p}} m_{q_i} \equiv  (T)\rsig(U)\rsig\mod{n-1}.
\]

Since $(\id)\rsig = 1$, it follows that $\rsig$ is a homomorphism from $\On$ to the group of units of $\Z_{n-1}$.

To see that $\ker(\rsig) = \Ons{1}$, observe that by Proposition~\ref{Proposition:sigdeterminesmembership} $(T)\rsig = 1$ if and only if $T \in \Ons{1}$.

\end{proof}

Given $T \in \On$ the following result enables us to compute the reduced signature of $T^{-1}$ directly from $T$ i.e. without computing the inverse.

\begin{proposition}
Let $T \in \On$ and $q$ be any state of $T$. Let $\nu \in \CCn$ be such that $U_{\nu} \subset \im(q)$ and $j \in \N$ be such that for any word $\Gamma \in \Xn^{j}$, $|\lambda_{T}(\Gamma, q)| \ge |\nu|$. Let $W \subset \Xn^{j}$ be maximal such that for any word $\Delta \in W$, $\nu$ is a prefix $\lambda_{T}(\Delta, \Gamma)$. Let $1 \le  w \le n-1$ be such that $w \equiv |W| \mod{n-1}$, then $w$ depends only on $T$, in particular $w = (T^{-1})\rsig$.
\end{proposition}
\begin{proof}
By Proposition~\ref{Proposition:signatureisequaltosizeofimagecovermodnminus1} it suffices to show that there is a state $p'$ of $T^{-1}$ and a subset $V  \subset \Xns$  such that $\cup_{\mu \in V} U_{\mu} = \im(p)$ and $|V| = |W|$.

Let $A_{q_0}$ a bi-synchronizing transducer with $\core(A_{q_0}) = T$ and $k \in \N_{1}$ be the minimal bi-synchronizing level of $A_{q_0}$. Let $\varphi =  (\nu)L_{q}$ i.e. $\varphi$ is the greatest common prefix of the set $h_{q}^{-1}(U_{\nu})$. Let $\nu_1 = \lambda_{T}(\varphi, q)$, $p = \pi_{T}(\varphi, q)$ and $\nu_2 = \nu - \nu_1$. We claim that $(\nu_2,p)$ is  a state of $A_{(\epsilon, q_0)}$ and is $\omega$-equivalent to a state of $T^{-1}$.

 For, let $l \in \N$ be such that for any word $\Gamma \in \Xn^{l}$, $|\lambda_{A}(\Gamma, q_0)| \ge k$.
 Let $\Gamma \in \Xn^{l}$ be such that $\pi_{A}(\Gamma, q_0) = q \in \core(A_{q_0})$. Such a word $\Gamma$ exists because of the bi-synchronizing condition. Let $\Delta = \lambda_{A}(\Gamma, q_0)$. Then observe that $(\Delta\nu)L_{q_0} = \Gamma (\nu)L_{q} = \Gamma \varphi$ since $U_{\nu} \subset \im(q)$. Thus, in the transducer $A_{(\epsilon, q_0)}$, $\pi_{A}(\Delta\nu, q_0) = (\Delta\nu - \lambda_{A}(\Gamma \varphi, q_0), \pi_{A}(\Gamma\varphi, q_0))= (\nu_2,p)$. Now since $A_{(\epsilon, q_0)}$ has no states of incomplete response, is $\omega$-equivalent to the minimal transducer $B_{p_0}$ representing $h_{q_0}^{-1}$, and by the choice of $l$, we therefore have that $(\nu_2,p)$ is $\omega$-equivalent to a state in $\core(B_{p_0}) = T^{-1}$.
 
 Now, suppose $W = \{ \rho_1', \rho_2', \ldots, \rho_{m}'\}$. We observe that by definition of $W$, $\varphi$ is the greatest common prefix of $W$, thus for $1\le i \le m$, let $\rho_i \in \Xns$ be such that $\varphi\rho_i = \rho_i'$. Observe that for any $\delta \in \CCn$, there is a unique $1 \le i \le m$,  and $\xi \in \CCn$ such that $\lambda_{A}(\rho_i'\xi, q) = \nu \delta$, therefore $\lambda_{A}(\rho_i \xi, p) = \nu_2\delta$. Moreover, for any $1 \le i \le m$ and $\xi \in \CCn$, $\lambda_{A}(\rho_i\xi, p) = \nu_2 \delta$ for some $\delta \in \CCn$. Since $\im((\nu_2,p)) = \{(\nu_2\delta)L_{p} \mid \delta \in \CCn \} = (U_{\nu_2})h_{p}^{-1}$, it follows that $\im((\nu_2,p)) = \cup_{1 \le i \le m} U_{\rho_i'}$. Therefore it follows, from Proposition~\ref{Proposition:signatureisequaltosizeofimagecovermodnminus1}, that $m = |W| \equiv (T^{-1})\rsig \mod{n-1}$.

\end{proof}

The corollary below follows straight-forwardly from Theorem~\ref{Theorem:rsigisahomomorphism}

\begin{corollary}
Let  $1< r< n$, then the following hold:
\begin{enumerate}[label = (\alph*)]
\item $\Onr$ is a normal subgroup of $\Ons{n-1} = \On$, in particular for non-zero $i, j \in \Z_{n}$ such that $i$ divides $j$ in the additive group $\Z_{n-1}$, $\Ons{i} \unlhd \Ons{j}$,
\item $[\On, \On] \le \Onr$,
\item $\Onr/\Ons{1}$ is isomorphic to a subgroup of the group of units of $\Z_{n-1}$ and $\On/\Onr$ is isomorphic to a quotient of the group of units of $\Z_{n-1}$.   
\end{enumerate}
   
\end{corollary}
\begin{proof}
Since the group of units of $\Z_{n-1}$ is abelian, it follows that $[\On, \On] \le \Ons{1}$ and by Theorem~\ref{Theorem:TOnisubsetofTOnjifidividesjinZnminus1} $[\On, \On] \le \Ons{r}$.

To see that $\Onr$ is normal in $\On$, we observe that given $T \in \Onr$, and $U \in \On$, then $(U^{-1}TU)\rsig \equiv (U^{-1})\rsig(T)\rsig(U)\rsig \mod{n-1}$ and so $(U^{-1}TU)\rsig = (T)\rsig$ and by Proposition~\ref{Proposition:sigdeterminesmembership}, we have $U^{-1}TU \in \Onr$.

That $\Onr/\Ons{1}$ is isomorphic to a subgroup of the group of units of $\Z_{n-1}$, follows by the correspondence theorem  since $\Ons{1} \le \Onr$. 

That $\On/\Onr$ is isomorphic to a quotient of the group of units of $\Z_{n-1}$, follows from the third isomorphism theorem: $\On/\Onr \cong (\On/\Ons{1})/(\Onr/\Ons{1})$.
\end{proof}

\begin{Question}\label{Question:issigsurjective}
Is it the case that $[\On, \On] = \Ons{1}$? Is the map $\rsig$ from $\On$ to the group of units of $\Z_{n-1}$ surjective?
\end{Question}

\begin{Remark}
Since $\TOn \le \On$ the results above and questions can be restated with $\TOn$ in place of $\On$.
\end{Remark}

\section{ Nesting properties of the groups $\TOnr$}\label{Section:nestingproperties2}

In this section we focus on the group $\TOn$ however, {\bfseries{all the results below are equally valid when all occurrences of $\TOn$ are replaced with $\On$}}.

The following result is a direct consequence corollary of Proposition~\ref{Proposition:sigdeterminesmembership} and generalises Corollary~\ref{Corollary:TOniisequaltoTOnjforsomejdividingnminus1}:

\begin{corollary}\label{Corollary:OnrisequaltoOnjifandonlyif}
Let $1\le r,s<n$ be natural numbers and let $(\TOn)\rsig$  $[(\On)\rsig]$ denote the image of $\TOn$ [$\On$] under the map $\rsig$. Then $\TOnr = \TOns{s}$ [$\Onr = \Ons{s}$] if and only if for all $j \in (\TOn)\rsig$ [$j \in (\On)\rsig$], $r(j-1)\equiv s(j-1) \equiv 0 \mod{n-1}$.
\end{corollary}

\begin{Remark}\label{Remark:negativesolutiontoquestionofCollinetal}
Observe that when $n=7$, then $(\On)\rsig \subseteq  \{1,5\}$. Further observe, by Corollary~\ref{Corollary:TOniisequaltoTOnjforsomejdividingnminus1}, that for $1 \le r \le 6$, $\Oms{7}{r}$ is equal to  $\Oms{7}{a}$ for $a \in \{1,2,3,6\}$. Now notice that as $12$ is divisible by $6$, then $\Oms{7}{3} = \Oms{7}{6}$ and, since $10$ is not a multiple of $6$, $\Oms{7}{1} = \Oms{7}{2}$.  However $\gcd(1,6) \ne \gcd(2,6)$ and $\gcd(3,6) \ne \gcd(6,6)$. Therefore it is not the case that for $1 \le r,s \le n$ $\Onr \cong \Ons{s}$ if and only if $\gcd(r,n-1) = \gcd(s, n-1)$. This yields a negative solution to a question in the paper \cite{BCMNO}. We notice that the same result holds if $\On$ is replaced with $\TOn$. More specifically, $\TOms{7}{a} = \TOms{7}{b}$ for $(a,b) \in \{(1,2), (3,6)\}$. 
\end{Remark}

Remark~\ref{Remark:negativesolutiontoquestionofCollinetal} proves the following result:

\begin{Theorem}
There are numbers $n \in \N$, $n >2$, and $1 \le r, s \le n-1$ such that, for $\T{X}= \T{TO}, \T{O}$, $\XOnr = \XOns{s}$ but $\gcd(n-1, r) \ne \gcd(n-1, s)$.
\end{Theorem}

Partition the set $\Z_{n} \backslash \{0\}$ as follows. For $i \in \Z_{n} \backslash \{0\}$, set $[i] := \{j \in \Z_{n} \backslash \{0\} \mid \TOns{j}  = \TOns{i}\} \subset  \Z_{n}$, $[\Z_{n}^{0}] \seteq \{ [i] \mid i \in \Z_{n} \backslash \{0\} \}$ and, for $[i] \in [\Z_{n}^{0}]$, $\TOns{[i]} \seteq \TOns{i}$. Observe that the set $[\Z_{n}^{0}]$ inherits an ordering from $\Z_{n}$ where $[i] < [j]$ if the smallest element of $[i]$ is less than the smallest element of $[j]$. Further observe that if $n-1$ is prime, then $[\Z_{n}^{0}]$ has size at most $2$ by Theorem~\ref{Theorem:TOnisubsetofTOnjifidividesjinZnminus1}.

\begin{Definition}[Atoms]
Let $[r] \in [\Z_{n}^{0}]$, then we say $[r]$ is \emph{an atom (of $[\Z_{n}^0]$)} if  there is an element $T \in \TOns{[r]}$ which is not an element of $\TOns{[s]}$ for any $[s] < [r]$. Let $[i] \in [\Z_{n}^{0}]$, an \emph{atom of $[i]$} is an atom $[r]$ of $[\Z_{n}^0]$ such that $\TOns{[r]} \le \TOns{[i]}$. 
\end{Definition}

\begin{Remark}\label{Remark:atomsdividenminus1}
Let $[r] \in [\Z_{n}^0]$ be an atom with $r$ the minimal element of $r$, then as a consequence of Lemma~\ref{Lemma:partialconverselemma} and Theorem~\ref{Theorem:TOnisubsetofTOnjifidividesjinZnminus1}, we have that $r|n-1$. Observe that by Remark~\ref{Remark:negativesolutiontoquestionofCollinetal} it is not always the case that for every element $r$ dividing $n-1$ that $[r]$ is an atom of $[\Z_{n}^{0}]$. 
\end{Remark}

Observe that if the map $\rsig$ is surjective, then for each $i \in \Z_{n} \backslash \{0\}$,  we may completely determine the elements of $[i]$. 

\begin{Question}
Is it the case that all elements $[i] \in [\Z_{n}^{0}]$ are atoms?
\end{Question}

In the interim we make the following definitions.

\begin{Definition}
Let $[i],[j] \in [\Z_{n}^{0}]$ where $i,j \in \Z_{n}$ are the minimal elements of $[i]$ and $[j]$ respectively. We say that $[i]$ \emph{divides} $[j]$ if $r|j$ for any atom $[r]$ of $[i]$ with $r$ the minimal element of $[r]$.
\end{Definition}

\begin{Definition}
Given two elements $[i],[j] \in [\Z_{n}^0]$, the \emph{lowest common multiple of $[i]$ and $[j]$} is the smallest element $[l]$  of $[\Z_{n}^0]$ such that $[i]$ and $[j]$ divide $[l]$. Notice that by Remark~\ref{Remark:atomsdividenminus1} the lowest common multiple of any pair of numbers $[i], [j] \in [\Z_{n}^0]$ always exists. We extend the definition of lowest common multiple to tuples of elements of $[\Z_{n}^0]$ in the usual way.
\end{Definition}

\begin{Definition}
 Let $i, j \in \Z_{n} \backslash \{0\}$ be the minimal elements of $[i]$ and $[j]$ respectively, then we define the \emph{greatest common divisor of $[i]$ and $[j]$}  to be  $[r] \in [\Z_{n}^0]$ such that $r$, the minimal element of $[r]$, is the greatest common divisor of $i$ and $j$.
\end{Definition}

\begin{Remark}
Notice that for $[i], [j], [r] \in [\Z_{n}^{0}]$ where $i,j,r$ are the minimal elements of $[i]$, $[j]$ and $[r]$ respectively, if $[r]$ is an atom of $[i]$ then, by Lemma~\ref{Lemma:partialconverselemma}, $r|i$ and so if $i|j$ then $r|j$ and $[i]|[j]$. Further observe that for $[i] \in [\Z_{n}^{0}]$ either $[i]$ is an atom or $\TOns{[i]}$ is a union of groups $\TOns{[r]}$ for atoms $[r]$ of $[i]$. This is because  for any element $T \in \TOns{[i]}$, either $T$ is not an element of $\TOns{[r]}$ for any $[r]<[i]$  and so $[i]$ is an atom, or there is a  minimal $[j] \in [\Z_{n}^0]$ such that $T \in \TOns{[j]}$ and so $[j]$ is an atom of $[i]$ by Lemma~\ref{Lemma:partialconverselemma} and Lemma~\ref{Lemma:TOni is a subset of TOnj if mi is congruent to j mod n-1}.
\end{Remark}

 As a consequence of Lemma~\ref{Lemma:partialconverselemma} we have the following result.

\begin{proposition}\label{Proposition:partiallatticestructure}
Let $1 \le i, \le j \le n-1$ be integers such $i$ is the smallest element of $[i]$ and $j$ is the smallest element of $[j]$, then the following things hold:
\begin{enumerate}[label=(\alph*)]
\item $\TOns{[i]} \le \TOns{[j]}$ if and only if $[i]$ divides $[j]$ \label{Proposition:partiallatticestructure order},
\item if $[r]$ is the lowest common multiple of $[i],[j]$, then  $[r]$ is minimal such that $\gen{\TOns{[i]}, \TOns{[j]} } \le \TOns{[r]}$ \label{Proposition:partiallatticestructure join},
\item if $[r]$, where $r \in \Z_{n}$ is the smallest element of $[r]$, is the greatest common divisor of $[i]$ and $[j]$   then $\TOns{[i]} \cap \TOns{[j]} = \TOns{[r]}$. \label{Proposition:partiallatticestructure meet}
\end{enumerate}
 
\end{proposition}
\begin{proof}
 For part~\ref{Proposition:partiallatticestructure order}, let $i, j$ be elements of $\Z_{n} \backslash \{0\}$ which are the minimal elements of sets $[i]$ and $[j]$ respectively. Suppose $[i]$ divides $[j]$. If $[i]$ is an atom then  $i|j$, and so  by Lemma~\ref{Lemma:TOni is a subset of TOnj if mi is congruent to j mod n-1} we have that $\TOns{i} \le \TOns{j}$, in particular $\TOns{[i]} \le \TOns{[j]}$. If $[i]$ is not an atom, then $\TOns{[i]}$ is the union of the groups $\TOns{[r]}$ over all atoms $[r]$ of $[i]$. Therefore let $[r]$ be any atom of $[i]$, where $r$ is the smallest element of $[r]$, since $[i]$ divides $[j]$ then $r|j$ and so $\TOns{r} \le \TOns{j}$ by Lemma~\ref{Lemma:TOni is a subset of TOnj if mi is congruent to j mod n-1} once more. Since $[r]$ was an arbitrary atom of $[i]$, we conclude that $\TOns{[i]} \le \TOns{[j]}$.  On the other hand suppose that $\TOns{i} \le \TOns{j}$.If $[i]$ is an atom then $i|j$ and so $[i]$ divides $[j]$. If $[i]$ is not an atom, then let $[r]$ be any atom of $[i]$ where $r$ is the minimal element of $[r]$. Since $\TOns{[i]} \le \TOns{[j]}$, we also have $\TOns{[r]} \le \TOns{[j]}$ by definition of an atom of $[i]$. Now making use of the definition of an atom together with Lemma~\ref{Lemma:partialconverselemma}, we have that $r|j$. Therefore, since $[r]$ was an arbitrarily chose atom of $[i]$, we conclude that $[i]$ divides $[j]$.
 
 Part~\ref{Proposition:partiallatticestructure join} follows from Part~\ref{Proposition:partiallatticestructure order} since for any $s \in \Z_{n}\backslash\{0\}$ such that $\gen{\TOns{[i]}, \TOns{[j]}}\le \TOns{[s]}$, then  $[i]$ and $[j]$ divide $[s]$. 
 
 For Part~\ref{Proposition:partiallatticestructure meet} let $[r]$ be the greatest common divisor of $[i]$ and $[j]$. First observe that if $[l]$ is an atom of $[r]$, where $l$ is the minimal element of $[l]$, then $l|i$ and $l|j$. Therefore $[r]$ divides $[i]$ and $[j]$ and by Part~\ref{Proposition:partiallatticestructure order}, $\TOns{[r]} \le \TOns{[i]} \cap \TOns{[j]}$. Now let $[s]$, with $s$ the minimal element of $[s]$ be  an atom of $[i]$  and $[j]$. This means that $s|i$ and $s|j$ and so $s|r$, since $r$ is the greatest common divisor of $i$ and $j$. By Part~\ref{Proposition:partiallatticestructure order} once more, we conclude that $s$ is an atom of $[r]$ and so $\TOns{[s]} \le \TOns{[r]}$. Now since every element $T \in \TOns{[i]} \cap \TOns{[j]}$ lies in some atom of both $[i]$ and $[j]$, we conclude that  $\TOns{[i]} \cap \TOns{[j]} = \TOns{[r]}$.
\end{proof}

The proposition above prompts the following question:

\begin{Question}\label{Question:canjoinbereplacedwithubgroupgenerated}
Let $n$ be a natural number and let $[i], [j] \in [\Z_{n}^0]$ is it true that $\gen{\TOns{[i]}, \TOns{[j]}} = \TOns{[r]}$ where $[r]$ is the lowest common multiple of $[i]$ and $[j]$? 
\end{Question}

The proposition below addresses this question by showing that in the case where the map $\rsig$ is unto the group of units of $\Z_{n-1}$, then Question~\ref{Question:canjoinbereplacedwithubgroupgenerated} has an affirmative answer.

\begin{proposition}
Let $\T{X}= \T{TO}, \T{O}$ and  suppose the map $\rsig: \XOn  \to  \Z_{n-1}$ is unto the group of units of $\Z_{n-1}$. Then for $[i],[j] \in [Z_{n}^0]$, $\gen{\XOns{[i]},\XOns{[j]}} = \XOns{[r]}$ for $[r]$ the lowest common multiple of $[i]$ and $[j]$.
\end{proposition}
\begin{proof}
The result is essentially a consequence  of the following claim, as we have not been able to find a reference for it, we also provide a proof.

\begin{claim}\label{Claim:groupgeneratedbypairequallcmgroup}
Let $n \in \N_{2}$ and consider the integer ring $\Z_{n}$. Let  $\Z_{n}^{*}$ denote the group of units, and for each $i \in \Z_{n}$ let $(\Z_{n}^{*})_{i} := \{ a \in \Z_{n}^{*} \mid ai \equiv i \mod{n}\} \le \Z_{n}^{*}$. Let $i,j \in \Z_{n}$ let $i_1 = \gcd(i,n)$, $j_1 = \gcd(j,n)$ and $r = lcm(i_1,j_1)$ then $ \gen{ (\Z_{n}^{*})_{i_1}, (\Z_{n}^{*})_{j_1}} = \gen{ (\Z_{n}^{*})_{i}, (\Z_{n}^{*})_{j}} = (\Z_{n}^{*})_{r}$.
\end{claim}
\begin{proof}
We first observe that for $i \in \Z_{n}$ the group $(\Z_{n}^*)_{i} = (\Z_{n}^*)_{d}$ where $d = \gcd(i, n)$. This is because for any $a \in \Z_{n}^{*}$ such that $ad \equiv d \mod{n-1}$ then $ai \equiv i \mod{n-1}$ as well. If, on the other hand $ai \equiv i \mod{n-1}$, then observe that, as, by Bezout's lemma, there is a $u \in \Z_{n}$ such that $iu \equiv d \mod{n-1}$, then $ad \equiv d \mod{n-1}$ also.

Thus let $i, j \in \Z_{n}$ be divisors of $n$ with $r = lcm(i,j)$ (a divisor of $n$). First assume that $n = p^{\alpha}$ for a prime $p$. Then $i = p^{\beta}$ and $j = p^{\gamma}$ for $\beta, \gamma \le \alpha$. Without loss of generality we may assume that $i \le j$, and so, in particular, that $i$ divides $j$. However, it then follows that $(\Z_{n}^{*})_{i} \le (\Z_{n}^{*})_{j}$, from which we conclude that $\gen{ (\Z_{n}^{*})_{i}, (\Z_{n}^{*})_{j}} = (\Z_{n})^{*}_{j}$ noting that $lcm(i,j) = j$.

Now suppose that $n= p_1^{\alpha_{1}}p_2^{\alpha_{2}}\ldots p_m^{\alpha_{m}}$ where the  $p_{l}$'s, $1 \le l \le m$ are distinct primes. We may further assume that $\gcd(i,j) = 1$. This is because if $\gcd(i,j) = d$, then setting $i_1 = i/d$ and $j_1 = j/d$ we observe that $(\Z_{n}^{\ast})_{i_1} \le (\Z_{n}^{\ast})_{i}$, $(\Z_{n}^{\ast})_{j_1} \le (\Z_{n}^{\ast})_{j}$ and $lcm(i,j) = lcm(i_1, j_1)$. Thus, noting that $\gen{ (\Z_{n}^{\ast})_{i}, (\Z_{n}^{\ast})_{j}} \le \gen{(\Z_{n}^{\ast})_{lcm(i,j)}}$, the result holds for $i_1$ and $j_1$ precisely if it holds for $i$ and $j$.  

Making use of the Chinese remainder theorem, there is a ring isomorphism from $\Z_{n} \to \Z_{p_1^{\alpha_{1}}} \times \Z_{p_2^{\alpha_{2}}} \times \ldots \times \Z_{p_m^{\alpha_{m}}}$ defined by $k \mapsto (k\mod{p_1^{\alpha_1}},k\mod{p_2^{\alpha_1}}, \ldots, k\mod{p_m^{\alpha_m}})$. For $1 \le l \le m$ let $i_l = i \mod{p_l^{\alpha_{m}}}$ likewise define the sequence $j_{l}$ and $r_l$, where $r = ij$. It  follows that $(\Z_{n}^{\ast})_{i} \cong (\Z_{p_1^{\alpha_{1}}}^{\ast})_{i_1} \times (\Z_{p_1^{\alpha_{2}}}^{\ast})_{i_2} \times \ldots \times (\Z_{p_m^{\alpha_{m}}}^{\ast})_{i_m}$, likewise $(\Z_{n}^{\ast})_{j} \cong (\Z_{p_1^{\alpha_{1}}}^{\ast})_{j_1} \times (\Z_{p_1^{\alpha_{1}}}^{\ast})_{j_2} \times \ldots \times (\Z_{p_m^{\alpha_{m}}}^{\ast})_{j_m}$. Observe that $$\gen{\left((\Z_{p_1^{\alpha_{1}}}^{\ast})_{i_1} \times (\Z_{p_2^{\alpha_{2}}}^{\ast})_{i_2} \times \ldots \times (\Z_{p_m^{\alpha_{m}}}^{\ast})_{i_m}\right),\left((\Z_{p_1^{\alpha_{1}}}^{\ast})_{j_1} \times (\Z_{p_2^{\alpha_{2}}}^{\ast})_{j_2} \times \ldots \times (\Z_{p_m^{\alpha_{m}}}^{\ast})_{j_m}\right) }$$ is precisely the group  $$\gen{(\Z_{p_1^{\alpha_{1}}}^{\ast})_{i_1}, (\Z_{p_1^{\alpha_{1}}}^{\ast})_{j_1}} \times \gen{(\Z_{p_2^{\alpha_{2}}}^{\ast})_{i_2}, (\Z_{p_2^{\alpha_{2}}}^{\ast})_{j_2}} \times \ldots \times \gen{(\Z_{p_m^{\alpha_{m}}}^{\ast})_{i_m}, (\Z_{p_m^{\alpha_{m}}}^{\ast})_{j_m}}.$$ Since $i$ and $j$ are divisors of $n$ then $i = p_1^{\beta_1}p_2^{\beta_2} \ldots p_{m}^{\beta_m}$, $j= p_1^{\gamma_1}p_2^{\gamma_2} \ldots p_{m}^{\gamma_m}$ where for $1 \le l \le m$, $0 \le \beta_l, \gamma_{l} \le \alpha_{l}$ and $\gamma_{l}+ \beta_{l} = \max\{\gamma_{l}, \beta_{l}\}$ (since $i,j$ are assumed co-prime). It therefore follows that $i \mod p_l^{\alpha_{l}}$ is either a unit  of $\Z_{p_l^{\alpha_{l}}}$ otherwise it is  equal to a unit of $\Z_{p_l^{\alpha_{l}}}$ times a non-trivial (appropriate) power of $p_l$ and likewise for $j$. We note that if one of $i_l$ or $j_l$ is equal to  a unit of $\Z_{p_l^{\alpha_{l}}}$ times a non-trivial (appropriate) power of $p_l$, then the other must be equal to a unit of $\Z_{p_l^{\alpha_{l}}}$. Therefore for $1 \le l \le m$, $\gen{ (\Z_{p_l^{\alpha_{l}}}^{\ast})_{i_l}, (\Z_{p_l^{\alpha_{l}}}^{\ast})_{j_l} }$ is either equal to the trivial group or it is equal to $(\Z_{p_l^{\alpha_{l}}}^{\ast})_{p_l^{\delta}}$ where $p_l^{\delta} \ne 1$ is the maximum power of $p_l$ dividing $ij$. In the case that $\gen{ (\Z_{p_l^{\alpha_{l}}}^{\ast})_{i_l}, (\Z_{p_l^{\alpha_{l}}}^{\ast})_{j_l} }$ is the trivial group, then $i_l$ and $j_l$ are both units of $\Z_{p_l^{\alpha_{l}}}$ and so  $r_l = i_lj_l \mod p_{l}^{\alpha_{l}}$ is also a unit of $\Z_{p_l^{\alpha_{l}}}$. In the case that $\gen{ (\Z_{p_l^{\alpha_{l}}}^{\ast})_{i_l}, (\Z_{p_l^{\alpha_{l}}}^{\ast})_{j_l} }$ is equal to $(\Z_{p_l^{\alpha_{l}}}^{\ast})_{p_l^{\delta}}$ where $p_l^{\delta} \ne 1$ is the maximum power of $p_l$ dividing $ij$, then it follows that one of $i_l$ or $j_l$ is a unit and the other is equal to a unit times $p_{l}^{\delta}$. In this case we have that $r_l$ is equal to a unit time $p_{l}{\delta}$ and so $(\Z_{p_l^{\alpha_{l}}}^{\ast})_{p_l^{\delta}} = (\Z_{p_l^{\alpha_{l}}}^{\ast})_{r_l}$. In either case we see that $(\Z_{p_l^{\alpha_{l}}}^{\ast})_{r_l} = \gen{ (\Z_{p_l^{\alpha_{l}}}^{\ast})_{i_l}, (\Z_{p_l^{\alpha_{l}}}^{\ast})_{j_l} }$. Thus we conclude that that $\gen{(\Z_{n}^{\ast})_{i}, (\Z_{n}^{\ast})_{j}} = (\Z_{n}^{\ast})_{r}$ as required.
\end{proof}

We can now prove the proposition. We observe that for $1 \le r < n$ since $\rsig$ is surjective, then $(\XOnr)\rsig = (\Z_{n-1}^*)_{r}$. Let $1 \le i,j \le  n-1$ such that $i$ the the minimal element of $[i]$ and $j$ is the minimal element of $[j]$ so that both $i$ and $j$ are divisors of $n$. Let $r = lcm(i,j)$. By Claim~\ref{Claim:groupgeneratedbypairequallcmgroup} we have that $\gen{(\Z_{n-1}^*)_{i}, (\Z_{n-1}^*)_{j}} = (\Z_{n-1}^*)_{r}$. Thus, let $T \in \XOnr$ be any element. By Proposition~\ref{Proposition:sigdeterminesmembership} $(T)\rsig \in (\Z_{n-1}^*)_{r}$ and so there are elements $u \in (\Z_{n-1}^*)_{i}$ and $v \in (\Z_{n-1}^*)_{j}$ such that $(T)\rsig = uv \mod{n-1}$. Let $U \i \XOns{i}$ and $V \in \XOns{j}$ be such that $u= (U)\rsig$ and $v = (V)\rsig$, then $(TU^{-1}V^{-1})\rsig = 1$ and so $TU^{-1}V^{-1} \in \XOns{1} \le  \gen{\XOns{i}, \XOns{j}}$. It therefore follows that $\XOns{r}\le \gen{\XOns{i}, \XOns{j}}$. Proposition~\ref{Proposition:sigdeterminesmembership} guarantees that $\gen{\XOns{i}, \XOns{j}} \le \XOns{r}$ Thus by Proposition~\ref{Proposition:partiallatticestructure} $[r]$ is the lowest common multiple of $[i]$ and $[j]$.
\end{proof}

We conclude our investigation of the nesting properties of the groups $\TOns{r}$ for $1 \le r < n$ by making  use of Proposition~\ref{Proposition:sigdeterminesmembership} to construct an element of $\T{TO}_{4}$ which is not an element of $\T{TO}_{4,r}$ for any $1 \le r \le 3$. This indicates that in general the group $\TOns{r}$ depends on $r$.

Let $g$ be the transducer below:

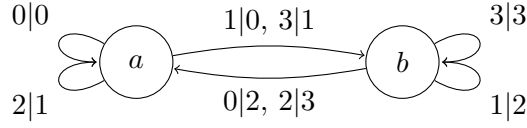
\begin{figure}[H]
\centering
\begin{tikzpicture}[shorten >= .5pt,node distance=3cm,on grid,auto] 
   \node[state] (q_0)   {$a$};
   \node[state, xshift=0.5cm] (q_1) [right=of q_0] {$b$}; 
    \path[->] 
    (q_0) edge[in=170, out=10]  node {$1|0$, $3|1$} (q_1)
          edge[in=180, out=150,loop] node[swap] {$0|0$} ()
          edge[in=180, out=210, loop] node {$2|1$} () 
    (q_1) edge[in=350, out=190] node {$0|2$, $2|3$} (q_0)
          edge[in=0, out=30,loop] node {$3|3$}()
          edge[in=0, out=330,loop] node[swap] {$1|2$}();
\end{tikzpicture}
\caption{An element  $g \in \T{TO}_{4}$ which is not in $\T{TO}_{4,r}$ for any $0 \le r <3$}.
\label{Figure:elementinTo4butnotinanyotherTOnr}
\end{figure}

We make the following observations about $g$ which the reader may verify:
\begin{enumerate}[label=(\roman*)]
\item  For all states $\alpha$ of $g$ the map $g_{\alpha}: \CCn \to \CCn$ preserves the lexicographic ordering on $\CCn$. Thus, since $g^2 = \id$ under the product defined for $\TOn$, by Theorem~\ref{Thm: Equivalent conditions for an element of $On$ to belong to TOnr} there is some non-zero $r \in \Z_{4}\backslash$ such that $g \in \TOnr$. \label{Observation:gisanelementofTOn}
\item We have $\im(a) \cap \im(b) = \emptyset$ and $\im(a)\sqcup\im(b) = \CCmr{4}$. In particular we  have $\im(a) = U_{0} \cup U_{1}$ and $\im_{b} = U_{2}\cup U_{3}$.\label{Observation:gneedstostatestofillacone}

\item  The following fact is not essential to our discussion here but is nevertheless worth mentioning. The paper forthcoming article \cite{BleakCameronOlukoya} shows that there is subgroup $\T{L}_{n} \le \On$ which is isomorphic to the quotient $\aut{\Xn^{\Z}, \sigma_{n}}/\gen{\sigma_{n}}$ of the automorphisms of the two-sided full shift on $n$ letters by the group generated by the shift ($\sigma_{n}$ denotes the shift map on $n$ letters). The transducer $g$ is in fact an element of $\T{L}_{4}$. 
\end{enumerate}

\begin{lemma}
Let $1 \le r \le 3$ and let $A_{q_0} \in \TBmr{4}{r}$ be such that $\core(A_{q_0})$ is equal to the transducer $g$ of Figure~\ref{Figure:elementinTo4butnotinanyotherTOnr}, then $r = 3$.
\end{lemma}
\begin{proof}

We compute $(g)\sig$. First observe that $g$ is synchronizing at level 1 and moreover the states of $g$ forced by $0$ and $2$, $q_{0}$ and  $q_{2}$ respectively are both equal to  $a$ and the states of $g$ forced by $1$ and $3$, $q_{1}$ and $q_3$ respectively are both equal to $b$. Further observe that $m_{a} = 2$ and $m_{b} = 2$ (where $m_{a}$ and $m_{b}$ are as in Definition~\ref{Definition:signature}). Therefore $(g)\sig = 2.4 = 8 \equiv 2 \mod 3$. The only number $r \in \Z_{4} \backslash \{0\}$ such that $ 2r \equiv r \mod{n-1}$ is $3$. Therefore by Proposition~\ref{Proposition:sigdeterminesmembership} we conclude that if $g \in \TBmr{4}{r}$ then $r$ must be equal to 3.

\end{proof}

In the next section we address the question of the surjectivity of the homomorphism $\rsig$.

\section{On the surjectivity of \texorpdfstring{$\rsig$}{Lg}}\label{Section:onsurjectivityofrsig}

In this section we show that there are infinitely many numbers $n \in \N$ such that the  map $\rsig$ from $\On$ to the group of units of $\Z_{n-1}$ is surjective. More specifically, for $n \in \N$ we show that for every  divisor $d$  of $n$, $1 \le d <n$, there is an element $T \in \On$ with $(T)\rsig = d$. We then argue that there are infinitely many numbers $n \in \N$ such that the divisors of $n$ generate the group of units of $\Z_{n-1}$.

We first define some subgroups of $\On$ highlighted in the paper \cite{BCMNO} and outline an algorithm, `the collapsing procedure', given in  that paper for determining whether or not an automata (or transducer) is synchronizing.

\begin{Definition}
Let $\T{L}_{n}$ be the subgroup of $\On$ consisting of all elements of $T \in \On$, such that if $q \in Q_{T}$ and $\Gamma \in \Xns$ satisfies, $\pi_{T}(\Gamma, q) = q$, then $|\lambda_{T}(\Gamma, q)| = |\Gamma|$. Let $\hn{n}$ be the subgroup of $\T{L}_{n}$ consisting of synchronous, invertible transducers. 
\end{Definition}

Let $A$ be an automaton, the collapsing procedure constructs a new automaton $A_1$ as follows. For each $q \in Q_{A}$ set $[q] := \{p \in Q_{A} \mid \pi_{A}(x, p) = \pi_{A}(x, q) \mbox{ for all } x \in \Xn\}$ and set $Q_{A_1}:= \{ [q] \mid q \in Q_{A}\}$. Let $\pi_{A_1}: \Xn \times Q_{A_1} \to Q_{A_1}$ be defined by  $\pi_{A}(x, [q])  = [\pi_{A}(x,q)]$ for all $x \in \Xn$ and let $A_1= \gen{\Xn, Q_{A_1}, \pi_{A_1}}$. Let $A = A_0$, $A_1, A_2 \ldots $ be a sequence of automaton, where $A_{l+1}$, for $l \in \N$ is obtained from $A_{l}$ by applying the collapsing procedure. We observe that for $l \in \N$ either $|A_{l}| > |A_{l+1}|$ or $A_{l} = A_{l+1}$. Thus let $k$ be minimal in $\N$ such that $A_{k} = A_{k+1}$. The automaton $A$ is synchronizing if and only if $A_{k}$ is the single state automaton over $\Xn$ (\cite{BCMNO}). If $T = \gen{\Xn, Q_{T}, \pi_{T}, \lambda_{T}}$ is a transducer, then $T$ is synchronizing if and only if the automaton  $\T{A}(T) = \gen{\Xn, Q_T, \pi_T}$ is synchronizing. Typically we when applying the collapsing procedure, we shall not distinguish between  $T$ and $\T{A}(T)$.

We have the following general construction, which is in some sense an extension of a construction given in the paper \cite{Olukoya1} for embedding direct sums of the group of automorphisms of one-sided shift over alphabets sizes summing to $n$ into the group of automorphisms of the one-sided shift on $n$ letters.

Let $n \in \N_{2}$ and $d$ be a divisor of $n$ not equal to $n$. Let $m \in \N$ be such that $md =n$ and fix an element $T \in \hn{d}$. For $0\le i \le m-1$, partition the set $\Xn$ into sets $X_{n,i}:= \{di, di+1, \ldots, d(i+1)-1 \}$, likewise form sets $Q_{T,i} = \{ q(i) \mid q \in Q_{T}\}$. Form a transducer $\oplus_{d}T = \gen{ \Xn, \cup_{0 \le i \le m-1} Q_{T,i}, \pi_{\oplus_{d}T}, \lambda_{\oplus_{d}T}}$ with transition and output defined such that:

\begin{enumerate}[label=(\arabic*.)]
\item the restriction $\oplus_{d}T\restriction_{Q_{T,i}} = \gen{X_{n,i}, Q_{T,i},\pi_{\oplus_{d}T}\restriction_{Q_{T,i}}, \lambda_{\oplus_{d}T}\restriction_{Q_{T,i}}}$ of $\oplus_{d}T$ to the set of states $Q_{T,i}$ is a transducer equal to $T$ up to relabelling the alphabet and the set of states,

\item for $0 \le i,j \le m-1$ such that $i \ne j$, for any $b \in X_{d}$, and for any $q \in Q_{T}$,  $\pi_{\oplus_{d}T}(dj +b, q(i)) = (q_{b})(j)$, where $\pi_{T}(b, q_{b})= b$, and  $\lambda_{\oplus_{d}T}(dj+b, q(i)) = di+b$.

\end{enumerate}

We have the following result:

\begin{proposition}\label{Propositon:forncompositeforeverydivisorofnthereisanelementwithsigequaltodivisor}
Let $n \in \N_{2}$, $d$ be a divisor of $n$ such that $n= md$ for some $m \in \N_{2}$, and $T \in \hn{d}$. Then, the transducer $\oplus_{d}T$ is bi-synchronizing, and is fact an element of $\Ln{n}$. Moreover, for $0 \le i \le m-1$, and any state $q(i) \in Q_{T,i}$, $\bigcup_{0 \le b \le d-1} U_{di+b}=\im(q_i)$. In particular, $(\oplus_{d}T)\rsig = d$.
\end{proposition}
\begin{proof}
We begin by arguing that the transducer  $\oplus_{d}T$ is synchronizing. However, this follows straight-forwardly from the observation that for $0 \le i,j \le m-1$ such that $i \ne j$, any pair of states $q(i), p(i) \in  Q_{T,i}$, and any $0 \le b < d$, $\pi_{\oplus_{d}T}(dj+b, q(i)) = \pi_{\oplus_{d}T}(dj+b, p(i)) = (q_b)(i)$. Thus since, each transducer $\oplus_{d}(T)\restriction_{Q_{T,i}}$ for $0 \le i \le m-1$ at the same level $k \in \N$, then after applying the collapsing procedure at most $k+1$ times we have reduced $\oplus_{d}T$ to the single state automaton.

Further observe that as $T$ is in fact a synchronous transducer, then if $T \in \On$ it is in fact an element of $\Ln{n}$.

Therefore in order to show that $\oplus_{d}T \in \Ln{n}$ it suffices to  show that each state is injective and has clopen image and that it has a synchronizing inverse.

 We begin by showing that each state of $\oplus_{d}T$ is injective and has clopen image. Let $0 \le i \le m-1$ and $q(i)$ be a state of $T$. Let  $\mu, \nu \in \Xnp$ such that $|\mu| = |\nu|  \ge 2$ and $\mu \ne \nu$. We now show that either $\lambda_{\oplus_{d}T}(\mu, q(i)) \ne \lambda_{\oplus_{d}T}(\nu, q(i))$ or for any $x \in \Xn$, $\lambda_{\oplus_{d}T}(\mu, q(i)) \ne \lambda_{\oplus_{d}T}(\nu, q(i))$. If $\mu, \nu \in X_{n,i}^{+}$, then this follows from the fact that $T \in \hn{n}$ and so each state of $T$ induces a permutation of $X_{d}$. Thus, we may assume that there are $1 \le j_1, j_2  \le m-1$ with $j_1$ and $j_2$ not both equal to  $i$, $a_1 \in X_{n, j_1}$, $a_2 \in X_{n, j_2}$, and  $\mu_1, \nu_1 \in \Xn^{+}$ such that $\mu = a_1 \mu_1$ and $\nu= a_2 \nu_1$. If $j_2 \ne j_2$, then as $\pi_{\oplus_{d}T}(a_1, q(i))  \in Q_{T, j_1}$ and $\pi_{\oplus_{d}T}( a_2, q(i)) \in Q_{T, j_2}$, it follows that the first letter of $\lambda_{\oplus_{d}T}(\mu_1, \pi_{\oplus_{d}T}(a_1, q(i)))$ is an element of $X_{n, j_1}$ and the first letter of $\lambda_{\oplus_{d}T}(\nu_1, \pi_{\oplus_{d}T}(a_2, q(i)))$ is an element of $X_{n, j_2}$. Thus $\lambda_{\oplus_{d}T}(\mu, q(i)) \ne \lambda_{\oplus_{d}T}(\nu, q(i))$. Therefore we may assume that $j_1 = j_2 = j$. In this case, since $\mu \perp \nu$, there are words $\mu_2, \nu_2, \phi \in \Xn^{\ast}$ and $a \ne b \in \Xn$ such  that $\mu= \phi a \mu_2$ and $\nu = \phi b \nu_2$. If $a, b \in X_{n,l}$ for some $0 \le l \le m-1$, then since for any state $t \in \cup_{0\le i \le m-1} Q_{T,i}$, $\lambda_{\oplus_{d}T}(a,t) \ne \lambda_{\oplus_{d}T}(b, t)$, we have $\lambda_{\oplus_{d}T}(\mu, q(i)) \ne \lambda_{\oplus_{d}T}(\nu, q(i))$. Thus, we may assume that $a \in X_{n, l_1}$, $b \in X_{n,l_2}$, by adding a letter to  $\mu_2$ and $\nu_2$ if necessary, we may assume they are non-empty. In this case, the fact that $\lambda_{\oplus_{d}T}(\mu, q(i)) \ne \lambda_{\oplus_{d}T}(\nu, q(i))$ follows from the observation above that for any state $t \in \cup_{0\le i \le m-1} Q_{T,i}$, the first letter of $\lambda_{\oplus_{d}T}(\mu_2, \pi_{\oplus_{d}T}(a,t))$ is an element of $X_{n,l_1}$ and the first letter of $\lambda_{\oplus_{d}T}(\nu_2, \pi_{\oplus_{d}T}(b,t))$ is an element of $X_{n,l_2}$. Therefore, since $T$ is synchronous, each state must induce an injective map from $\CCn$ to itself. 
 
 Let $q(i) \in Q_{T,i}$ be an arbitrary state of $T$, $b \in \Z_{d}$ be arbitrary. We now show, by induction, that $U_{di+b} \subset \im(q_i)$. 
 
 Let $j \in \Z_{m}$ and $a \in \Z_{d}$ be arbitrary so that $jd +a \in X_{n,j} \subset \Xn$ is arbitrary. There is a letter $x \in X_{n,j}$ such that $\lambda_{\oplus_{d}T}(x, q(i)) = di+b$ and $\pi_{\oplus_{d}T}(x, q(i)) \in Q_{T, j}$. Since $\pi_{\oplus_{d}T}(x, q(i)) \in Q_{T, j}$, by construction, for any $l \in \Z_{m}$, there is a letter $y \in X_{n,l}$ such that $\lambda_{A}(y, \pi_{\oplus_{d}T}(x, q(i)))  = jd +a$ and $\pi_{\oplus_{d}T}(y, \pi_{\oplus_{d}T}(x, q(i))) \in Q_{T,l}$. Thus, for any $j \in \Z_{m}$, $a \in \Z_{d}$  and any $l \in   \Z_{m}$, there is a word $xy \in \Xn^{2}$ such that $\pi_{\oplus_{d}T}(xy, q(i)) \in Q_{T,l}$ and $\lambda_{\oplus_{d}T}(xy, q(i)) = (di+b) (jd+a)$.
 
 Assume by induction that there for any $\nu \in \Xn^{M}$ and any $l \in \Z_{m}$, there is a word $\mu \in \Xn^{M+1}$, such that $\lambda_{\oplus_{d}T}(\mu, q(i)) = (di+b) \nu$ and $\pi_{\oplus_{d}T}(\mu, q(i)) \in Q_{T,l}$. Let $\xi \in \Xn^{M+1}$ be arbitrary. Let $j \in \Z_{m}$, $a \in \Z_{d}$ be such that $ \xi = \chi(jd + a)$ for some $\chi \in \Xn^{M}$. By the inductive assumption, there is a word $\zeta \in \Xn^{M+1}$ such $\lambda_{\oplus_{d}T}(\zeta, q(i)) = (di+b) \chi$ and $\pi_{\oplus_{d}T}(\zeta, q(i)) = p(j) \in Q_{T,j}$. Now by construction, for any $l \in \Z_{m}$, there is a word $y \in X_{n,l}$ such that $\lambda_{\oplus_{d}T}(y, p(j)) = dj + a$ and  $\pi_{\oplus_{d}T}(y, p(j)) \in Q_{T,l}$. Therefore, we have that for any $l \in \Z_{m}$, there is a $y \in X_{n,l}$ such that $\lambda_{\oplus_{d}T}(\zeta y, q(i)) = (di+b) \chi (dj+a)$ and  $\pi_{\oplus_{d}T}(\zeta y, q(i)) \in Q_{T,l}$. Therefore we see, by transfinite induction, that $U_{di+b} \subset \im(q_i)$. Therefore every state of $T$ is injective and has clopen image.
 
 We now show that $\oplus_{d}T$ has an inverse in $\On$. We do this by showing that for any transducer $A_{q_0}$, with $\core(A_{q_0}) = \oplus_{d}T$, then set of states of $A_{(\epsilon, q_0)}$ which are reached by arbitrarily long words, form a synchronizing transducer. To this end, fix a transducer $A_{q_0}$ with $\core(A_{q_0}) = \oplus_{d}T$. 
 
 Let $0 \le i \le m-1$ and $q(i) \in  Q_{T,i}$, we observe that for $a \in \Z_{d}$, $(di + a)L_{q(i)} = \epsilon$. Set $Q_{U} = \{ (di + a, q(i)) \mid i \in \Z_{m}, a \in \Z_{d} \}$. Define  $\lambda_{U}:  \Xn \times Q_{U} \to \Xn$ by $\lambda_{U}(dj+ b, (di+a, q(i))) = ((di+a)(dj+b))L_{q(i)}$ for all $i, j \in \Z_{m}$ and $a,b  \in \Z_{d}$. Further define $\pi_{U}: \Xn \times Q_{U} \to Q_{U}$ by $\pi_{U}(dj+ b, (di+a, q(i))) = ((di+a)(dj+b))L_{q(i)} = ( (di+a)(dj+b) - \lambda_{A}(((di+a)(dj+b))L_{q(i)}, q(i)), \pi_{A}(((di+a)(dj+b))L_{q(i)}, q(i)))$. Now, observe that $((di+a)(dj+b))L_{q(i)}$ is precisely the element $x \in X_{n,j}$ such that $\lambda_{\oplus_{d}T}(x, q(i)) = di +a$, thus if $p(j) = \pi_{\oplus_{d}T}(x, (q(i)))$, then $\pi_{U}(dj+ b, (di+a, q(i))) = (dj+b, p(j))$. Let $U = \gen{ \Xn, Q_{U}, \pi_{U}, \lambda_{U} }$. We now show that $U$ is synchronizing, then we argue that the set of states of $A_{(\epsilon, q_0)}$ reached by arbitrarily long words forms a transducer precisely equal to $U$.
 
 We partition the sets of $U$ as follows: for $0 \le i \le m-1$ and $a \in  \Z_{d}$, set $Q_{U,i,a}:= \{ (di+ a, q(i)) \mid  q(i) \in Q_{T,i} \}$. Fix $0 \le i \le m-1$ an $a \in \Z_{d}$. 
 
 Let $0 \le j \le n-1$ such that $i \ne j$ and $(di+ a, q(i)), (di+ a, p(i)) \in Q_{U,i,a}$  and $dj +b$, $b \in \Z_{d}$, be  arbitrary. We observe that, by construction, if $x, y \in X_{n,j}$ satisfy $\lambda_{\oplus_{d}T}(x, q(i)) = \lambda_{\oplus_{d}T}(y, p(i)) = di+a$, then $x = y = dj +a$. Observe, moreover, that $\pi_{\oplus_{d}T}(dj + a, q(i)) = \pi_{\oplus_{d}T}(dj + a, p(i)) = (q_{a})(j)$. Therefore it follows that $\pi_{U}(dj+ b, (di+a, q(i))) = (dj+b, (q_a)(j)) = \pi_{U}(dj+ b, (di+a, p(i)))$.  Thus we conclude that for any $0 \le i \le m-1$ and any $a \in \Z_{d}$, the set of elements $Q_{U,i,a} := \{ (di+a, q(i)) \mid q(i) \in Q_{T,i}\}$ transition identically on all elements of $X_{n,j}$ for $j \ne i$.

  Now let, $T^{-1} = \gen{ \Xn, Q_{T^{-1}}, \pi_{T^{-1}}, \lambda_{T^{-1}}}$, a synchronous and synchronizing transducer,  denote the inverse of $T$ and denote by $(a)q^{-1}$, for $a \in  \Z_{d}$, the element of $X_{d}$  such that $\lambda_{T}((a)q^{-1}, q) = a$. Consider a letter $di +c$, we observe that $\pi_{U}(di + c, (di+a,q(i))) = (di + c, \pi_{\oplus_{d}T}(di+(a)q^{-1}, q(i)))$. Note moreover that $\pi_{T^{-1}}(a, q^{-1}) = (\pi_{T}((a)q^{-1}, q))^{-1}$. Thus, $$(di + c, \pi_{\oplus_{d}T}(di+(a)q^{-1}, q(i))) = (di+c, (\pi_{T}((a)q^{-1}, q))^{-1}(i) ) = (di+c, (\pi_{T^{-1}}(a, q^{-1}))^{-1}(i)). $$ A simple induction argument now shows that given a word $\nu = v_1 v_2 \ldots v_{l} \in X_{d}^{+}$, with corresponding word $\nu(i) = (di + v+1) (di+v_2) \ldots  (di+ v_{l}) \in X_{n,i}^{+}$, if $\pi_{T^{-1}}(\nu, q^{-1}) = p^{-1}$, for a state $q \in Q_{T}$, then, setting $\nu(i)_{1} = (di+v_2) \ldots  (di+ v_{l})$,  $\pi_{U}(\nu(i)_{1}(di+b), (di+v_1,q(i))) = (di+b, p(i))$. Let $k$ be the synchronizing level of $T^{-1}$, it therefore follows that, all elements of  $Q_{U,i,a}$ transition identically on all words in $X_{n,i}^{k}$.
 
 Finally, since by an observation above we have that for any state $(dj+b, q(j)) \in Q_{U}$, and any letter $dl +c \in X_{n}$ $\pi_{U}(dl+c,(dj+b, q(j))) \in Q_{U,l,c}$, it therefore follows that $U$ is synchronizing at level $k+1$. 
 

To conclude the proof we have to demonstrate that $(\oplus_{d}T)^{-1} = U$. We do this by showing that $A_{(\epsilon, q_0)}$ is synchronizing and has core equal to $U$. Let $q(i)$ be an arbitrary state of $\oplus_{d}T$. Since $A_{q_0}$ is synchronizing, there is a word $\Gamma \in \Xn^{+}$ such that $\Delta = \lambda_{A}(\Gamma, q_0) \ne \epsilon$ and $\pi_{A}(\Gamma, q_0) = q(i)$. It therefore follows that, for $0 \le a \le d-1$, $(\Gamma (di + a))L_{q_0}= \Delta (di+a)L_{q}$ and so $\pi'_{A}(\Gamma (di+a), (\epsilon,q_0)) = (di+a, q(i))$. Thus we see that $A_{(\epsilon, q_0)}$ is synchronizing and has $\core(A_{(\epsilon, q_0)}) = U$. Therefore $T^{-1} = U$. This concludes the proof.      
\end{proof}

\begin{corollary}
Let $n \in \N_{2}$, then the image of the map $\rsig$ from  $\On$ to $\Z_{n-1}^{*}$, the group of units of $Z_{n-1}$, contains the subgroup of $\Z_{n-1}^{*}$ generated by the divisors of $n$. \qed  
\end{corollary} 
\begin{corollary}
Let $n \in \N_{2}$ be such that the group of units of $\Z_{n-1}$ is generated by the divisors of $n$, then the map $\rsig: \On \to \Z_{n-1}^{*}$ is surjective. \qed
\end{corollary} 

\begin{corollary}
There are infinitely many number $n \in \N$ such that the map $\rsig$ from $\On$ to the group of units of $Z_{n-1}$ is surjective.
\end{corollary}
\begin{proof}
It is an elementary result in number theory that for an $m \ge 1$, $2$ generates the group of units of $3^{m}$, thus we may take $n = 3^{m} + 1$.
\end{proof}

Thus, we are left with the question:

\begin{Question}
for $n \in \N_{2}$ and $p$ in the group $\Z_{n-1}^{\ast}$ of units of $\Z_{n-1}$ such that $p$ is not an element of the subgroup of $\Z_{n-1}$ generated by the divisors of $n$, is there an element $T \in \On$ such that $(T)\rsig = p$? 
\end{Question}

We observe that in order to answer this question, it suffices to show that for any element $p \in \Z_{n-1}$ co-prime both to $n$ and $n-1$, there is an element $T \in \On$ such that $(T)\rsig = p$.

\section{The group $\TOnr$ and $\TBnr$ for $n>2$ and $1\le r < n$ contain an isomorphic copy of Thompson's group $F$}\label{Section:OutTncontainscopyofF}

We now show that for $n >2$, $n \in N$, and for all valid $r$, the group $\TOns{r}$ is infinite. This extends the result of Brin and Guzm{\'a}n stating that the group $\out{\mathcal{T}_{n,n-1}}$ is infinite whenever $n>3$ (\cite{MBrinFGuzman}) to the groups $\out{\Tnr}$. We observe that the group $\out{T_{2}}$ is the cyclic group of order $2$ and so finite. By Theorem~\ref{Theorem:TOnisubsetofTOnjifidividesjinZnminus1} it suffices to demonstrate  that $\TOns{1}$ is infinite. First we have the following straight-forward lemma.

\begin{lemma}\label{Lemma:possessingahomeostatemeanslementsofTOn1}
Let $T \in \TOn$ be an element with a homeomorphism state, then $T \in \TOns{1}$.
\end{lemma}
\begin{proof}
Let $q \in Q_{T}$ be an homeomorphism state of $T$, then the initial transducer $T_{q}$ is an element of $\TBmr{n}{1}$.
\end{proof}

Therefore to show that $\TOns{1}$ is infinite, for $n>2$, it suffices to find an element of infinite order of $\TOn$ which has a homeomorphism state.

Let $T$ be the  element of $\TOn$ depicted in Figure~\ref{Figure:elementinTon1ofinfiniteorder}. In this figure we use the symbol $x$  to represent an element of $\Xn$ strictly greater than zero and less than or equal to $n-2$.

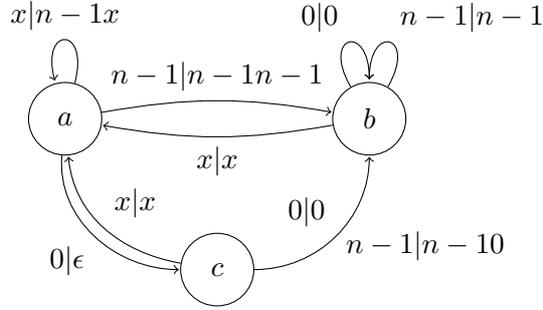
\begin{figure}[H]
\centering
\begin{tikzpicture}[shorten >= .5pt,node distance=3cm,on grid,auto] 
   \node[state] (q_0)   {$a$};
   \node[state, xshift=4cm] (q_1) {$b$}; 
   \node[state, xshift=2cm, yshift=-2cm] (q_2) {$c$};
    \path[->] 
    (q_0) edge[in=170, out=10] node {$n-1|n-1n-1$}  (q_1)
          edge[in=105, out=75,loop] node[swap] {$x|n-1x$} ()
          edge[in=180, out=265] node[swap] {$0|\epsilon$} (q_2) 
    (q_1) edge[in=350, out=190] node {$x|x$}(q_0)
          edge[in=90, out=120,loop] node {$0|0$}()
          edge[in=90, out=60,loop] node[swap] {$n-1|n-1$}()
    (q_2) edge[in=275, out=170] node[swap] {$x|x$} (q_0)
          edge[in=270, out=0] node[swap, yshift=0.2cm] {$n-1|n-10$} node {$0|0$} (q_1);
\end{tikzpicture}
\caption{An element  $T \in \T{TO}_{n,1}$ of infinite order.}
\label{Figure:elementinTon1ofinfiniteorder}
\end{figure}

The states $a$ and $b$ of $T$ are homeomorphism states, and so $T$ is in fact an element of $\TOns{1}$.

In order to show that $T$ has infinite order we make use of the action of the group $\On$, as introduced in the paper \cite{BCMNO}, on the space $\Xnz/ \gen{\sigma_{n}}$ where $\sigma_{n}$ is the shift map on $\Xnz$. Let $U \in \On$  and suppose that $k \in \N$ is the minimal synchronizing level of $U$. The action of $U$ on $\Xnz/\gen{\sigma_{n}}$ is given as follows: let $y= \ldots y_{-1}y_{-1}y_0y_1y_2\dots \in \Xnz$ represent an equivalence class of $\Xnz/\gen{\sigma_{n}}$, the image of this class under $U$ is the equivalence class of the bi-infinite word $w$ defined as follows: for $i \in \Z$, let $y(i+1,k) = y_{i+1}y_{i+1}\ldots y_{i+k}$ and $q_{y(i+1,k)}$ be the state of $U$ forced by $y(i+1,k)$, then set $w_i = \lambda_{U}(y_i,  q_{y(i+1,k)}$ and let $w = \ldots w_{-2}w_{-1}w_{-1}w_0w_1w_2\ldots$.  Consider the bi-infinite word $z:=\ldots(x n-1)(x n-1)\ldots$ where $ x \in \Xn \backslash \{0,n-1\}$, then $(z)T = \ldots(xn-1n-1)(xn-1n-1)\ldots$, $(z)T^2 = \ldots x(n-1)^3x(n-1)^3\ldots$, and $(z)T^{i} = \ldots x(n-1)^{i+1}x(n-1)^{i+1}\ldots$. Therefore we see that $z$ is on an infinite orbit under the action of $T$, demonstrating that $T$ is an element of infinite order. We have shown the following:

\begin{Theorem}\label{Theorem:Tnrinfinitewhenevernbiggerthan2}
For $n>3$ and $1 \le r \le n-1$, the group $T_{n,r}$ is infinite.
\end{Theorem}

Let $U$ be the  following transducer where, once more, $n >2$ and $x \in \Xn$ is any element strictly bigger than $0$ and strictly less than $n-1$:

\begin{figure}[H]
\centering
\begin{tikzpicture}[shorten >= .5pt,on grid,auto] 
   \node[state] (q_0)   {$p$};
   \node[state, xshift=4cm] (q_1) {$q$}; 
   \node[state, xshift=2cm, yshift=-4.3cm] (q_2) {$s$};
   \node[state, xshift=2cm, yshift=-2cm] (q_3) {$t$};
    \path[->] 
    (q_0) edge[in=170, out=10] node {$0|0$}  (q_1)
          edge[in=105, out=75,loop] node[swap] {$x|x$} ()
          edge[in=195, out=185] node[swap] {$n-1|n-1$} (q_2) 
    (q_1) edge[in=350, out=190] node {$x|n-1x$}(q_0)
          edge[in=10, out=260] node[swap] {$0|\epsilon$}(q_3)
          edge[in=0, out=0] node{$n-1|(n-1)^2$}(q_2)
    (q_2) edge[in=200, out=180] node[swap] {$x|x$}(q_0)
          edge[in=270, out=240,loop] node[swap, yshift=0.1cm] {$0|0$}()
          edge[in=270, out=300,loop] node[yshift=0.1cm] {$n-1|n-1$}()
    (q_3) edge[in=275, out=170] node[swap] {$x|x$} (q_0)
          edge node[swap] {$n-1|n-10$} node {$0|0$} (q_2);
\end{tikzpicture}
\caption{An element  $U \in \T{TO}_{n,1}$ of infinite order. }
\label{Figure:ThetransducerUinTOn1}
\end{figure}
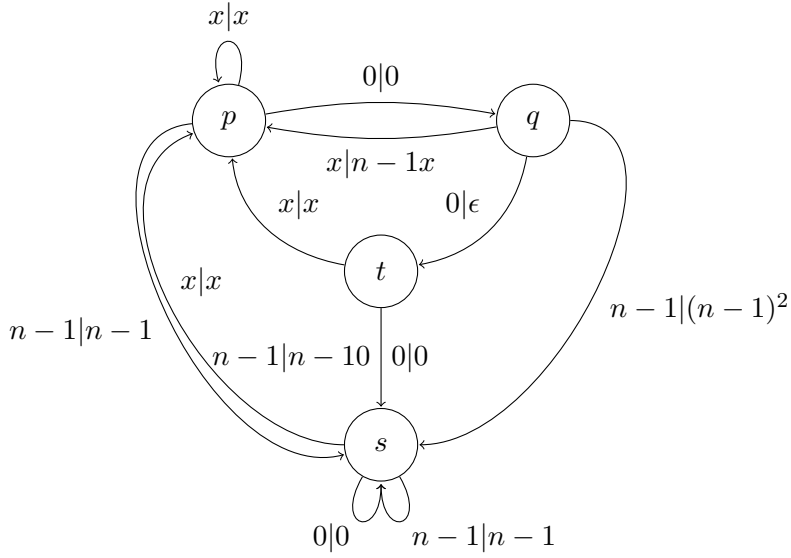

The state $p$ of $U$ is a homeomorphism state and  moreover $U$ is an element of $\TOns{1}$ of infinite order.

In order to state our next result we require the following notion and result from \cite{BCMNO}.

\begin{Definition}
Given an element $g = \gen{\Xn, Q_{g}, \pi_g, \lambda_g} \in \On$ a state $q$ of $g$ is called a \emph{loop state} if there is some $i \in \Xn$ such that $\pi_{g}(i, q) = q$.
\end{Definition}

The following lemma is proven in \cite{BCMNO}:

\begin{lemma}\label{Lemma: reading loops give bijection from set of Xns to Xns}
Let $g = \gen{\Xn, Q_g, \pi_g, \lambda_g} \in \On$ and let $w \in \Xnp$ then there is a unique state $q_{w} \in Q_{g}$ such that $\pi_{g}(w, q_w) = q_w$. Moreover for any periodic equivalence class of the tail equivalence $\sim_{t}$ with minimal period $w$ there exists a unique  $j \in \N_{1}$ and a unique circuit, decorated by some prime word $v$, in $A$ with output $w$. 
\end{lemma}

\begin{lemma}\label{Lemma:TandUgenerateF}
The elements $T$ and $U$ of $\TOns{1}$ depicted in Figures~\ref{Figure:elementinTon1ofinfiniteorder} and \ref{Figure:ThetransducerUinTOn1} generate a subgroup of $\TOns{1}$ isomorphic to R. Thompson's group $F$.
\end{lemma}
\begin{proof}
First we make some observations about $T$ and $U$. Let $q$ be a state of $T$ or of $U$,let $x \in \Xn \backslash \{0, n-1\}$, and let $(D,d) \in \{(T,a),(U,p)\}$, then all transitions $\pi_{D}(x,q)  = d$ and $\lambda_{D}(x,q)$ end in the symbol $x$. The output of all other transitions in $T$ or $U$ do not involve the symbol $x$. Therefore the transducers $A$ and $B$ below are in fact sub-transducers of $T$ and  $U$ respectively.

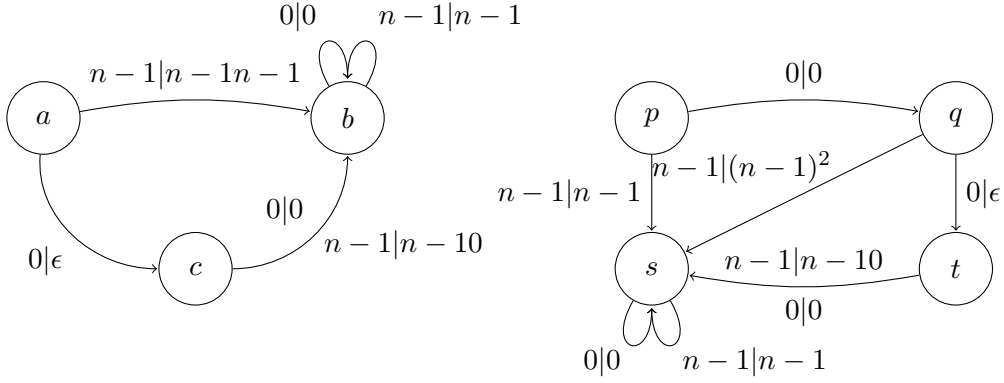
\begin{figure}[H]
\centering
\begin{tikzpicture}[shorten >= .5pt,node distance=3cm,on grid,auto] 
   \node[state] (q_0)   {$a$};
   \node[state, xshift=4cm] (q_1) {$b$}; 
   \node[state, xshift=2cm, yshift=-2cm] (q_2) {$c$};
   \node[state, xshift = 8cm] (p_0)   {$p$};
   \node[state, xshift=12cm] (p_1) {$q$}; 
   \node[state, xshift=8cm, yshift=-2cm] (p_2) {$s$};
   \node[state, xshift=12cm, yshift=-2cm] (p_3) {$t$};
    \path[->] 
    (q_0) edge[in=170, out=10] node {$n-1|n-1n-1$}  (q_1)
          edge[in=180, out=265] node[swap] {$0|\epsilon$} (q_2) 
    (q_1) edge[in=90, out=120,loop] node {$0|0$}()
          edge[in=90, out=60,loop] node[swap] {$n-1|n-1$}()
    (q_2) edge[in=270, out=0] node[swap, yshift=0.2cm] {$n-1|n-10$} node {$0|0$} (q_1)
    (p_0) edge[in=170, out=10] node {$0|0$}  (p_1)
          edge node[swap] {$n-1|n-1$} (p_2) 
    (p_1) edge node {$0|\epsilon$}(p_3)
          edge node[swap, xshift=0.5cm]{$n-1|(n-1)^2$}(p_2)
    (p_2) edge[in=270, out=240,loop] node[swap, yshift=0.1cm] {$0|0$}()
          edge[in=270, out=300,loop] node[yshift=0.1cm] {$n-1|n-1$}()
    (p_3) edge[out=190, in=350] node[swap] {$n-1|n-10$} node {$0|0$} (p_2);
\end{tikzpicture}
\caption{The subtransducers $A$ and $B$ of $T$ and $U$ respectively.}
\label{Figure:subtransducersAandB}
\end{figure}

Notice that $A_{a}$ and $B_{p}$ induce self-homeomorphisms of the Cantor space $\{0,n-1\}^{\omega}$ equal, respectively, to the restrictions of the homeomorphisms $T_{a}$ and $U_{p}$ to the space $\{0,n-1\}^{\omega}$. It is not hard to see that $F \cong \gen{A_{a}, B_{p}} \le H(\{0,n-1\}^{\omega})$. In fact $A_{a}^{-1}$ and $A_{a}^{-2}B_{p}A_{a}$ are the standard generators for $F$ acting on the space $\{0,n-1\}^{\omega}$. Moreover $A_{a}$ and $B_{p}$ satisfy the defining relations for $F$, and so we see that $\gen{A_{a}, B_{p}} = \gen{A_{a}, B_{p}\mid [B_{p}^{-1}A_{a}, A_{a}B_{p}A_{a}^{-1}], [B_{p}^{-1}A_{a}, A_{a}^{2}B_{p}A_{a}^{-2}]}$. Moreover $T_{a}^{-1} \restriction_{\{0, n-1\}^{\omega}} = A_{a}^{-1}$ and  $U_{p}^{-1}\restriction_{\{0, n-1\}^{\omega}} = B_{p}^{-1}$. We also observe (one can verify this by direct computation) that the transducers $T^{-1}$ and $U^{-1}$ representing the inverses of $T$ and $U$ in $\TOn$ respectively, have states $a^{-1}, b^{-1}$ of $T^{-1}$ and $p^{-1}, s^{-1}$ of $U^{-1}$ such that $T_{a^{-1}}^{-1} = T_{a}^{-1}$, $T_{s^{-1}}^{-1} = T_{s}^{-1}$, $U_{p^{-1}}^{-1} = U_{p}^{-1}$ and $U_{s^{-1}}^{-1} = U_{s}^{-1}$. 

We now define a homomorphism from $\gen{T, U}$ to $\gen{A_{a}, B_{p}}$. To do this we make the following observation. Let $W(T,U) = w_1 w_2 \ldots w_l$ be a word in $\{T,U,T^{-1},U^{-1}\}^{+}$. For each letter $w_i$ of $W(T,U)$ let $\alpha_{i}$ be the state of $w_i$ so that $\alpha_{i}$ is the unique loop state of $x$ in $w_i$. Observe that, for all $1 \le i \le l$, $\alpha_i \in \{a,a^{-1}, p, p^{-1}\}$. Let $S_{W(T,U)}$ be the element of $\TOns{1}$ representing the word $W(T,U) \in \gen{T,U}$. Under the product defined on  $\TOns{1}$, we have that $S_{W(T,U)}$ is $\omega$-equivalent to the core of the minimal transducer representing the product of the initial transducers $(w_{1}\ast w_{2} \ast \ldots \ast w_{l})_{(\alpha_{1},\alpha_{2},\ldots,\alpha_{l})}$. However, since $\alpha_i$ is the unique loop state state of $x$ in $w_i$, $\lambda_{w_i}(x, \alpha_i)$ ends in $x$ and for any state $q$ of $T$, $U$, $T^{-1}$ or $U^{-1}$, the resulting state when $x$ is read from $q$ is the appropriate element of the set $\{a, a^{-1}, p, p^{-1}\}$, we must have that the state $(\alpha_{1},\alpha_{2},\ldots,\alpha_{l})$ is the unique loop state of $x$ in  the product $(w_{1}\ast w_{2} \ast \ldots \ast w_{l})_{(\alpha_{1},\alpha_{2},\ldots,\alpha_{l})}$. Therefore, since the core is strongly connected, $(w_{1}\ast w_{2} \ast \ldots \ast w_{l})_{(\alpha_{1},\alpha_{2},\ldots,\alpha_{l})}$ is equal to its core. Therefore if $s$ is the unique loop state of $x$ in $S_{W(T,U)}$, we must have that $(S_{W(T,U)})_{s}$ is $\omega$-equivalent to $(w_{1}\ast w_{2} \ast \ldots \ast w_{l})_{(\alpha_{1},\alpha_{2},\ldots,\alpha_{l})}$. Now let $D(A,B) = d_1d_2\ldots d_l$ be the word in $\{A_{a}, B_{p}, A_{a}^{-1}, B_{p}^{-1} \}^{+}$ defined as follows, for all $1 \le i \le l$, if $w_i = T^{\pm 1}$ then $d_i = A_{a}^{\pm 1}$ and if $w_i = U^{\pm 1}$ then $d_i = B_{p}^{\pm 1}$. Let $C_{D(A,B)}$ be the element of $\gen{A_{a}, B_{p}}$ representing the word $D_{(A,B)}$, then it follows that $C_{D(A,B)} = (S_{W(T,U)})_{s}\restriction_{\{0,n-1\}^{\omega}}$ since for all $i$, $(w_{i})_{\alpha_{i}}\restriction_{\{0,n-1\}^{\omega}} = d_{i}$. Therefore the map $\phi: \gen{T,U} \to \gen{A_{a}, B_{p}}$ by $S \mapsto S_{s}\restriction_{\{0,1\}^{\omega}}$, where $s$ is the unique loops state of $x$ in $S$, is a surjective homomorphism. 

In order to conclude that $\gen{T,U} \cong \gen{A_{a}, B_{p}}$ it suffices to show that the map $\varphi$ which maps $A_{a}$ to $T$ and $B_{p}$ to $U$ extends to a group homomorphism $\overline{\varphi}: \gen{T,U} \to \gen{A_{a}, B_{p}}$. To see this, it suffices to show that $[U^{-1}T, TUT^{-1}] =1$ and $[U^{-1}T, T^{2}UT^{-2}] =1$. This is easily verified by direct computation.

\end{proof}

\begin{Remark}\label{Remark:TaandUpgenerateF}
Notice that, for $T$ and $U$ as in Figures~\ref{Figure:elementinTon1ofinfiniteorder} and \ref{Figure:ThetransducerUinTOn1}, the proof above shows that the subgroup $\gen{T_{a}, U_{p}} \le \TBmr{n}{1}$, for $n \ge 3$ is isomorphic to $F$. 
\end{Remark}

As a corollary we have the following theorem generalising the result of Brin and Guzman (\cite{MBrinFGuzman}) stating that the groups $\out{T_{n,n-1}}$ and $\aut{T_{n,n-1}}$ contain an isomorphic copy of $F$.

\begin{Theorem}
For $n \ge 3$ and $1 \le r \le n-1$, the groups $\TOnr \cong \out{T_{n,r}}$ and $\TBnr \cong \aut{T_{n,r}}$ contain an isomorphic copy of Thompson's group $F$.
\end{Theorem}
\begin{proof}
That $\TOnr$, for $n\ge 3$ and $1 \le r \le n-1$, contains a group isomorphic to  $F$ is a consequence of Theorem~\ref{Theorem:TOnisubsetofTOnjifidividesjinZnminus1}. To deduce implications for $\TBnr$, let $T$ and $U$ be the transducers in Figures~\ref{Figure:elementinTon1ofinfiniteorder} and \ref{Figure:ThetransducerUinTOn1}. Observe that the map $f$ defined by, for all $\dot{d} \in \dotr$ and $\xi \in \CCn$, $\dot{d}\xi \mapsto \dot{d}(\xi)T_{a}$, and the map $g$ given by, for all $\dot{d} \in \dotr$ and $\xi \in \CCn$, $\dot{d}\xi \mapsto \dot{d}(\xi)U_{p}$, are elements of $\TBnr$ since $\core(f) = T$ and $\core(g) = U$. Moreover, they generate a subgroup isomorphic to $F$ by Remark~\ref{Remark:TaandUpgenerateF}. 
\end{proof}
\section{Further properties of \texorpdfstring{$\TOn$}{Lg} and the \texorpdfstring{$R_{\infty}$}{Lg} property for \texorpdfstring{$\Tn$}{Lg} }\label{Section:RftyTn}

In this section we deduce some properties of elements of $\TOn$ which will enable us to demonstrate that the groups $\Tnr$ for $1< r<n$ have the $R_{\infty}$ property. The case $n=2$ has been dealt with in the paper \cite{BMV}.

Throughout this section we shall identify elements of $[0,r]$, $1 \le r \le n-1$, with their $n$-ary expansions.

\begin{Remark}\label{Remark: homorphism from On into Sym(Xns)}
A consequence of Lemma~\ref{Lemma: reading loops give bijection from set of Xns to Xns} is that for $g \in \On$, and letting $\mathcal{X}_{n}^{+}$ be the set of equivalences classes of $\Xnp$ under the relation that two words are related if they are rotations of each other,  there is a bijection $\bar{g}: \RXnp \to \RXnp$ by $[w] \mapsto  [\lambda_g(w, q_{w})]$. Let $\mbox{Sym}(\RXnp)$ be the group consisting of bijections from the set $\RXnp$ to itself. The map from  $\On$ to $\mbox{Sym}(\RXnp)$ given by $g \mapsto \bar{g}$ is an injective homomorphism from $\On$ into $\mbox{Sym}(\RXnp)$ (this follows from results in \cite{BCMNO} and the forthcoming article \cite{BleakCameronOlukoya}). We may fix a choice of representatives for the equivalence classes of $\RXnp$ such that, for all $i \in  \Xn$, $[i] = i$.
\end{Remark}

The following Lemma explores properties of the map $\bar{g}$ when $g \in \TOn$.

\begin{lemma}
Let $g \in \WTOn$ then $0$ and $n-1$ are fixed by the map $\bar{g}$.
\end{lemma}
\begin{proof}
We proceed by contradiction. We first consider the case that $(0)\bar{g} \ne 0$. By Lemma~\ref{Lemma: reading loops give bijection from set of Xns to Xns} there is some prime word $w  \in \Xnp$ such that $(w)\bar{g} = 0$. Let $q_w$ be the unique state such that $\pi_{g}(w, q_w) = w$.  Then since $h_{q_{w}}$ preserves the lexicographic ordering on $\CCn$ we must have that $(00\ldots)h_{q_w}\lelex (ww\ldots)h_{q_{w}}$ however as $(ww\ldots)h_{q_{w}} = 00\ldots$ and  $h_{q_{w}}$ is injective, we have the desired contradiction.

The case $(1)\bar{g} \ne 1$ is dealt with analogously.
\end{proof}

By applying Lemma~\ref{Lemma: multiplication by piR bijection from orientation preserving to reversing and vice versa} we immediately deduce the following result:

\begin{lemma}\label{Lemma: elements of TOn either fix 0 or swap 0 and n-1}
Let $g \in \TOn$ then if $g \in \WTOn$, $(0)\bar{g} = 0$ otherwise $(0)\bar{g} = n-1$.
\end{lemma}

As a corollary we have a different proof that for $A_{q_0} \in \TBnr$, $([0,r] \cap  \Z[1/n] ) h_{q_0} = [0,r] \cap  \Z[1/n]$.

\begin{corollary}\label{Corollary:given hq0 in TBnr can assume that induced map on the Sr fixes 0 and r}
Let $A_{q_0} \in \TBnr$ then $([0,r] \cap  \Z[1/n] ) h_{q_0} = [0,r] \cap  \Z[1/n]$
\end{corollary} 
\begin{proof}
Observe that since $A_{q_0}$ is synchronizing then after reading a sufficiently long string of $0$'s or $n-1$'s the next input if processed from the state $q_{0}$ or $q_{n-1}$ respectively. However these states either swap $0$ with $n-1$ and vice versa or fix $0$ and $n-1$ (depending on whether or not $A_{q_0}$ is orientation preserving or reversing) by Lemma~\ref{Lemma: elements of TOn either fix 0 or swap 0 and n-1}. Hence since elements of $[0,r] \cap  \Z[1/n]$ are precisely those elements of $\CCnr$ of the form $w00\ldots$ or $w n-1 n-1 \ldots$ for $w \in \Xnrp$, the result follows.
\end{proof}

From this we may deduce the following:

\begin{lemma}\label{Lemma: automorphisms can be modified to contain fixed point}
Let $A_{q_0} \in \TBnr$ then there is an element $h \in \Tnr$ such that $h_{q_0} \circ h$ fixes the point $\dot{0}0\ldots \simeq \dot{r-1}n-1\ldots$.
\end{lemma}

The following constructions demonstrate that $\Tmr{n}{r}$ has the $R_{\infty}$ property for any $1 \le  r <  n$. and naturally breaks up into two cases. We first require the following standard lemma.

\begin{lemma}\label{Lemma:changingcosetrepdoesnotaffectRinfty}
Let $A_{q_0}, B_{p_0} \in \TBnr$ and $g \in \Tnr$ be such that $h_{q_0} g = h_{p_0}$ then $\Tnr$ has infinitely many $h_{q_0}$-twisted conjugacy classes if and only if it has infinitely many $h_{p_0}$-twisted conjugacy classes. 
\end{lemma}
\begin{proof}
Let $f_1$ and $f_2$ be elements of $\Tnr$ such that $f_1$ is $h_{p_0}$-twisted conjugate to $f_2$ by an element  $f_3$.  The computation below demonstrates that $f_1g^{-1}$ and $f_2g^{-1}$ are $h_{q_0}$-twisted conjugate to each other by the element $f_3$. 
\[ f_{3}^{-1} f_1 g^{-1} h_{q_0}^{-1} f_3 h_{q_0} =  f_{3}^{-1} f_1 h_{p_0}^{-1} f_3 h_{p_0}g^{-1} = f_2g^{-1} \]  

In a similar way we see that if $f_1'$ and $f_2'$ are elements of $\Tnr$ such that $f_1'$ and $f_2'$  are $h_{q_0}$-twisted conjugate to each other by an element $f_3'$, then $f_1'g$ and $f_2'g$ are $h_{p_0}$ twisted conjugate to each other  by $f_3'$. 
\end{proof}

We also observe that given $A_{q_0} \in \TBnr$,  by straight-forward manipulations, there are infinitely many $h_{q_0}$-twisted conjugacy classes if and only if there is an infinite subset $\T{J}$  of $\Tnr$ such that for any pair $f_1, f_2 \in  \T{J}$, $f_1h_{q_0}^{-1}$ is not conjugate to $f_2h_{q_0}^{-1}$ by an element of $\Tnr$.

{\bfseries{Orientation preserving case:}} First let $A_{q_0} \in \TBmr{n}{1}$ be an orientation preserving element. By Lemma~\ref{Lemma:changingcosetrepdoesnotaffectRinfty} and Lemma~\ref{Lemma: automorphisms can be modified to contain fixed point} we may assume that $A_{q_0}$  fixes the points $00\ldots$ and $n-1n-1\ldots$ of $\CCn$. Let $\ac{u}$ be a complete antichain of $\Xns$ of length $l >0$. Suppose that $\ac{u} = \{ u_0, u_1, \ldots, u_{l-1} \}$ (recall that antichains  are always assumed to be ordered lexicographically). We form a homeomorphism of $\CCn$ as follows. Let $\sigma: \ac{u} \to \ac{ u}$ be given by $u_{i} \mapsto u_{(i+1)\mod{l}}$ for $1 \le i \le l-1$. Let $h_{l}: \CCn \to \CCn$ be given by $(u_i\delta) h_{l} = (u_{i})\sigma (\delta)h_{A_{q_0}}$. Then since $A_{q_0}$ is a homeomorphism of $\CCn$, $h_{l}$ is also homeomorphism of $\CCn$. Moreover, as $h_{q_0}$ is orientation preserving, and preserves the relation $\simeqI$, by choice of the permutation $\sigma$ it follows that $h_{l}$ is an orientation preserving element of $\TBmr{n}{1}$. Moreover,  if $B_{p_0}$ is the transducer such that $h_{p_0} = h_{l}$, then $\core(B_{p_0}) = \core(A_{q_0})$. Thus  there is an element of $f \in \Tmr{n}{1}$ such that $h_{l}f = h_{q_0}$. Furthermore, observe that as $A_{q_0}$ fixes $00\ldots$, we have that the point $00\ldots$ is on an orbit of length precisely $l$ under $h_{l}$. Moreover, any point on a finite orbit under $h_{l}$ must have orbit length a multiple of $l$. Thus for every $l \in \mathbb{N}$ such that $l \cong 1 \mod{n-1}$, let $h_{l} \in \TBmr{n}{1}$ be the map constructed as above, and let $f_{l}$ be such that $h_{l} = f_{l}h_{q_0} $.  Therefore we  conclude that for any $A_{q_0} \in \widetilde{\TBmr{n}{1}}$ there are infinitely may $h_{q_0}$-twisted conjugacy classes.

Now fix $r>1$, and let $\bar{r}$ be minimal such that $\TOns{\bar{r}} = \TOns{r}$, by Lemma~\ref{Lemma:partialconverselemma} $\bar{r}$ divides $r$ and $\bar{r}$ divides $n-1$. Let $A_{q_0} \in \TBmr{n}{\bar{r}}$ be an element which fixes $\dot{0}00\ldots$ and $\dot{r-1}n-1n-1\ldots$. Let $\ac{u}$ be a complete antichain of $X_{n,r}^{*}$ of length $l$. Observe that as $l \cong r \mod{n-1}$ hence there is an $m \in \N$ such that  $m\bar{r} = l$. For $1 \le i \le m$ let $\ac{u}_{i} : = \{ u_{i,0}, \ldots, u_{i,r-1}\} \subset \ac{u}$ be such that for $1 \le i_1 < i_2 \le m$ all elements of $\ac{u}_{i_1}$ are less than (in the lexicographic ordering) all elements of $\ac{u}_{i_2}$.  Let $\sigma: \{1,2 \ldots, m\} \to \{1,2 \ldots, m\}$ be given by $ (i)\sigma = i+1 \mod{m}$. We construct an element $h_{l} \in \TBnr$ as follows: for $\Gamma \in \CCn$ and $1 \le i \le m$ and $0 \le a \le r-1$ let $(u_{i,a}\Gamma) h_{l} =  u_{(i)\sigma,b}\delta$ if and only if $(\dot{a}\Gamma)A_{q_0} =  \dot{b}\delta$. Therefore since $A_{q_0}$ induces an orientation preserving homeomorphism on  $\CCmr{\bar{r}}$, it follows that the ma $h_{l}$ induces a orientation preserving homeomorphism of $\CCnr$. Moreover since $\core(A_{q_0}) \in \TOnr$ it follows that $h_{l}$ is in fact and element of $\TOnr$ and moreover (identifying $h_{l}$ with the initial transducer inducing inducing the homeomorphism $h_{l}$) $\core(h_{l}) = \core(A_{q_0})$. Lastly observe that the point $\dot{0}00\ldots$ is on an orbit of length $m$ under the action of $h_{l}$ and any other finite orbit of $h_{l}$ has length at least $l$.

Now let $A_{q_0} \in \TBnr$ be arbitrary. Since $\core(A_{q_0}) \in \TOns{\bar{r}} = \TOnr$, there is an element $B_{p_0} \in \TBmr{n}{\bar{r}}$ such that $\core(B_{p_0}) = \core(A_{q_0})$. As before for each $l \in \N$ such that $l \equiv r \mod n-1$, there is an element $h_{l}$ with a minimal finite orbit of length $l$ and such that $\core(h_{l}) = \core(A_{q_0})$.  Thus, there is a $f_{l} \in \Tnr$ such that $h_{q_0}f_{l}  = h_{l}$. Therefore as in the case $r=1$, we conclude that for any there are infinitely many $h_{q_0}$-twisted conjugacy classes for any element $A_{q_0} \in \TBnr$.

{\bfseries{Orientation reversing case:}} We now consider the orientation reversing case. Let $A_{s_0} \in \TBmr{n}{1}$ be orientation reversing. By Lemma~\ref{Lemma: automorphisms can be modified to contain fixed point} we may assume that $(000\ldots)A_{s_0} = n-1n-1n-1\ldots$ and $(n-1n-1n-1\ldots)A_{s_0} = n-1n-1n-1\ldots$. Let $i$ be minimal such that $\pi_{A}(0^{i}, s_0) = q_0$   and $\pi_{A}((n-1)^{i}, s_0)  = q_{n-1}$ (where $q_0$ and $q_{n-1}$ are the zero and one loop state of $A_{s_0}$ respectively). Let $(n-1)^{l_1} = (0^{i+1})A_{q_0}$ and $0^{l_2} = ((n-1)^{i+1})A_{q_0}$. Observe  that $0<\min\{l_1, l_2\}$ since $\lambda_{A}(0, q_0) = n-1$ and $\lambda_{A}(n-1, q_0) = n-1$ by Lemma~\ref{Lemma: elements of TOn either fix 0 or swap 0 and n-1}. There is an element $h \in \Tn$ such that for any $\delta \in \CCn$, $((n-1)^{l_1}\delta)h = (n-1)^{i}\delta$ and $(0^{l_2}\delta h) = 0^{i}\delta$. Let $C_{s'_0} \in \TBnr$ be such that $h_{s'_0} = h_{s_0} h$. Observe that $C_{s'_0}$ satisfies, $\pi_{C}(0^{i}, s'_0) = q'_{0}$, $\pi_{C}((n-1)^{i}, s'_0) = q'_{n-1}$, where $q'_{0}$ and $q'_{(n-1)}$ are the $0$ and $n-1$ loop states of $C_{s'_0}$ respectively, and, for $j \ge i$, $\lambda_{C}(0^{j}, s'_0) = (n-1)^{j}$ and $\lambda_{C}((n-1)^{j}, s'_0) = 0^{j}$. As before we may replace $A_{s_0}$ with $C_{s'_0}$ without loss of generality, since there infinitely many $A_{s_0}$-twisted conjugacy classes if and only if there are infinitely many $C_{s'_0}$-twisted conjugacy classes.  Therefore we may assume that there is an $i \in \N$, $i>0$ such that $A_{s_0}$ satisfies $\pi_{A}(0^{i}, s_0) = q_{0}$, $\pi_{A}((n-1)^{i}, s_0) = q_{n-1}$, where $q_{0}$ and $q_{(n-1)}$ are the $0$ and $n-1$ loop states of $A_{s_0}$ respectively, and, for $j \ge i$, $\lambda_{A}(0^{j}, s_0) = (n-1)^{j}$ and $\lambda_{A}((n-1)^{j}, s_0) = 0^{j}$.

We now construct an element $g \in \Tn$ such that $g h_{s_0}$ such $(g h_{s_0})^2$ has the point $000\ldots$ and $n-1n-1\ldots$ as attracting fixed points.

By a branch of length $k$, for $k \in \N$, $k>0$, we shall mean the finite rooted $n$-ary tree with $1+ k(n-1)$ leaves and $k$ the length of the geodesic from the root to the left-most leaf. 

Let $\bar{f} \in \Tn$ be defined as follows. Let $E$ and $F$ be the finite trees where the tree $E$ is obtained from the single caret,  by attaching  a branch of length $i+1$ to the first and penultimate leaves of the branch and the tree $F$ is obtained from the single caret by attaching a branch to the first and last leaves of the caret. Label the leaves of $E$ left to right by $1,2, \ldots, 1 + (2k+1)(n-1)$ and the leaves of $F$ from right to left by $1,2, \ldots, 1 + (2k+1)(n-1)$. The map $\bar{f}$ is now defined by, for $\Gamma \in \CCn$ and $\mu_{j}$ and $\nu_{j}$ the addresses of the $j$'th  leaves of $E$ and $F$ respectively,  $\mu_{j}\Gamma \mapsto \nu_{j}\Gamma$.

Let $f \in \TBmr{n}{1}$, be defined by, for $\Gamma \in \CCn$,  $\mu_{j}\Gamma \mapsto \nu_{j}(\Gamma)A_{s_0}$ where  $\mu_{j}$ and $\nu_{j}$ are the addresses of the $j$'th  leaves of $E$ and $F$ respectively. Essentially $f$ is obtained from $\bar{f}$ by attaching the state $s_0$ to all the leaves of $F$ the range tree of $\bar{f}$. Moreover since $\core(f) = \core(A_{s_0})$ there is an element $g \in \Tn$ such that $f = gA_{s_0} $.

Below we shall consider the orbit of cones $U_{\nu}$ for $\nu \in  X_{n}^{+}$ under $f$, and to simplify notation, we identify the cone $U_{\nu}$ with the finite word $\nu$. Let us consider the orbit of the cone $(n-1)$ (a leaf of $E$) under $f$: \[(n-1)\mapsto 00^{i} \mapsto n-1 (n-1)^i \mapsto 00^i 0^i \mapsto n-1 (n-1)^i (n-1)^i \mapsto 00^i(0)^{2i} \ldots \]
This is because for $(\alpha, \bar{\alpha}) \in \{ (0, n-1), (n-1, 0) \}$ and $j \in \N$ $j\ge i$, $\lambda_{A}(\alpha^{j}, s_0) = \bar{\alpha}^{j}$ and $\pi_{A}(\alpha^{j}, s_0) = q_{\alpha}$. Hence we see that the point $n-1n-1\ldots$ is an attracting fixed point of $f^2$ and in a similar way the point $00\ldots$ is an attracting fixed point of $f$. In particular we have, for $j \in \N_{1}$,  $(n-1)f^{2j} = (n-1)(n-1)^{ji}$ and $(00^{i})f^{2j} = 00^{i}0^{ji}$. We call the words $(n-1)^{i}$ and $0^{i}$ the attracting paths (in $f^2$) of $00\ldots$ and $n-1n-1\ldots
$ respectively. 

Let $h \in \Tn$ be arbitrary and consider $h^{-1} f h$. Now as $h$ is a prefix exchange map, by an abuse of notation for a finite word $\Gamma \in \Xn^{\ast}$ such that $h$ acts as a prefix exchange on the cone $U_{\Gamma} \subset \CCn$, we shall write, $(\Gamma)h$ for the word $\Delta \in \Xn^{*}$ such that  $(U_{\Gamma}) h = U_{\Delta}$. There is an $l \in \N$ such that $((n-1)^{l}) h$ is a finite word in $X_n^{+}$ and we may assume that $l \ge i+1$. Thus we have,
\[
 ((n-1)^{l})h h^{-1}f^2h = ((n-1)^{l})f^2h = ((n-1)^{l})h(n-1)^{i}
\]

Hence we deduce that $((n-1)^l)h h^{-1}f^{2j}h = ((n-1)^{l})h(n-1)^{ji}$ and we see that conjugation of $f$ by an element of $\Tn$ maps the attracting fixed point $n-1n-1\ldots$ to an attracting fixed point and preserves the path $(n-1)^{i}$. We also observe that since an orientation reversing homeomorphism of the circle has precisely two fixed points (on the circle), then $h$ must map the pair of fixed points of $f$ to the pair of fixed of  $h^{-1}fh$. Notice that thinking of the fixed points as elements of $\CCn$, since $\Tn$ elements  $\Z[1/n] \to \Z[1/n]$, this means that either $h$ fixes the points $00\ldots$ and $n-1 n-1 \ldots$ otherwise, the other fixed point of $f$ is an element of $\Z[1/n]$ of the form $\eta_1 \simeqI \eta_2$ (where $\eta_1 > \eta_2$ in the lexicographic ordering on $\CCn$), $(00\ldots)h =\eta_1$ and $(n-1n-1\ldots)h = \eta_2$. We make use of these facts to show that there are infinitely many $h_{s_0}$-twisted conjugacy classes.

Let $j$ be any integer greater than $i$, as before we may construct a map $f_{j} \in \TBmr{n}{1}$ such that $f_{j} = g_{j}h_{s_0}$ for some $g_{j} \in \Tn$, and $00\ldots$ and $n-1n-1\ldots$ are attracting fixed points of $f_{j}^2$ with attracting paths $0^{j}$ and $(n-1)^{j}$. We  have the following claim:

\begin{claim}
There is an infinite subset $\T{J} \subset \N_{i}$ for which the set $\{f_{j} \mid j \in \T{J}\}$ are pairwise not conjugate by an element of $\Tn$.
\end{claim}
\begin{proof}
Suppose there is an $N \in \N$ such that every element $f_{j}$, $j \in \N_{i}$ is conjugate to some element of $\{f_{i+1}, \ldots
 f_{i+N}\}$  by an element of $\Tn$. Since $\N_{i}$ is infinite, there is an infinite subset $\T{I} \subset \N_{i}$ and $1 \le M \le N$ such that every element of $\{f_{l} \mid l \in \T{I}\}$ is conjugate to $f_{i+M}$ by an element of $\Tn$. Let $l > i+M$. Since conjugation by an element of $\Tn$ preserves attracting paths and $l>M$ it must be the case that any conjugator $h \in \Tn$ such that $h^{-1}f_{l}h = f_{i+M}$ must satisfy $(n-1n-1\ldots)h \ne n-1 n-1 \ldots$. Hence the other fixed point of $f_{i+m}$ must be a dyadic rational of the form $\eta_1 \simeqI \eta_2 \in \CCn/\simeqI$ where $\eta_1$ is greater than $\eta_2$ in the lexicographic ordering of $\CCn$. This means by arguments above that  $(n-1n-1\ldots)h = \eta_2$ and $\eta_2$ is an attracting fixed point of $f_{i+M}$ with attracting path $(n-1)^i$. However for $l' \in \T{I}$ such that $l' >l$, it must be the case that $f_{l'}$ is not conjugate to $f_{i+M}$ since the attracting path in $f_{l'}^2$ of $n-1n-1\ldots$  is longer than the attracting path in $f_{i+M}^2$ of both attracting fixed points of $f_{i+M}$.
\end{proof}

Now fix $1< r < n$. Let $A_{s_0} \in \TBnr$ be an orientation reversing element,  $\bar{r}$ be minimal such that $\TOns{\bar{r}} = \TOns{r}$ and $B_{p_0}$ be an orientation reversing element of $\TBmr{n}{\bar{r}}$ with $\core(B_{p_0}) = \core(A_{s_0})$. As in the case $r=1$ we may assume that there is an $i \in \N$, $i>0$ such that $B_{p_0}$ satisfies $\pi_{B}(\dot{0}0^{i}, p_0) = q_{0}$, $\pi_{B}(\dot{r-1}(n-1)^{i}, p_0) = q_{n-1}$, where $q_{0}$ and $q_{(n-1)}$ are the $0$ and $n-1$ loop states of $\core(A_{s_0}) = \core(B_{p_0})$ respectively, and, for $j \ge i$, $\lambda_{B}(\dot{0}0^{j}, p_0) = \dot{r-1}(n-1)^{j}$ and $\lambda_{B}(\dot{r-1}(n-1)^{j}, p_0) = \dot{0}0^{j}$. Now as in the orientation preserving case we construct elements $f_{j}$, for $j \in \N_{i}$ with desirable properties by simulating the element $B_{p_0}$ appropriately on cones.

Let $\ac{u} = \{ \dot{a} \mid 1 \le a \le r \}$, that is, $\ac{u}$ corresponds to the roots of the disjoint union of $r$ copies of the $n$-ary tree. Since $\bar{r}$ divides $r$ there is an $M \in \N$ such that $M\bar{r} = r$. Let $\ac{u}_{k}:= \{ u_{k,1}, \ldots, u_{k,r}\}$, $1 \le k \le M$ be subsets of $\ac{u}$, such that for $1 \le k_1 < k_2 < M$, all elements of $\ac{u}_{k_1}$ are less than (in the lexicographic ordering) all elements of $\ac{u}_{k_2}$. Observe that $\ac{u}_{1}$ corresponds to the roots of the disjoint union of the $\bar{r}$ copies of the $n$-ary tree. Replace $\ac{u}_1$, still retaining the symbol $\ac{u}_{1}$ for the resulting antichain, with the antichain corresponding to attaching a branch of length $i+1$ to the root $u_{1,1}$. Likewise replace $\ac{u}_M$ with the antichain corresponding to attaching a branch of length $i+1$ to the root $u_{M, r-1}$. In a similar way let $\ac{v}_M$ be the antichain obtained from $\ac{u}_{M}$ corresponding to attaching a branch of length $i+1$ to the root  $\ac{u}_1$, and let $\ac{v}_{k}:= \ac{u}_{k}$ for $1 \le k < M$. Observe that $\ac{u}_{1}$, $\ac{u}_{M}$ and $\ac{v}_{M}$ all have equal length $d$ congruent to $\bar{r}$ modulo $n-1$. Since $\bar{r}$ divides $n-1$ let $m \in \N$ be such that $m\bar{r} = d$. Let  $\ac{u}_{1} = \cup_{1 \le a \le m} \ac{u}_{1,a}$ where $\ac{u}_{1,a} =  \{u_{1,a,b} \mid 1 \le b \le r \}$, likewise let $\ac{u}_{M} = \cup_{1 \le a \le m} \ac{u}_{M,a}$ where $\ac{u}_{M,a} =  \{u_{M,a,b} \mid 1 \le b \le r \}$ and $\ac{v}_{M} = \cup_{1 \le a \le m} \ac{v}_{M,a}$ where $\ac{v}_{M,a} =  \{v_{M,a,b} \mid 1 \le b \le r \}$. 

We construct a homeomorphism $f_{i}$ of $\CCnr$ as follows. Let $\sigma: \{1,2,\ldots,M\} \to \{1,2, \ldots M\}$ by $k \mapsto M-k +1$ and let $\rho: \{1,2,\ldots, m\} \to \{1,2,\ldots m\}$ by $a \mapsto m-a+1$. For $\Gamma \in \CCn$, $1< k < M$ and $1 \le b \le r$, $(u_{k,b}\Gamma)f_{i} = v_{(k)\sigma, b'}\delta$ if and only if $\lambda_{B}(\dot{b}\Gamma, p_0) \in U_{\dot{b'}}$ and $\delta = \lambda_{B}(\Gamma, \pi_{B}(\dot{b}, p_0))$. For $\Gamma \in \CCn$ $1 \le a \le m$ and $1 \le b \le r$, $(u_{1,a,b}\Gamma)f_{i} = v_{M,((a)\rho),b'}\delta$ if and only if $\lambda_{B}(\dot{b}\Gamma,p_0) \in U_{\dot{b'}}$ $\delta = \lambda_{B}(\Gamma, \pi_{B}(\dot{b}, p_0))$. For $\Gamma \in \CCn$ $1 \le a \le m$ and $1 \le b \le r$, $(u_{M,a,b}\Gamma)f_{i} = v_{1,((a)\rho),b'}\delta$ if and only if $\lambda_{B}(\dot{b}\Gamma,p_0) \in U_{\dot{b'}}$ $\delta = \lambda_{B}(\Gamma, \pi_{B}(\dot{b}, p_0))$. Since $B_{p_0}$  induces an orientation reversing homeomorphism on $\CCmr{\dotr}$, we see that $f_{i}$ is in fact an element of $\TBnr{r}$. Moreover, $\core(f_i) = \core(B_{p_0}) = \core(A_{s_0})$ and so there is an element $g_i \in  \Tnr$ such that $f_i = h_{s_0}g_i$.

We now argue that the points $\dot{0}0\ldots$ and $\dot{r-1}n-1n-1\ldots$ are attracting fixed points of $f_{i}$ with attracting paths $0^{i}$ and $(n-1)^{i}$. We consider the orbit of the cone  $\dot{r-1}(n-1)^{i}$:
\[
\dot{r-1}(n-1)^{i} \mapsto \dot{0}0^{i}0^{i} \mapsto \dot{r-1}(n-1)^{i}(n-1)^{i} \mapsto \dot{0}0^{i}0^{i}0^{i} \mapsto \dot{r-1}(n-1)^{i}(n-1)^{2i} \ldots
\] 

this follows by making use of the definition of $f_{i}$ and the facts: $\pi_{B}(\dot{0}0^{i}, p_0) = q_{0}$, $\pi_{B}(\dot{r-1}(n-1)^{i}, p_0) = q_{n-1}$, where $q_{0}$ and $q_{(n-1)}$ are the $0$ and $n-1$ loop states of $\core(A_{s_0}) = \core(B_{p_0})$ respectively, and, for $j \ge i$, $\lambda_{B}(\dot{0}0^{j}, p_0) = \dot{r-1}(n-1)^{j}$ and $\lambda_{B}(\dot{r-1}(n-1)^{j}, p_0) = \dot{0}0^{j}$. In general we see that, for $l \in \N_{1}$, $(\dot{r-1}(n-1)^{i})f_{i}^{2l} = \dot{r-1}(n-1)^{i}(n-1)^{l}$ and $(\dot{0}0^{i}0^{i})f_{i}^{2l} = \dot{0}0^{i}0^{i}0^{l}$.

For each  $j > i$ we may repeat the construction above to get  elements $h_{j}$ such that $f_{j} = h_{s_0}g_{j}$ for some $g_{j} \in \Tnr$ and the points $\dot{0}0\ldots$ and $\dot{r-1}n-1n-1\ldots$ are attracting fixed points of $f_{j}$ with attracting paths $0^{j}$ and $(n-1)^{j}$. Since $f_{j}$ for $j \in \N_{i}$ induces a homeomorphism on circle of length $r$, $[0,r]$ with end points identified, we may repeat the arguments as in the case $r =1$ to conclude that there is an infinite subset $\T{J} \subset \N_{i}$ for which the set $\{f_{j} \mid j \in \T{J}\}$ are pairwise not conjugate by an element of $\Tnr$.

Thus we have demonstrated the following:

\begin{Theorem}\label{Theorem:TnrhastheRftyproperty}
The group $\Tnr$ for $1 \le r < n-1$ has the $R_{\infty}$ property.
\end{Theorem}

\printbibliography 

\end{document}